\documentclass[11pt]{amsart}

\usepackage[margin=3cm]{geometry}

\usepackage[foot]{amsaddr}
\usepackage{amsmath}
\usepackage{amssymb}
\usepackage{amsthm}
\usepackage{graphicx}
\usepackage{float}
\usepackage{color}
\usepackage{dsfont}
\usepackage{comment}
\usepackage{epstopdf}
\usepackage{hyperref}
\usepackage{appendix}

\allowdisplaybreaks[4]
\numberwithin{equation}{section}

\newtheorem{theorem}{Theorem}[section]
\newtheorem{corollary}{Corollary}[section]
\newtheorem{assumption}{Hypothesis}[section]

\newtheorem{proposition}{Proposition}[section]

\theoremstyle{definition}
\newtheorem{definition}[assumption]{Definition}
\newtheorem{remark}[assumption]{Remark}

\newcommand{\ep}{\varepsilon}

\title[ ] 
      { 
      Principal eigenvalue for some elliptic operators with large drift: Neumann boundary conditions}

\author
{Shuang Liu, \,\, Yuan Lou,\,\, Maolin Zhou}

\thanks{{S. Liu}: School of Mathematics and Statistics, Beijing Institute of Technology, Beijing, 100081, China.
}

\thanks{{Y. Lou}: School of Mathematical Sciences,  CMA-Shanghai, Shanghai Jiao Tong
University, Shanghai, 200240, China.
}

\thanks{{M. Zhou}: Chern Institute of Mathematics, Nankai University, Tianjin, 300071,  China. }

 \email{liushuangnqkg@bit.edu.cn; yuanlou@sjtu.edu.cn; zhouml123@nankai.edu.cn}



\begin{document}
\maketitle
\newcommand\tbbint{{-\mkern -16mu\int}}
\newcommand\tbint{{\mathchar '26\mkern -14mu\int}}
\newcommand\dbbint{{-\mkern -19mu\int}}
\newcommand\dbint{{\mathchar '26\mkern -18mu\int}}
\newcommand\bint{
{\mathchoice{\dbint}{\tbint}{\tbint}{\tbint}}
}
\newcommand\bbint{
{\mathchoice{\dbbint}{\tbbint}{\tbbint}{\tbbint}}
}

\begin{abstract}
The paper  is concerned with  the principal eigenvalue of some linear elliptic operators with drift  in two dimensional space.  We provide a refined description of the asymptotic behavior for the principal eigenvalue  as the  drift rate approaches infinity. Under some non-degeneracy assumptions, our results illustrate that these asymptotic behaviors are completely determined by some connected components in the omega-limit set of the  system of ordinary differential equations associated with the drift term, which includes stable fixed points,  stable limit cycles,  hyperbolic saddles connecting homoclinic orbits,  and  families of closed orbits. Some discussions on degenerate cases are also included.
\end{abstract}

\bigskip
\section{\bf Introduction}\label{S1}
We consider the linear eigenvalue problem 
\begin{equation}\label{1}
 \left\{\begin{array}{ll}
\medskip
-\Delta\varphi-A\mathbf{b}\cdot\nabla\varphi+c(x)\varphi=\lambda\varphi &\mathrm{in} \,\, \Omega,\\
  \nabla\varphi\cdot \mathbf{n}=0 & \mathrm{on}~\partial\Omega,  
  \end{array}\right.
 \end{equation}
 where $\Omega$ is 
 a bounded domain of $\mathbb{R}^{2}$ with smooth boundary $\partial\Omega$, and $\mathbf{n}(x)$ is the outward unit normal vector at $x\in\partial\Omega$. Here $\mathbf{b}(x)$ denotes a 
 $C^1$ vector field in $\mathbb{R}^2$,  the positive parameter $A$ represents the 
 drift
 rate, and $c\in C(\bar\Omega)$.
It is well-known  that problem \eqref{1}
admits a principal eigenvalue,
denoted by $\lambda (A)$, which is real and simple, and the corresponding eigenfunction can be chosen to be positive. Furthermore, $\lambda (A)<{\rm Re}(\lambda)$ holds for any
other eigenvalue $\lambda$ of \eqref{1}. This paper is 
devoted to the study of
the asymptotic behavior of $\lambda (A)$ as $A\to \infty$ for general vector field $\mathbf{b}(x)$,
under zero Neumann boundary conditions.

The question concerning the influence of 
drift on 
the principal eigenvalue 
 of problem \eqref{1} 
arises naturally in 
many biological and physical problems.
  In the study of spatial population dynamics in advective environments, 
 reaction-diffusion-advection models have been proposed to understand the  persistence, competetion, and evolution of  single or multiple  species. The analysis
of these models, particularly the stability of equilibrium solutions,  requires a deep understanding of the asymptotic behaviour of  principal eigenvalue of problem \eqref{1}
\cite{ALL2017,LLP2022,LZZ2019}. Another related area of active research concerns the effect of drift on  the speed of  propagation of travelling fronts in heterogeneous media \cite{FP1994, HZ2013,N2010,SK2014}. Therein  the analysis is centered around a variational formula 
\cite{BH2005,BHN2005-1}, which
is in turn characterized by the principal eigenvalue of \eqref{1}. The asymptotic   behaviour of  principal eigenvalue thereby   plays a fundamental role in this issue \cite{BHN2005}. 
Problem \eqref{1} is also connected with  the enhancement of diffusive mixing by  a fast advecting flow \cite{CKR2001,CKR2008} and the rearrangement inequalities of principal eigenvalue for  non-symmetric elliptic operators \cite{HNR2011}.

 Hodge decomposition theorem implies that  
 a $C^1$ vector field {\bf b} can be decomposed in the form ${\bf b}={\bf v}+\nabla m$ with  a divergence-free field ${\bf v}$ satisfying $\nabla\cdot{\bf v}=0$ in $\Omega$ and ${\bf v}\cdot{\bf n}=0$ on $\partial\Omega$, as well as a gradient  field $\nabla m$.
When  ${\bf b}={\bf v}$ is a divergence-free vector field, 
the asymptotic behavior for the principal eigenvalue of problem \eqref{1} is 
established by Berestycki et al.
\cite{BHN2005}. Among other things, they 
showed that
 \begin{equation}\label{liu0913-2}
     \lim_{A\to\infty}\lambda(A)=\inf_{\phi\in \mathbf{I}}\left[\frac{\int_{\Omega} (|\nabla\phi|^2+c(x)\phi^2)\mathrm{d}x}{\int_{\Omega }\phi^2\mathrm{d}x}\right],
 \end{equation}
 whereas $\mathbf{I}:=\{\phi\in H^1(\Omega):  \mathbf{v}\cdot\nabla\phi =0 \text{ a.e. in }\Omega\}$ is
 the set of the first integrals of vector field ${\bf v}$.
 As the drift rate $A$ approaches infinity,
 \eqref{liu0913-2} 
 suggests that the normalized principal eigenfunction 
converges to some  first integral in $\mathbf{I}$,
which 
reflects the mixing effect
of divergence-free vector fields.

On the other hand,  when ${\bf b}=\nabla m$ is a gradient field, by assuming that all critical points of $m$ are non-degenerate (i.e. the Hessian matrix
of $m$ at every critical point is non-singular),  
it is proved in \cite{CL2008}
that 
 \begin{equation}\label{liu0913-1}
     \lim_{A\to\infty}\lambda(A)=\min_{x\in\mathcal{M}}c(x).
 \end{equation}
Here $\mathcal{M}$ denotes the set of points of local maximum of $m$. 
Instead of the 
mixing effect for the divergence-free vector field,
\eqref{liu0913-1} 
indicates that a properly normalized principal eigenfunction of \eqref{1} 
is concentrated on some local maximum of $m$.
Problem \eqref{1} involving  degenerate  potential  $m$
in dimension one has been investigated in \cite{PZ2018}, where 
the limiting behavior of the principal  eigenvalue is determined by some ``degenerate intervals" on which $m$ remains constant. 
We refer to \cite{CL2012, LL2020, LLPZ2021-1,LLPZ2021-2, N2010, PZZ2019} for more
 discussions.

An interesting question is to understand the  connection between
two rather different asymptotic behaviors
of the principal eigenvalue, namely, \eqref{liu0913-2} and \eqref{liu0913-1}, 
which are associated with divergence-free
vector fields and  gradient fields, respectively.
This issue turns out to be  rather complicated,
as illustrated by the following  example.

{\it Example}.
  Consider problem \eqref{1} with $\Omega=\{x\in\mathbb{R}^2:|x|^2<1\}$ and
$$\mathbf{b}(x)=(1-\alpha)(-x_2,x_1)+\alpha(-1,0)$$
for $\alpha\in[0,1]$, which is
a combination of  a divergence-free
vector field and a gradient field. In this case, the limiting behavior of the principal eigenvalue $\lambda(A)$ can be  characterized as follows, for which the proof is postponed to the appendix.

\begin{proposition}\label{theorem0822}
    The following assertions hold.
\begin{itemize}
    \item[(i)] For   $\alpha\in[0,\frac{1}{2})$, set $B_\alpha:=\{x\in \Omega: |x-(0,-\frac{\alpha}{1-\alpha})|
    <\frac{1-2\alpha}{1-\alpha}\}$, then
$$\lim_{A\to\infty}\lambda(A)=\inf_{\phi\in \mathbf{I}_\alpha}\frac{\int_{B_\alpha} (|\nabla\phi|^2+c(x)\phi^2){\rm d}x}{\int_{B_\alpha} \phi^2{\rm d}x},$$
whereas  $\mathbf{I}_\alpha=\{\phi\in H^1(B_\alpha): \mathbf{b}\cdot\nabla\phi=0 \text{ a.e. in } B_\alpha\}$.
\smallskip
\item[(ii)] For   $\alpha\in[\frac{1}{2},1]$, set $x_\alpha:=(-\frac{\sqrt{2\alpha-1}}{\alpha},\,-\frac{1-\alpha}{\alpha})$, then
$$\lim_{A\to\infty}\lambda(A)=c(x_\alpha).$$
\end{itemize}
\end{proposition}

\begin{figure}[http!!]
  \centering
\includegraphics[height=1.3in]{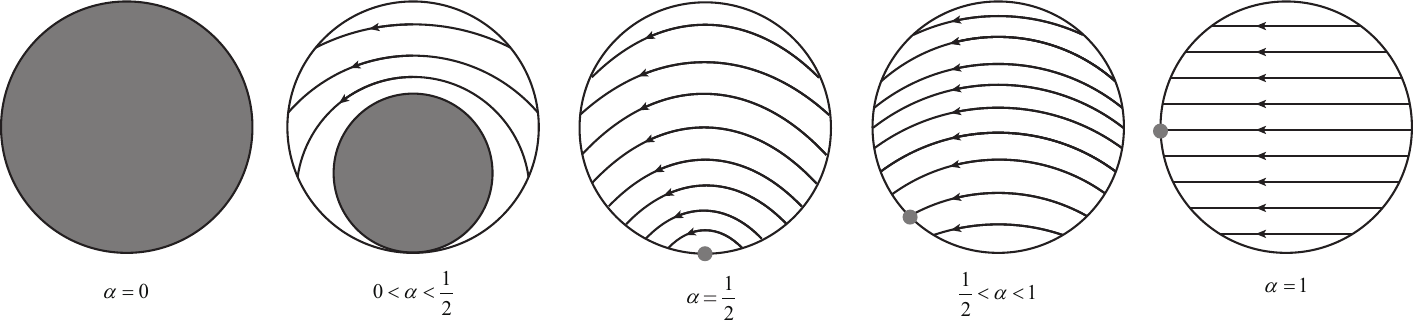}
  \caption{\small Illustrations for the phase-portraits of the vector field  $\mathbf{b}(x)=(1-\alpha)(-x_2,x_1)+\alpha(-1,0)$ for $\alpha\in[0,1]$, whereas the shaded components  represent the region $B_\alpha$ and the point $x_\alpha$ defined in Proposition \ref{theorem0822}.
  }\label{figureliu0822}
  \end{figure}

When $\alpha=0$ and $\alpha=1$, Proposition \ref{theorem0822} is  a direct consequence of \eqref{liu0913-2} and \eqref{liu0913-1},
respectively. 
When $\alpha\in(0,\frac{1}{2})$,
 the limit of the principal  eigenvalue as $A\to\infty$ is  determined by the minimization of the
Rayleigh quotient over all first integrals on the 
region $B_\alpha$, which is the whole domain when $\alpha=0$ as given in \eqref{liu0913-2}. 
See the shaded  regions in Fig. \ref{figureliu0822}. As $\alpha\nearrow \frac{1}{2}$, the 
region $B_\alpha$ shrinks 
to the 
point $(0,-1)$, whence the
asymptotic behavior of the principal eigenvalue  is  determined by the point $x_\alpha$ whenever $\alpha\in[\frac{1}{2},1)$. See the shaded  points in Fig. \ref{figureliu0822}.  In particular, as $\alpha\nearrow 1$, the critical point $x_\alpha$  moves
to $(-1,0)$ along the boundary $\partial\Omega$, which  is consistent with 
\eqref{liu0913-1}.
The transitions from $\alpha=0$
to  $\alpha=1$ in Fig. 1
suggest the  
complexity of problem \eqref{1} once
that the 
vector field is not necessarily a divergence-free field or a gradient field.
The methods developed
 in \cite{BHN2005,CL2008} are inapplicable for this example.

 The 
 goal of this paper is to investigate problem \eqref{1} involving  general 
 vector field ${\bf b}$ and provide an  unified characterization for the asymptotic behavior of the principal eigenvalue.
Our analysis is based on the dynamics of  the system
of ordinary differential equations
\begin{equation}\label{system}
   \frac{{\rm d}x(t)}{{\rm d} t} =\mathbf{b}(x(t)),\qquad t>0. 
\end{equation}
Under some non-degenerate assumptions, we will prove that these asymptotic
behaviors are completely determined by some connected components in the omega-limit set of \eqref{system},
 which include stable fixed points,
stable limit cycles, hyperbolic saddles connecting homoclinic orbits, and  family of closed orbits. Some degenerate situations, which generalizes the results in \cite{PZ2018} to two-dimensional space, are also discussed.
The proofs are based upon delicate constructions of super/sub-solutions and the
applications of comparison principles. 

This paper is organized as follows: 
We formulate  some assumptions and
state the main results in Section \ref{Main result}. In Section \ref{S2}  we establish the asymptotic behavior of the principal eigenvalue 
when  the omega-limit set of  system \eqref{system} consists of a finite number of hyperbolic fixed points only. 
Section \ref{S3} is devoted to the cases that the omega-limit set
of \eqref{system}
contains limits cycles, which can be stable, unstable, or semi-stable, by 
formulating a new coordinate near the limit cycles. The new coordinate is used in Section \ref{S5} to analyze the limiting behavior of the principal eigenvalue when the omega-limit set of \eqref{system} contains  homoclinic orbits, 
by combining with the delicate analysis near the hyperbolic saddles. In Section \ref{S4-1},
we investigate the case when the omega-limit set of \eqref{system} contains a family of closed orbits and establish the results analogues to  \cite{BHN2005} but without 
the divergence-free assumption. Some  degenerate cases are discussed in Section \ref{S5-1},
and the proof of Proposition \ref{theorem0822} 
is 
given in the appendix.

\section{\bf Preliminaries and main results}\label{Main result}
We first recall some definitions  associated with  system \eqref{system}. For each $x\in\Omega$,  by $\Phi^t(x)$ we denote the unique solution of system \eqref{system} with
 initial value $x$ at $t=0$, which defines a {\it solution curve}  of \eqref{system}. 
Define $O(x):=\{\Phi^t(x): t\in I_x\}$ as the {\it orbit} of \eqref{system} passing through the point $x$ with
$I_x$ being the maximal interval of definition of $\Phi^t(x)$. 

\begin{definition}\label{definition1}
Assume that ${\bf b}$ is a $C^1$ vector field in $\mathbb{R}^2$.
\begin{itemize}
    \item[{(1)}] A {\it fixed point} $x$ is a point in $\Omega$ satisfying ${\bf {b}}(x)=0$. Let ${\rm B}$ be  the Jacobian matrix of ${\bf b}$ evaluated at $x$. The fixed point $x$ is called {\it hyperbolic} if the real parts of all eigenvalues of ${\rm B}$ are non-zero.  A hyperbolic fixed point is {\it stable} (resp.  {\it unstable}) provided the real parts of all eigenvalues of ${\rm B}$ are negative (resp.  positive); otherwise, the hyperbolic fixed point is {\it a saddle};
    \smallskip
    \item[{(2)}] A {\it periodic orbit} is the orbit $O(x)$  of some point $x$ which satisfies $\Phi^t(x)=x$ for some $t\in I_x$. A {\it limit cycle}  is an isolated periodic orbit;
    \smallskip
   \item[{(3)}] For a fixed point $x_*\in\Omega$, the orbit $ O(x)$ ($x\neq x_*$) is called a {\it homoclinic orbit} with respect to $x_*$ provided that $\lim_{t\to+\infty}\Phi^t(x)=\lim_{t\to-\infty}\Phi^t(x)=x_*$.
\end{itemize}
\end{definition}

Throughout the paper, we assume
\begin{equation}\label{assumption1}
    {\bf b}(x)\cdot \mathbf{n}(x)<0,\quad x\in\partial\Omega,
\end{equation}
where $\mathbf{n}(x)$ is the outward unit normal vector on $\partial \Omega$.
We denote by $\omega(x)$ and $\alpha(x)$ the omega-limit set and  alpha-limit set of  point $x\in\Omega$ in the usual way.
Denote by $\{\omega(x)\cup \alpha(x):x\in\Omega\}$ the limit set  of system \eqref{system} in $\Omega$, which is nonempty, compact, and connected. 
Assumption \eqref{assumption1} implies that the limit set is contained in the interior of $\Omega$.

Next,
we impose the following assumptions on vector field ${\bf b}$. 
\begin{assumption}\label{assum1}
 Suppose that  the limit set 
of system \eqref{system} in $\Omega$ contains 
a finite number of connected components, 
and each connected component 
is  one of the following cases:
\begin{itemize}
  \item[{(i)}] A hyperbolic fixed point;
    \smallskip
    \item[{(ii)}] A limit cycle;
    \smallskip
    \item[{(iii)}] A union of two  homoclinic orbits  connected by a hyperbolic saddle (
    $\infty$-shaped orbit);
    \smallskip
    \item[{(iv)}] An isolated  homoclinic orbit with respect to a hyperbolic saddle;
    \smallskip
     \item[{(v)}] A family of  closed orbits composed of the union of periodic orbits, and/or two  homoclinic orbits  connected by a hyperbolic saddle, and center points, for which the boundary consists of periodic orbits.
\end{itemize}
\end{assumption}

\begin{remark}\label{remark_assumption}
The well-known Poincar\'e-Bendixson theorem  indicates that  the structure of the limit
set of system \eqref{system} can be described by fixed points, limit
cycles, or finite number of saddles together with homoclinic and heteroclinic orbits connecting them. The main  restriction in Hypotheses \ref{assum1} is two-fold.
(1) 
All fixed points of \eqref{system} except for center points are hyperbolic. 
This assumption implies $|\Omega_*|=0$, where  $\Omega_*:=\{x\in\Omega:{\bf b}(x)=0\}$ denotes the collection of fixed points. The degenerate case $|\Omega_*|>0$ turns out to be complicated, which is discussed in Section \ref{S5-1}. (2) 
We require that the limit set of  \eqref{system} allows for two homoclinic
saddle connections only, and exclude  the 
homoclinic and heteroclinic orbits connected by multiple saddles. This requirement is mainly 
for the sake of clarity in the presentation of the paper, and we leave the general cases for future studies.
See Fig. \ref{figureliuintroduction3}  for some examples of connected components in  Hypotheses \ref{assum1}.
\end{remark}

\begin{figure}[http!!]
  \centering
\includegraphics[height=4.8in]{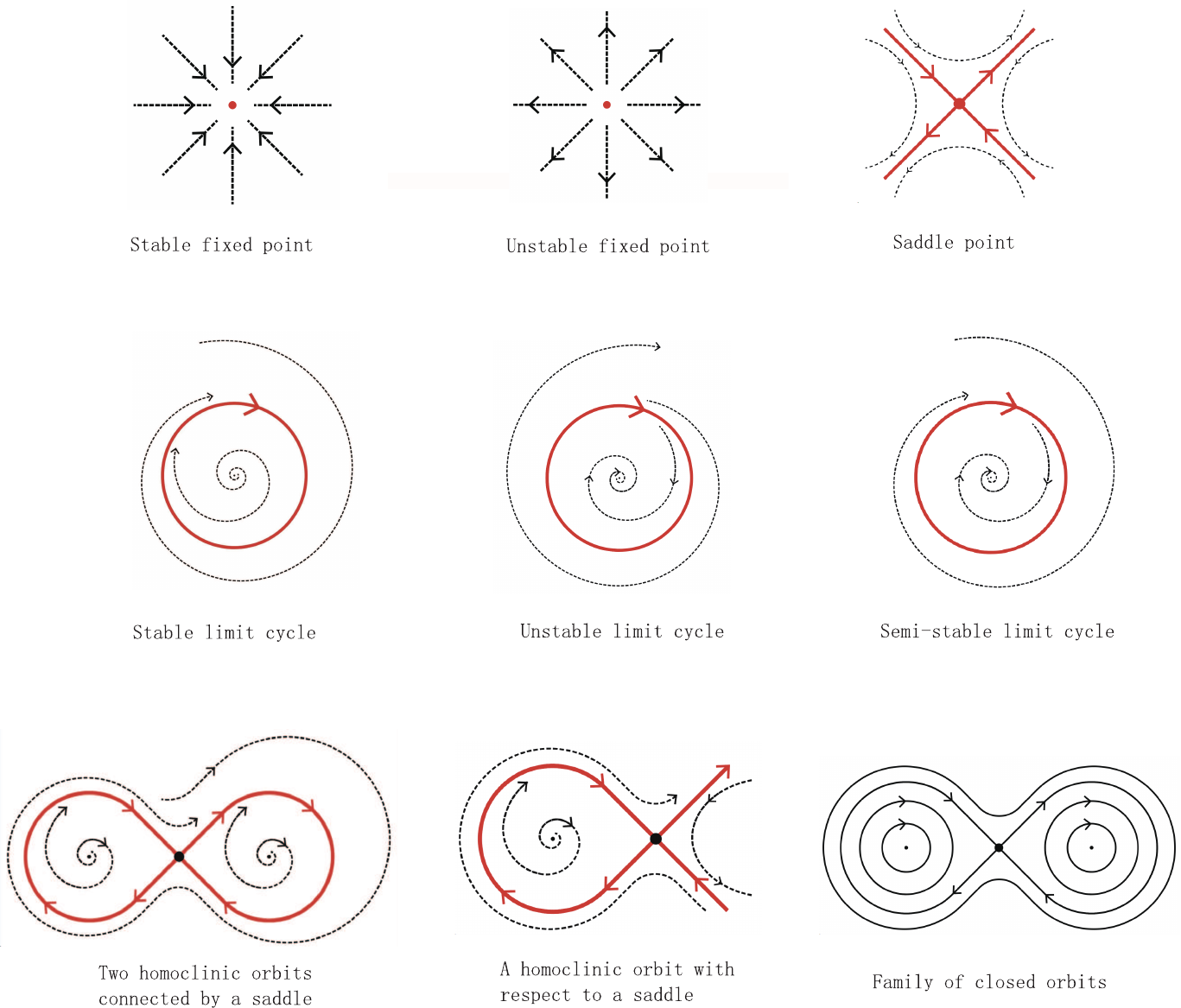}
  \caption{\small Illustrations of some typical examples of the connected components in the limit set of \eqref{system} as 
  in Hypothesis \ref{assum1}, for which the stability is given in Definition \ref{definition2}.
  }\label{figureliuintroduction3}
  \end{figure}

In the sequel, we present some definitions associated to the stability of  connected components 
given in Hypothesis \ref{assum1}.

\begin{definition}\label{definition2}
The stability of limit cycles and closed orbits family can be defined as follows.
\begin{itemize}
    \item[{\rm (1)}]{\it Stability of limit cycles}. For each limit cycle $\Gamma$, Jordan
curve theorem states that $\Gamma$ separates
any neighborhood $U_\Gamma$ of $\Gamma$ into two disjoint sets having $\Gamma$ as a boundary. We can regard $U_\Gamma$ as the disjoint union of $U_{\rm i}\cup \Gamma\cup U_{\rm e}$, where $U_{\rm i}$ and $U_{\rm e}$ are open sets
situated in the interior and exterior of $\Gamma$,  respectively.  If  $\lim_{t\to+\infty} d_{\mathcal{H}}(\Phi^t(x),\Gamma) = 0$ (resp. $\lim_{t\to-\infty} d_{\mathcal{H}}(\Phi^t(x),\Gamma) = 0$) for any $x\in U_{\rm i}\cup U_{\rm e}$ with $d_{\mathcal{H}}(\cdot,\cdot)$ being the distance between sets in the Hausdorff sense, then $\Gamma$ is said to be {\it stable} (resp. {\it unstable}). If  $\lim_{t\to+\infty}
d_{\mathcal{H}}(\Phi^t(x),\Gamma) = 0$  for any $x\in U_{\rm i}$ whereas $\lim_{t\to-\infty} d_{\mathcal{H}}(\Phi^t(x),\Gamma ) = 0$  for any $x\in U_{\rm e}$ (or the other way round),  then  $\Gamma$ is said to be {\it semi-stable}.
\medskip

\item[{\rm (2)}] {\it  Stability of the family of closed orbits}. Given any connected boundary $\Gamma$ of the closed orbit  family $\mathcal{K}$ which is a periodic orbit as assumed in Hypothesis \ref{assum1}, we define $U_{\rm e}\subset U_\Gamma$ as the open set
situated in the exterior of $\mathcal{K}$.
If $\lim_{t\to+\infty} d_{\mathcal{H}}(\Phi^t(x),\Gamma) = 0$ (resp. $\lim_{t\to-\infty} d_{\mathcal{H}}(\Phi^t(x),\Gamma) = 0$)  for any $x\in U_{\rm e}$, then we say that $\mathcal{K}$ is {\it stable} (resp. {\it unstable}) on the boundary  $\Gamma$. 
The  family  $\mathcal{K}$ is stable (resp. unstable) if $\mathcal{K}$ is stable (resp. unstable) on all connected boundary.
\end{itemize}

\end{definition}
It is well-known that any limit cycle  of system \eqref{system} is either stable, unstable or semi-stable (see e.g. \cite{P2001}).  The stability of the union of two homoclinic orbits connected by a hyperbolic saddle can be defined by  the same way as in Definition \ref{definition2}. 



The main result in this paper can be formulated as follows.

\begin{theorem}\label{mainresult}
Suppose that \eqref{assumption1} is fulfilled and the limit set of system \eqref{system} in $\Omega$ consists of a finite number of connected components    $\{\mathcal{K}_i:1\leq i\leq n\}$ satisfying Hypothesis {\rm \ref{assum1}}. 
Let  $\lambda(A)$ denote the principal eigenvalue of \eqref{1}.
  Then
  $$\lim_{A\to\infty}\lambda(A)=\min_{1\leq i\leq n}\{\Lambda(\mathcal{K}_i) 
  \},$$
  whereas the value $\Lambda(\mathcal{K}_i)$ is defined as follows.
\begin{itemize}
    \item[{(i)}] If $\mathcal{K}_i$ is a stable fixed point, then $\Lambda(\mathcal{K}_i)=c(\mathcal{K}_i)$;
    \smallskip
    \item[{(ii)}] If $\mathcal{K}_i$ is a stable limit cycle given by $\mathcal{K}_i=\{P_i(t)\in \mathbb{R}^2:t\in[0,T)\}$ with some $T$-periodic solution $P_i(t)$ of \eqref{system},  then
    $$\Lambda(\mathcal{K}_i)=\frac{1}{T}\int_0^T c(P_i(t)){\rm d}t;$$
    \item[{(iii)}] If $\mathcal{K}_i$ is a union of two stable homoclinic orbits connected by a hyperbolic saddle $x_i$, 
    then  $\Lambda(\mathcal{K}_i)=c(x_i)$;
    \smallskip
     \item[{(iv)}] If $\mathcal{K}_i$ is a family of  closed orbits, 
     then
     $$\Lambda(\mathcal{K}_i)= \inf_{\phi\in \mathbf{I}}\left[\frac{\int_{\mathcal{K}_i} (|\nabla\phi|^2+c(x)\phi^2)\mathrm{d}x}{\int_{\mathcal{K}_i} \phi^2\mathrm{d}x}\right],$$
     whereas  $\mathbf{I}\subset H^1(\mathcal{K}_i)$ can be defined as follows: Let $\Gamma$ be  limit cycles on $\partial \mathcal{K}_i$ such that  $\mathcal{K}_i$ is unstable  on $\Gamma$ and is stable on  $\partial\mathcal{K}_i\setminus \Gamma$, which allows for an empty set. 
        Then $\mathbf{I}:=\{\phi\in H^1(\mathcal{K}_i): \phi=0 \text{ on }  \Gamma, \  \mathbf{b}\cdot\nabla\phi =0 \text{ a.e. in }\mathcal{K}_i\}$;
   \smallskip
   \item[{(v)}] Otherwise,  $\mathcal{K}_i$ is 
   either an unstable fixed point, or a saddle,  or an unstable/semi-stable limit cycle, or  unstable/semi-stable homoclinic orbits, then
   $\Lambda(\mathcal{K}_i)=+\infty$.
\end{itemize}
\end{theorem}

Under assumption \eqref{assumption1}, the connected components in the limit set of system \eqref{system} must include at least one stable  component, which is either a fixed point, or a limit cycle, or a union of two homoclinic orbits, or the family of closed orbits. Hence,  the limit of the principal eigenvalue given in Theorem \ref{mainresult} 
is unaffected by the unstable and semi-stable connected components 
as listed in Theorem \ref{mainresult}{\rm (v)}.

\begin{remark}\label{remark_theorem}
If ${\bf b}=\nabla m$ is a gradient field  with potential $m\in C^2(\Omega)$,  Hypothesis \ref{assum1} indicates that the limit set of \eqref{system} consists of a finite number of hyperbolic fixed points $\{x_i\}_{i=1}^n$. Among them the stable ones are denoted by $\{x_i\}_{i=1}^k$ for  $k\leq n$ without loss of generality, which turns out to be the local maximum  points
of potential $m$. Using  notations in Theorem \ref{mainresult}, we find that $\Lambda(x_i)=c(x_i)$ for $1\leq i\leq k$, while $\Lambda(x_i)=+\infty$ for $k+1\leq i\leq n$, then  
Theorem \ref{mainresult} yields
 $$\lim_{A\to\infty}\lambda(A)=\min\left\{\min_{1\leq i\leq k}\{c(x_i)
  \},\,\,+\infty\right\}=\min_{1\leq i\leq k}\{c(x_i)
  \},$$
 which is consistent with \eqref{liu0913-1}.
Consequently, Theorem  \ref{mainresult} extends the result in \cite{CL2008} to  general drift in two dimensional space.
\end{remark}

In the context of vanishing viscosity 
for advection-diffusion operators, the effect of the fixed points and limit cycles on the principal eigenvalue
can be traced back to
the works 
\cite{DEF1974,DF1978,WFS1997,F1973, K1980}. 
Some related results can also be found in  the Freidlin-Wentzell theory
 \cite{FW1984,W1975}. 
 However, the finding that the  saddles connecting homoclinic orbits
have the significant effect in our setting  appears to be new. 

\begin{remark}\label{saddle}
   When the limit set of \eqref{system} includes the combination of two stable homoclinic orbits connected by a hyperbolic saddle,  Theorem \ref{mainresult}{\rm(iii)} implies that the limit of the principal eigenvalue is dependent on the saddle only instead of the whole homoclinic orbits. This seems to be 
   surprising at the first look as such saddle is actually unstable. In fact, different from the isolated saddles, the saddles connecting homoclinic orbits can be viewed as a singular point on
the closed orbits constituted by homoclinic orbits. Hence,
    the limit in Theorem \ref{mainresult}{\rm(iii)} can be understood  as the average of function $c$ along the closed orbits with ``infinity large” weight on the saddle, which turns out to be the value of $c$ located at the saddle. See Remark \ref{remarkcoro} for more details and see also \cite{XCJ2022} for  some related discussions in the context of the limiting measures for stochastic ordinary differential equations.
\end{remark}

 To 
 facilitate the underlying connections among assertions {\rm (i)}-{\rm (iii)} in Theorem \ref{mainresult}, we consider the following  example.


{\it Example}. Set ${\rm H}(x):=\frac{x_2^2}{2}+\frac{x_1^4}{4}-\frac{x_1^2}{2}\in [-\frac{1}{4},\infty)$. Given any $\alpha\in[-\frac{1}{4},1)$, we define
$${\bf b}_\alpha(x)=(-\partial_{x_2}{\rm H}, \partial_{x_1}{\rm H})-({\rm H}(x)-\alpha)\nabla {\rm H}.$$
For any $\alpha\in(-\frac{1}{4},1)$,  system  \eqref{system} associated with the vector field ${\bf b}_\alpha$ possesses three hyperbolic fixed points $x_0=(0,0)$, $x_+=(1,0)$, and $x_-=(-1,0)$, with $x_0$ being a saddle. Notice that for any solution $\Phi^t(x)$ of  \eqref{system}, it holds that
\begin{equation}\label{liu1026-1}
    \frac{{\rm d}{\rm H}(\Phi^t(x))}{{\rm d}t}=-({\rm H}(\Phi^t(x))-\alpha)|\nabla {\rm H}(\Phi^t(x))|^2.
\end{equation}
This implies function $F(x)=(H(x)-\alpha)^2$ is a Lyapunov function for \eqref{system},
whence 
${\rm H}^{-1}(\alpha)$ is the unique stable component in the limit set of system \eqref{system}. See Fig. \ref{figureliu1-1-1} for some illustrations on the phase-portraits of system \eqref{system}.

The following result is a consequence of Theorem \ref{mainresult}, for which the proof  is
postponed to the appendix.

\begin{corollary}\label{coro1}
Assume $\Omega={\rm H}^{-1}([-\frac{1}{4},1))$ and let $\lambda(A)$ be the principal eigenvalue of \eqref{1} with
drift ${\bf b}={\bf b}_\alpha$ and  $\alpha\in[-\frac{1}{4},1)$. Then the following assertions hold.
\begin{itemize}
    \item[{\rm (i)}]  If $\alpha\in(0,1)$, then ${\rm H}^{-1}(\alpha)$ is a stable limit cycle of \eqref{system} which can be parameterized by ${\rm H}^{-1}(\alpha)=\{P_\alpha(t)\in \mathbb{R}^2:t\in[0,T_\alpha)\}$ with some $T_\alpha$-periodic solution $P_\alpha(t)$. Then
        $$\lim_{A\to\infty}\lambda(A)=\frac{1}{T_\alpha}\int_0^{T_\alpha} c(P_\alpha(t)){\rm d}t;$$
\item[{\rm (ii)}]  If $\alpha=0$, then ${\rm H}^{-1}(0)$ is a stable component consisting of two homoclinic orbits
connected by the hyperbolic  saddle $x_0$. Then
 $$\lim_{A\to\infty}\lambda(A)=c(x_0);$$
 \item[{\rm (iii)}] If $\alpha\in(-\frac{1}{4},0)$, then ${\rm H}^{-1}(\alpha)$ is composed of two stable limit cycles given by ${\rm H}^{-1}(\alpha)=\{P_{\alpha+}(t)\in \mathbb{R}^2:t\in[0,T_{\alpha+})\}\cup \{P_{\alpha-}(t)\in \mathbb{R}^2:t\in[0,T_{\alpha-})\}$. Then
         $$\lim_{A\to\infty}\lambda(A)=\min\left\{\frac{1}{T_{\alpha+}}\int_0^{T_{\alpha+}} c(P_{\alpha+}(t)){\rm d}t, \,\,\, \frac{1}{T_{\alpha-}}\int_0^{T_{\alpha-}} c(P_{\alpha-}(t)){\rm d}t\right\}.$$
    \item[{\rm (iv)}] If $\alpha=-\frac{1}{4}$, then ${\rm H}^{-1}(-\frac{1}{4})=\{x_+, x_-\}$ consists of two stable fixed points. Then
     $$\lim_{A\to\infty}\lambda(A)=\min\left\{c(x_+), \,\,\, c(x_-)\right\}.$$
\end{itemize}
\end{corollary}

\medskip
 \begin{figure}[http!!]
  \centering
\includegraphics[height=3.8in]{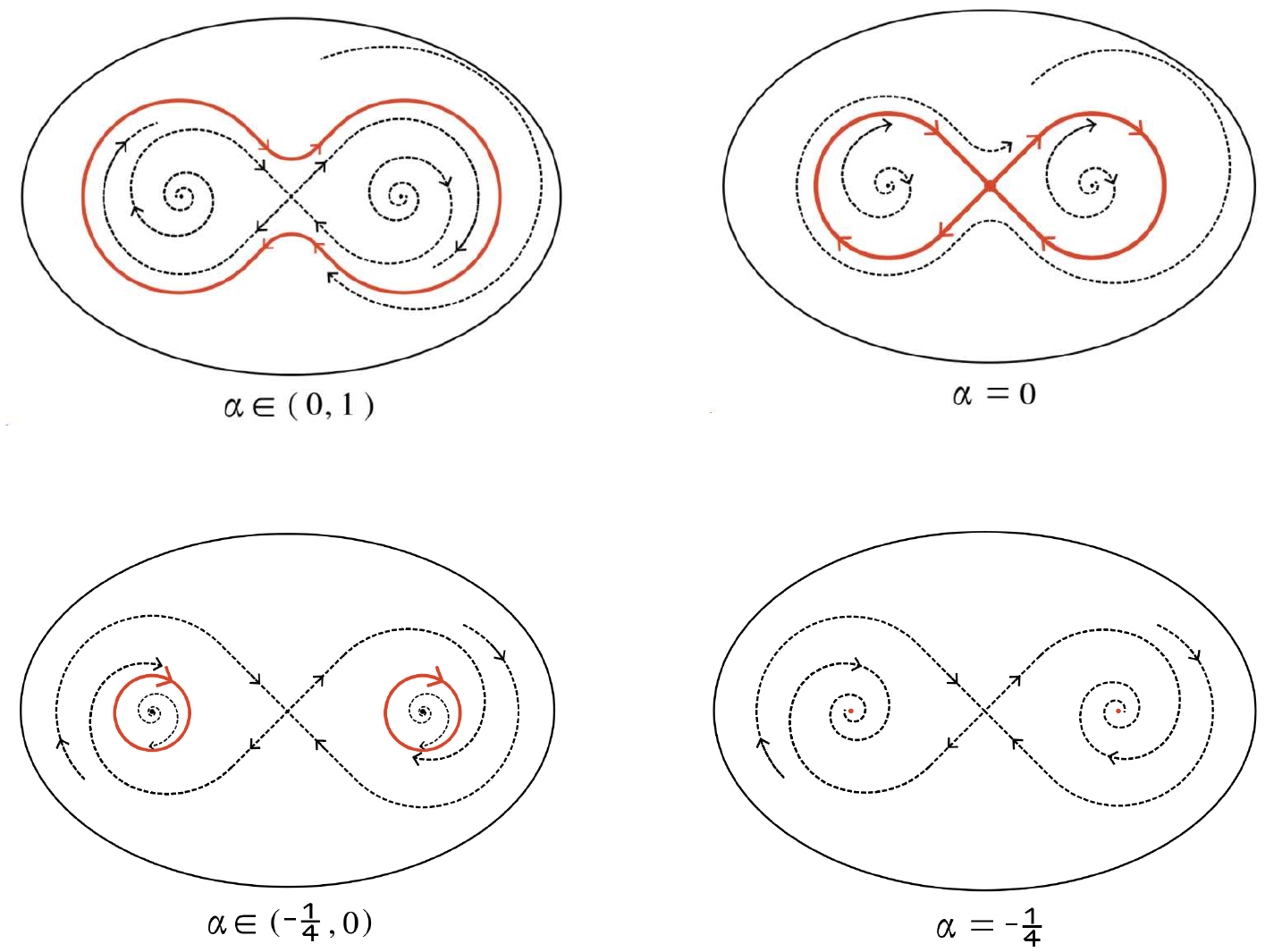}
   \caption{\small Illustrations for the phase-portraits of system \eqref{system} with ${\bf b}={\bf b}_\alpha$ for different $\alpha\in[-\frac{1}{4},1)$ as given in Corollary \ref{coro1}.
   }\label{figureliu1-1-1}
  \end{figure}

\begin{remark}\label{remarkcoro}
We illustrate the underlying connections among Corollary \ref{coro1}{\rm(i)} and {\rm(ii)}. For  $\alpha\in(0,1)$, let 
${\rm H}^{-1}(\alpha)=\{P_\alpha(t)\in \mathbb{R}^2:t\in[0,T_\alpha)\}$ as in Corollary \ref{coro1}{\rm(i)}.  We observe that
        $$\frac{1}{T_\alpha}\int_0^{T_\alpha} c(P_\alpha(t)){\rm d}t=\frac{\int_{{\rm H}^{-1}(\alpha)}\frac{c(x)}{|{\bf b}_\alpha(x)|}{\rm d}s}{\int_{{\rm H}^{-1}(\alpha)}\frac{1}{|{\bf b}_\alpha(x)|}{\rm d}s},$$
        where ${\rm d}s$ is the arc-length element  along  ${\rm H}^{-1}(\alpha)$.
Since $x_0$ is a hyperbolic saddle, 
we find
$$\int_{{\rm H}^{-1}(0)}\frac{1}{|{\bf b}_0(x)|}{\rm d}s=+\infty.$$
Hence, due to ${\bf b}_0(x)\neq0$ for all $x\in {\rm H}^{-1}(0)$ and $x\neq x_0$, by Corollary \ref{coro1}{\rm(i)} one can deduce
   $$\lim_{\alpha\to 0}\lim_{A\to\infty}\lambda(A)=\lim_{\alpha\to 0}\frac{\int_{{\rm H}^{-1}(\alpha)}\frac{c(x)}{|{\bf b}_\alpha(x)|}{\rm d}s}{\int_{{\rm H}^{-1}(\alpha)}\frac{1}{|{\bf b}_\alpha(x)|}{\rm d}s}=c(x_0), 
   $$
   which is consistent
with Corollary \ref{coro1}{\rm(ii)}. This means that the limit of $\lambda(A)$ as $A\to\infty$ is continuous at $\alpha=0$.
\end{remark}

In this example,
Corollary \ref{coro1} gives a complete description of the asymptotic behaviors
of the principal eigenvalue for different $\alpha\in [-\frac{1}{4},1)$. Since  $\Omega={\rm H}^{-1}([-\frac{1}{4},1))$,
the restriction  to $\alpha<1$ is  to ensure that \eqref{assumption1} holds true. In fact,  assumption \eqref{assumption1}  is imposed for the sake of convenience, 
so that the limit set of \eqref{system} lies in the interior of $\Omega$, and thereby simplifies our analysis on boundary $\partial \Omega$. Such restriction can be removed by some additional discussions.   We leave this for  future studies.




\section{\bf Case of hyperbolic fixed points}\label{S2}

In this section we establish Theorem \ref{mainresult}
in the case when  the limit set of  system \eqref{system} consists of hyperbolic fixed points only. The proofs will be used in subsequent sections to construct suitable super/sub-solutions near the fixed points.

\begin{theorem}\label{liuprop1}
Assume the limit set of  system \eqref{system} consists of a finite number of hyperbolic fixed points, and let $\{x_1,\cdots,x_k\}\subset\Omega$ denote the set of stable fixed points. Then
$$\lim_{A\to\infty}\lambda(A)=\min_{1\leq i\leq k}\{c(x_i)\}.$$
\end{theorem}
\begin{proof}
{\bf Lower bound estimate}: We first establish that the lower bound estimate
\begin{equation}\label{liu0619-1}
  \liminf_{A\to\infty}\lambda(A)\geq \min_{1\leq i\leq k}\{c(x_i)\}.
\end{equation}
Set ${\mathbb{S}}:=\{x_1,\cdots,x_k\}$. Let $n$ be
 the number of fixed points. We assume $n>k$, and the case $n=k$, namely all fixed points are stable, can be prove by the similar and simpler arguments.
 Denote by  ${\mathbb{U}}:=\{x_{k+1},\cdots,x_n\}\subset\Omega$  the set of the unstable fixed points and saddles of \eqref{system}. Given any $\epsilon>0$, we choose $\delta>0$ small such that $B^i_\delta\subset\Omega$ and $|c(x)-c(x_i)|\leq \epsilon/2$ for all $x\in B^i_\delta$ and $1\leq i\leq n$, where $B^i_\delta:=\{x\in\mathbb{R}^2: |x-x_i|<\delta\}$. To prove \eqref{liu0619-1}, we shall construct the positive
 super-solution $\overline{\varphi}\in C(\Omega)\cap C^2(\Omega\setminus(\cup_{i=1}^n\partial B^i_\delta))$ such that
\begin{equation}\label{liu0619-2}
 \left\{\begin{array}{ll}
-\Delta\overline{\varphi}-A\mathbf{b}\cdot\nabla\overline{\varphi}+c(x)\overline{\varphi}\geq\left[\min\limits_{1\leq i\leq k}\{c(x_i)\}-\epsilon\right] \overline{\varphi}&\mathrm{in} \,\, \Omega\setminus(\cup_{i=1}^n\partial B^i_\delta),\\
\medskip
(\nabla_-\overline{\varphi}(x)-\nabla_+\overline{\varphi}(x))\cdot \nu_i(x)> 0 & \mathrm{on}~\cup_{i=1}^n\partial B^i_\delta,\\
  \nabla\overline{\varphi}\cdot \mathbf{n}\geq 0 & \mathrm{on}~\partial\Omega  
  \end{array}\right.
 \end{equation}
provided that  $A$ is sufficiently large   and  $\delta$ is sufficiently small. Here $\nu_i(x)
$ denotes the outward unit normal vector on $\partial B^i_\delta$, and
\begin{equation}\label{liu0621-3}
    \nabla\overline{\varphi}_\pm(x)\cdot \nu_i(x):=\pm
 \lim_{t\to 0^+
 } \frac{\overline{\varphi}(x\pm t\nu_i(x))-\overline{\varphi}(x)}{t},\quad \forall x\in \partial B^i_\delta. 
\end{equation}
Then 
\eqref{liu0619-1} follows from the comparison principle and the arbitrariness of $\epsilon$.

{\bf Step 1.} For any $1\leq i\leq n$, we first construct $\overline{\varphi}\in C^2(B^i_\delta)$ such that \eqref{liu0619-2} holds on $B^i_\delta$. Given any fixed point  $x_i$,  denote by $\lambda_1$ and $\lambda_2$ the eigenvalues of the Jacobi matrix $D{\bf b}(x_i)$. We shall define the desired super-solution $\overline{\varphi}$ by considering different $\lambda_1$ and $\lambda_2$.

{\it Case 1}.  {\underline { $\lambda_1,\lambda_2\in\mathbb{R}$ and $\lambda_1\neq\lambda_2$}}: In this case,   there exists some invertible matrix $P$ such that
$
D{\bf b}(x_i)=P^{-1}\left(\begin{array}{cc}
    \lambda_1 & 0  \\
   0  & \lambda_2
\end{array}\right)P.
$
Next, we assume $P=I$ is an identity matrix, namely,
\begin{equation}\label{liu0620-1}
    D{\bf b}(x_i)=\left(\begin{array}{cc}
    \lambda_1 & 0  \\
   0  & \lambda_2
\end{array}\right)
\end{equation}
and the general case can be proved by the similar arguments.

{\rm (i)} If $\lambda_1,\lambda_2<0$,  i.e. $x_i\in \mathbb{S}$, then we define
\begin{equation}\label{supersolution1}
\overline{\varphi}(x):=\frac{8}{\epsilon}+\|x-x_i\|^2, \qquad \forall x\in B^i_\delta.
\end{equation}
By \eqref{liu0620-1} and in view of $\lambda_1,\lambda_2\neq 0$ by assumption, we note that
\begin{align*}
\mathbf{b}\cdot\nabla\overline{\varphi}(x)&=(x-x_i)^{\rm T}D{\bf b}(x_i)\nabla\overline{\varphi}(x)+o(\|x-x_i\|^2)\\
&\leq -(\min\{|\lambda_1|,|\lambda_2|\})\|x-x_i\|^2\leq 0, \quad\forall x\approx x_i.
\end{align*}
Hence, we may choose $\delta$ small if necessary  such that
\begin{align*}
&-\Delta\overline{\varphi}-A\mathbf{b}\cdot\nabla\overline{\varphi}+(c(x)-\min\limits_{1\leq i\leq k}\{c(x_i)\}+\epsilon)\\
\geq&-\Delta\overline{\varphi}+(c(x)-c(x_i)+\epsilon)\overline{\varphi}\\
    \geq & -4+\frac{\epsilon\overline{\varphi}}{2}\geq0,\qquad \forall x\in B^i_\delta,
\end{align*}
and thus  the constructed $\overline{\varphi}$ in \eqref{supersolution1} satisfies  \eqref{liu0619-2}.

{\rm (ii)} If $\lambda_1,\lambda_2>0$, i.e. $x_i\in {\mathbb{U}}$, then we define
\begin{equation}\label{supersolution2}
\overline{\varphi}(x):=\frac{8}{\epsilon^3}-\frac{\|x-x_i\|^2}{\sqrt{\delta}}, \qquad \forall x\in B^i_\delta.
\end{equation}
Similar to {\rm(i)}, by direct calculations we can choose $\delta$ small such that
$\mathbf{b}(x)\cdot\nabla\overline{\varphi}(x)\geq 0$ for $x\in B^i_\delta$, and furthermore,
\begin{align*}
    &-\Delta\overline{\varphi}-A\mathbf{b}\cdot\nabla\overline{\varphi}+(c(x)-\min\limits_{1\leq i\leq k}\{c(x_i)\})\overline{\varphi}
    \geq \frac{4}{\sqrt{\delta}}-\frac{16\|c\|_\infty}{\epsilon^3}>0,\quad \forall x\in B^i_\delta.
\end{align*}
Hence, such $\overline{\varphi}$ defined by \eqref{supersolution2} also satisfies \eqref{liu0619-2} on $B^i_\delta$.

{\rm (iii)} If $\lambda_1\lambda_2<0$, then we assume $\lambda_1<0<\lambda_2$ without loss of generality and define
$$\overline{\varphi}(x):=\frac{8}{\epsilon^2}-(x-x_i)^{\rm T}\left(\begin{array}{cc}
    \lambda_1/\sqrt{\delta} & 0  \\
   0  & \lambda_2
\end{array}\right)(x-x_i), \qquad \forall x\in B^i_\delta.$$
By \eqref{liu0620-1} we can verify that
\begin{align*}
\mathbf{b}\cdot\nabla\overline{\varphi}(x)
\leq -(\min\{\lambda_1^2/\sqrt{\delta},\lambda_2^2\}) \|x-x_i\|^2\leq 0, \quad\forall x\approx x_i.
\end{align*}
By choosing $\delta$ small, we arrive at
\begin{align*}
    &-\Delta\overline{\varphi}-A\mathbf{b}\cdot\nabla\overline{\varphi}+(c(x)-\min\limits_{1\leq i\leq k}\{c(x_i)\})\overline{\varphi}
    \geq \frac{-2\lambda_1}{\sqrt{\delta}}-2\lambda_2-\frac{16\|c\|_\infty}{\epsilon^2}>0,\quad \forall x\in B^i_\delta.
\end{align*}
Therefore,  the constructed $\overline{\varphi}$ satisfies \eqref{liu0619-2} on $B^i_\delta$.

{\it Case 2}.  {\underline {  $\lambda_1=\lambda_2$}}: If there exists some  invertible matrix  $P$ such that $
D{\bf b}(x_i)=P^{-1}\left(\begin{array}{cc}
    \lambda & 0  \\
   0  & \lambda
\end{array}\right)P,
$ for some $\lambda\in\mathbb{R}$,
then the super-solution $\overline{\varphi}$ can be constructed by the same way as in {\it Case 1}. Otherwise,  it holds that
$$
D{\bf b}(x_i)=P^{-1}\left(\begin{array}{cc}
    \lambda & 1  \\
   0  & \lambda
\end{array}\right)P,
$$
for some invertible matrix  $P$  and $\lambda\in\mathbb{R}$. We assume $P=I$ as in  {\it Case 1}.

{\rm (i)} If $\lambda<0$, then $x_i\in \mathbb{S}$ and we define
$$\overline{\varphi}(x):=\frac{4}{\epsilon}+\frac{1}{1/\lambda^2+1}(x-x_i)^{\rm T}\left(\begin{array}{cc}
    \frac{1}{\lambda^2} & 0  \\
   0  & 1
\end{array}\right)(x-x_i), \qquad \forall x\in B^i_\delta.$$
Direct calculations yield 
\begin{align*}
\mathbf{b}\cdot\nabla\overline{\varphi}(x)=&(x-x_i)^{\rm T}\left(\begin{array}{cc}
    \lambda & 1  \\
   0  & \lambda
\end{array}\right)\nabla\overline{\varphi}(x)+o(\|x-x_i\|^2)\\
=&-\frac{2}{1/\lambda^2+1}\left[\left(\left(\frac{1}{\sqrt{-2\lambda}},-\sqrt{-2\lambda}\right)^{\rm T}(x-x_i)\right)^2-(x-x_i)^{\rm T}\left(\begin{array}{cc}
    \frac{1}{-2\lambda} & 0  \\
   0  & \frac{-\lambda}{2}
\end{array}\right)(x-x_i)\right]\\
&+o(\|x-x_i\|^2)\leq 0, \quad\forall x\approx x_i,
\end{align*}
whence we can choose $\delta$ small such that 
\begin{align*}
&-\Delta\overline{\varphi}-A\mathbf{b}\cdot\nabla\overline{\varphi}+(c(x)-\min\limits_{1\leq i\leq k}\{c(x_i)\}+\epsilon)\overline{\varphi}\\
\geq&-\Delta\overline{\varphi}+(c(x)-c(x_i)+\epsilon)\overline{\varphi}\\
    \geq &-2+\frac{\epsilon}{2}\overline{\varphi}\geq0,\qquad \forall x\in B^i_\delta.
\end{align*}

{\rm (ii)} If $\lambda>0$, then  $x_i\in \mathbb{U}$ and we define
$$\overline{\varphi}(x):=\frac{4}{\epsilon^2}-\frac{(x-x_i)^{\rm T}}{\sqrt{\delta}}\left(\begin{array}{cc}
    \frac{1}{\lambda^2} & 0  \\
   0  & 1
\end{array}\right)(x-x_i), \qquad \forall x\in B^i_\delta.$$
Similar to {\rm(i)}, we can choose $\delta$ small such that
$\mathbf{b}\cdot\nabla\overline{\varphi}(x)\geq 0$, $\forall x\in B^i_\delta$, so that
\begin{align*}
    &-\Delta\overline{\varphi}-A\mathbf{b}\cdot\nabla\overline{\varphi}+(c(x)-\min\limits_{1\leq i\leq k}\{c(x_i)\})\overline{\varphi}\\
    \geq & \frac{\frac{1}{\lambda^2}+1}{\sqrt{\delta}}-\frac{8\|c\|_\infty}{\epsilon^2}>0,\quad \forall x\in B^i_\delta.
\end{align*}
Hence,  the constructed $\overline{\varphi}$ satisfies \eqref{liu0619-2} on $B^i_\delta$ in this case.

{\it Case 3}.  {\underline {  $\lambda_1=\alpha+i\beta$ and $\lambda_2=\alpha-i\beta$ with $\alpha,\beta\in\mathbb{R}$}}: In this case,
we assume that
$$
D{\bf b}(x_i)=\left(\begin{array}{cc}
    \alpha & \beta  \\
   -\beta  & \alpha
\end{array}\right)
$$
 without loss of generality.
Then we can define $\overline{\varphi}$ by \eqref{supersolution1} for $\alpha<0$, and define 
$\overline{\varphi}$ 
by \eqref{supersolution2} for $\alpha>0$. By the similar arguments as in {\it Case 1}, we can verify that $\overline{\varphi}$ satisfies \eqref{liu0619-2}. 


{\bf Step 2.} We complete the construction of $\overline{\varphi}$ on $\Omega$ such that \eqref{liu0619-2} holds.  By the construction  in Step 1, we can choose $\epsilon$  small if necessary such that
\begin{equation}\label{liu0621-1}
    \max_{x\in \cup_{x_i\in \mathbb{S}}B^i_\delta}\overline{\varphi}(x)<\min_{x\in \cup_{x_i\in {\mathbb{U}}}B^i_\delta}\overline{\varphi}(x).
\end{equation}
By assumption, we find the orbits of \eqref{system}
remain in $\Omega\setminus(\cup_{i=1}^n B^i_\delta)$ only a finite time.
By \eqref{assumption1} and \eqref{liu0621-1}, we may apply  \cite[Lemma 2.3]{DEF1974}  to deduce that there exists some positive function $\overline{\varphi}\in C^2(\Omega\setminus(\cup_{i=1}^n B^i_\delta))$ such that $\overline{\varphi}\in C(\Omega)$, $\nabla \overline{\varphi}\cdot {\bf n}\geq 0$ on $\partial\Omega$, and
\begin{equation}\label{supersolution3}
   {\bf b}(x)\cdot \nabla \overline{\varphi}(x)<0, \qquad \forall x\in \overline{\Omega}\setminus( \cup_{i=1}^n B^i_\delta).
\end{equation}
 Furthermore, according to the construction in Step 1, by \eqref{supersolution3} we can verify that
 $$(\nabla_-\overline{\varphi}(x)-\nabla_+\overline{\varphi}(x))\cdot \nu_i(x)> 0,\quad \forall x\in\cup_{i=1}^n\partial B^i_\delta,$$
 provided that $\delta$ is chosen  small if necessary. This implies  the boundary conditions in \eqref{liu0619-2} hold. Due to \eqref{supersolution3}, we can choose $A$ large such that the first inequality in \eqref{liu0619-2} holds. Combining with the construction of $\overline{\varphi}$ on $\cup_{i=1}^n  B^i_\delta$ as in Step 1, we conclude that the constructed super-solution $\overline{\varphi}$ satisfies \eqref{liu0619-2}  and thus the lower bound estimate \eqref{liu0619-1} follows.

\medskip

{\bf Upper bound estimate}: We next prove the upper bound estimate
\begin{equation}\label{liu0621-2}
  \limsup_{A\to\infty}\lambda(A)\leq \min_{1\leq i\leq k}\{c(x_i)\}.
\end{equation}
For any $x_i\in \mathbb{S}$, we assume without loss of generality that \eqref{liu0620-1} holds for some $\lambda_1,\lambda_2<0$. To prove \eqref{liu0621-2}, it suffices to show $\limsup_{A\to\infty}\lambda(A)\leq c(x_i)$ for any $x_i\in \mathbb{S}$.

To this end, we  fix  $x_i\in \mathbb{S}$ and choose some  $\delta>0$ such that $B^i_\delta\subset\Omega$, and $x_i$ is the unique fixed point of \eqref{system} in $B^i_\delta$, and moreover ${\bf b}(x)\cdot \nu_\delta(x)<0$, $\forall x\in \partial B^i_\delta$, with $\nu_\delta$ being the outward unit normal vector on $\partial B^i_\delta$. In what follows, given any $\epsilon>0$, we shall choose $r\in(0,\delta)$ and define positive sub-solution  $\underline{\varphi}\in C(B^i_\delta)$ such that
\begin{equation}\label{liu0621-4}
 \left\{\begin{array}{ll}
\medskip
-\Delta\underline{\varphi}-A\mathbf{b}\cdot\nabla\underline{\varphi}+c(x)\underline{\varphi}\leq\left(c(x_i)+\epsilon\right) \underline{\varphi}&\mathrm{in} \,\, B^i_\delta\setminus \partial B^i_r,\\
\medskip
(\nabla_-\underline{\varphi}(x)-\nabla_+\underline{\varphi}(x))\cdot \nu_i(x)< 0 & \mathrm{on}~\partial B^i_r,\\
  \underline{\varphi}= 0 & \mathrm{on}~\partial B^i_\delta,  
  \end{array}\right.
 \end{equation}
provided that $A$ is sufficiently large. Here
$\nabla\underline{\varphi}_\pm(x)\cdot \nu_i(x)$ is defined as in \eqref{liu0621-3}. Then  \eqref{liu0621-2} follows from the comparison principle and the arbitrariness of $x_i\in \mathbb{S}$.

For this purpose, we define
\begin{equation}\label{liu0621-5}
    \underline{\varphi}(x):=1-\frac{\epsilon}{8}\|x-x_i\|^2, \qquad \forall x\in B^i_r.
\end{equation}
Direct calculations from \eqref{liu0620-1} yield that
\begin{align*}
\mathbf{b}(x)\cdot\nabla\underline{\varphi}(x)&=(x-x_i)^{\rm T}D{\bf b}(x_i)\nabla\underline{\varphi}(x)+o(\|x-x_i\|^2)\\
&\geq \frac{\epsilon \min\{-\lambda_1,-\lambda_2\}}{8} \|x-x_i\|^2\geq 0, \quad\forall x\approx x_i.
\end{align*}
Thus, by \eqref{liu0621-5} we can choose $r$ small  such that
\begin{align*}
    &-\Delta\underline{\varphi}-A\mathbf{b}\cdot\nabla\underline{\varphi}+(c(x)-c(x_i)-\epsilon)\underline{\varphi}
    \leq \frac{\epsilon}{4}-\frac{\epsilon}{2}\underline{\varphi}\leq 0,\quad \forall x\in B^i_r.
\end{align*}
Note that  the orbits of \eqref{system} remain in $B^i_\delta\setminus B^i_r$ only a finite time, and
$\nabla_+\underline{\varphi}(x)\cdot \nu_i(x)=-\epsilon r/4<0$, $\forall x\in \partial B^i_r$.  We can apply  \cite[Lemma 2.3]{DEF1974} to define $\underline{\varphi}\in C(\overline{B}_\delta(x_i))$ such that the boundary conditions in \eqref{liu0621-4} hold, and
   ${\bf b}\cdot \nabla \underline{\varphi}>0$ on $B^i_\delta\setminus  B^i_r$,
from which we can choose $A$ large such that \eqref{liu0621-4} holds. This implies the upper bound estimate  \eqref{liu0621-2}. The proof of Theorem \ref{liuprop1} is thus complete.
\end{proof}

\section{\bf  Case of limit cycles 
}\label{S3}
In this section, we assume that the limit set of system \eqref{system} may contain limit cycles and determine the asymptotic behavior of the principal eigenvalue. 
The main result is formulated in Theorem \ref{Limit-cycle-4}, for which the proof is a combination of those of Theorems \ref{Limit-cycle}-\ref{Limit-cycle-3}.

\begin{theorem}[Stable case]\label{Limit-cycle}
Suppose that the limit set of system \eqref{system} consists of a finite number of hyperbolic fixed points and an {\bf stable  
limit cycle
} $\mathcal{C}:=\{P(t): t\in [0,T)\}$ inside $\Omega$ with period $T$.
Then there holds
$$\lim_{A\to\infty}\lambda(A)=\min\left\{\frac{1}{T}\int_0^T c(P(t)) {\rm d}t, \,\,\, \min_{1\leq i\leq k}\{c(x_i)\}\right\},$$
where  $\{x_1,\cdots,x_k\}$ denotes the set of stable fixed points.
\end{theorem}
\begin{proof}
{\bf Lower bound estimate}: We first prove that the lower bound estimate
\begin{equation}\label{liu-05}
  \liminf_{A\to\infty}\lambda(A)\geq \min\left\{\bbint c(P(t)) {\rm d}t, \,\,\, \min_{1\leq i\leq k}\{c(x_i)\}\right\},
\end{equation}
where we denote $\bbint c(P(t)) {\rm d}t:=\frac{1}{T}\int_0^T c(P(t)) {\rm d}t$ for simplicity.
Given any $\delta>0$, define  $\mathcal{O}_{\delta}:=\{x\in\Omega: d_{\mathcal{H}}(x,\mathcal{C})<\delta\}$ as the $\delta$-neighbourhood of  $\mathcal{C}$, with  $d_{\mathcal{H}}(\cdot,\cdot)$ being the distance between sets in the Hausdorff sense.
For any $x\in \mathcal{O}_{\delta}$ we perform a $C^2$ smooth change of coordinate $x\mapsto(t,r)$ such that
\begin{equation}\label{liu-change}
   x-P(t)=r\frac{\mathcal{J}{\bf b}(P(t))}{|{\bf b}(P(t))|}, \quad \forall (t,r)\in[0,T)\times (-\delta,\delta),
\end{equation}
where $\mathcal{J}=\left(
  \begin{array}{cc}
    0 & -1 \\
    1 & 0 \\
  \end{array}
\right)$, and thus $\frac{\mathcal{J}{\bf b}(P(t))}{|{\bf b}(P(t))|}$ is in fact the  outward unit normal vector of $\mathcal{C}$ at $x=P(t)$. We shall prove \eqref{liu-05} in the following two parts.

\medskip

\noindent{\bf Part 1.} We establish \eqref{liu-05} under the assumption
 \begin{equation}\label{liu0729-2}
     \begin{split}
         &{\bf b}(t,r)\cdot \mathcal{J}{\bf b}(P(t))\leq 0, \qquad\forall (t,r)\in[0,T)\times (0,\delta),\\
     &{\bf b}(t,r)\cdot \mathcal{J}{\bf b}(P(t))\geq 0, \qquad \forall (t,r)\in[0,T)\times (-\delta,0).
     \end{split}
 \end{equation}
This a special case under our assumption that  the limit cycle $\mathcal{C}$ is stable, which is considered first to illustrate our ideas of the proof for the general case.

{\bf Step 1}. Given any $\epsilon>0$, we shall define some $\overline{\phi}\in C^2(\mathcal{O}_{\delta})$ and choose $\delta$ small such that
\begin{equation}\label{liu-04}
-\Delta\overline{\phi}-A\mathbf{b}\cdot\nabla\overline{\phi}+c(x)\overline{\phi}\geq \left[\bbint c(P(s)){\rm d}s-\epsilon\right]\overline{\phi},\quad \forall x\in \mathcal{O}_{\delta},
\end{equation}
provided that  $A$ is sufficiently large, and
\begin{equation}\label{liu0524-3}
\nabla\overline{\phi}(x) \cdot \nu_1(x) >\varepsilon \delta\overline{\phi}(x), 
\quad \forall x\in \partial   \mathcal{O}_{\delta},
\end{equation}
where 
$\nu_1$ denotes the outward unit normal vector  on $ \mathcal{O}_{\delta}$, and $\varepsilon>0$ is a constant independent of $A$ to be determined later.

By \eqref{liu-change}, direct calculations yield
\begin{equation}\label{liu-01}
\begin{split}
   \partial_t x={\bf b}(P(t))+r\frac{{\rm d}}{{\rm d} t}\frac{\mathcal{J}{\bf b} (P(t))}{|{\bf b}(P(t))|} \quad \text{and}\quad\partial_r x=\frac{\mathcal{J}{\bf b} (P(t))}{|{\bf b}(P(t))|}.
    \end{split}
\end{equation}
Define
\begin{equation}\label{liu0524-2}
    \overline{\phi}(t,r):=\exp{\left\{\frac{1}{A}\left[\int_0^t c(P(s)){\rm d}s-t\bbint c(P(s)){\rm d}s\right]+\varepsilon r^2\right\}}, \quad \forall (t,r)\in\mathcal{O}_\delta.
\end{equation}
Then by \eqref{liu-change} and \eqref{liu-01} we can calculate that
\begin{equation}\label{liu-02}
\begin{split}
    \mathbf{b}(t,r)\cdot\nabla\overline{\phi}&= \mathbf{b}(t,r)\cdot\left[ \left(
  \begin{array}{cc}
  \smallskip
    \partial_t x_1 & \partial_t x_2 \\
   \partial_r x_1 & \partial_r x_2 \\
  \end{array}
\right)^{-1}\left(\begin{array}{c}
\smallskip
\partial_t \overline{\phi}\\
\partial_r \overline{\phi}
\end{array}
\right)\right]\\
=&\frac{1}{|{\bf b}(P(t))|+O(\delta)}\mathbf{b}(t,r)\cdot\left[ \left(
  \begin{array}{cc}
  \smallskip
    \partial_r x_2 & -\partial_t x_2 \\
   -\partial_r x_1 &  \partial_t x_1\\
  \end{array}
\right)\left(\begin{array}{c}
\smallskip
\partial_t \overline{\phi}\\
\partial_r \overline{\phi}
\end{array}
\right)\right]\\
=&\frac{1}{|{\bf b}(P(t))|+O(\delta)}\Bigg[(|{\bf b}(P(t))|+O(\delta))\partial_t \overline{\phi}+{\bf b}(t,r)\cdot \mathcal{J}{\bf b}(P(t)) \partial_r \overline{\phi}  \\
&-r\left(
{\bf b}(t,r)\cdot \frac{{\rm d}}{{\rm d} t}\frac{ {\bf b}(P(t))}{|{\bf b}(P(t))|}\right)\partial_r \overline{\phi} \Bigg],\quad  \forall (t,r)\in \mathcal{O}_\delta,
\end{split}
\end{equation}
where $O(\delta)$ is some constant independent of $A$ and satisfying  $O(\delta)\to 0$ as $\delta\searrow 0$.
By direct calculations, we derive that
\begin{align}\label{liu0628-5}
\frac{{\rm d}}{{\rm d} t}\frac{ {\bf b}(P(t))}{|{\bf b}(P(t))|}
=\left[ b_1(P(t))\frac{{\rm d}b_2(P(t))}{{\rm d} t}- b_2(P(t))\frac{{\rm d}b_1(P(t))}{{\rm d} t} \right]\frac{\mathcal{J}{\bf b}(P(t))}{|{\bf b}(P(t))|}=:Q(t)\mathcal{J}{\bf b}(P(t)).
\end{align}
We can derive from  
\eqref{liu0524-2}, \eqref{liu-02}, and \eqref{liu0628-5} that
\begin{align*}
    &-A\mathbf{b}(x)\cdot\nabla\overline{\phi}\\
    =&-\frac{A}{|{\bf b}(P(t))|+O(\delta)}\Bigg[(|{\bf b}(P(t))|+O(\delta))\partial_t \overline{\phi}+(1+Q(t)r){\bf b}(t,r)\cdot \mathcal{J}{\bf b}(P(t)) \partial_r \overline{\phi} \Bigg]\\
=&(1+O(\delta))\left[-c(P(t))+\bbint c(P(s)){\rm d}s\right]\overline{\phi}-\underbrace{\frac{2A\varepsilon r(1+ Q(t)r){\bf b}(t,r)\cdot \mathcal{J}{\bf b}(P(t))}{|{\bf b}(P(t))|+O(\delta)} \overline{\phi}}_{\leq 0 \,\,\, (\text{by } \eqref{liu0729-2})}\\
\geq&(1+O(\delta))\left[-c(P(t))+\bbint c(P(s)){\rm d}s\right]\overline{\phi}, 
\end{align*}
which together with \eqref{liu0524-2} gives 
\begin{equation}\label{liu0729-6}
    \begin{split}
         &-\Delta\overline{\phi}-A\mathbf{b}(x)\cdot\nabla\overline{\phi}+c(x)\overline{\phi}\\
\geq&\left[O(1/A)+O(\varepsilon)\right]\overline{\phi}+(1+O(\delta))\left[-c(P(t))+\bbint c(P(s)){\rm d}s\right]\overline{\phi}+(c(P(t))+O(\delta))\overline{\phi}\\
\geq& \left[\bbint c(P(s)){\rm d}s+O(1/A)+O(\varepsilon)+O(\delta)\right]\overline{\phi},\qquad \forall (t,r)\in \mathcal{O}_\delta.
    \end{split}
\end{equation}
Hence, for the given $\epsilon>0$, we  may choose $\delta, \varepsilon$ small and $A$ large such that \eqref{liu-04} holds.

It remains to verify \eqref{liu0524-3}. By \eqref{liu0628-5}  we have
 $${\bf b}(P(t))\cdot\frac{{\rm d}}{{\rm d} t}\frac{ {\bf b}(P(t))}{|{\bf b}(P(t))|}=Q(t){\bf b}(P(t))\cdot \mathcal{J}{\bf b}(P(t))=0.$$
 Hence, for any $x\in \partial \mathcal{O}_\delta$, under the coordinate \eqref{liu-change} we deduce that
\begin{equation*}
\begin{split}
&\nabla\overline{\phi}(x) \cdot \nu_1(x)\\
=& \frac{1}{|{\bf b}(P(t))|+O(\delta)}\left[ \left(
  \begin{array}{cc}
  \smallskip
    \partial_r x_2 & -\partial_t x_2 \\
   -\partial_r x_1 &  \partial_t x_1\\
  \end{array}
\right)\left(\begin{array}{c}
\smallskip
\partial_t \overline{\phi}\\
\partial_r \overline{\phi}
\end{array}
\right)\right] \cdot \frac{\mathcal{J}{\bf b}(P(t))}{|{\bf b}(P(t))|}\Bigg |_{r=\delta}\\
=&\frac{1}{|{\bf b}(P(t))|+O(\delta)}\Big[ |{\bf b}(P(t))|-\frac{r}{|{\bf b}(P(t))|}\left(
{\bf b}(P(t))\cdot\frac{{\rm d}}{{\rm d} t}\frac{ {\bf b}(P(t))}{|{\bf b}(P(t))|}\right)\Big]\partial_r \overline{\phi}\Bigg |_{r=\delta}\\
=&\frac{2\varepsilon \delta|{\bf b}(P(t))|}{|{\bf b}(P(t))|+O(\delta)}\overline{\phi},
\end{split}
\end{equation*}
which implies \eqref{liu0524-3} holds for $r=\delta$. The verification of \eqref{liu0524-3} for $r=-\delta$ is similar. Therefore, \eqref{liu-04} and \eqref{liu0524-3} are proved, and Step 1 is now complete.


{\bf Step 2}. We prove \eqref{liu-05}. Let $\{x_1,\cdots,x_n\}$ be the set of fixed points of \eqref{system} with $n\geq k$. Set
   $\tilde{\mathcal{O}}_\delta:=\{x\in\Omega:\exists 1\leq i\leq n, \,\, |x-x_i|<\delta \}$.
Given any $\epsilon>0$, we shall construct  a positive super-solution $\overline{\varphi}\in C(\overline{\Omega})\cap C^2(\Omega\setminus(\partial \mathcal{O}_\delta\cup \partial \tilde{\mathcal{O}}_\delta))$ such that
\begin{equation}\label{liu0524-1}
 \left\{\begin{array}{ll}
-\Delta\overline{\varphi}-A\mathbf{b}\cdot\nabla\overline{\varphi}+c(x)\overline{\varphi}\\
\medskip
\qquad\qquad\geq\left[\min\left\{\bbint c(P(t)) {\rm d}t, \,\, \min\limits_{1\leq i\leq k}\{c(x_i)\}\right\}-\epsilon\right] \overline{\varphi}&\mathrm{in} \,\, \Omega\setminus(\partial \mathcal{O}_\delta\cup \partial \tilde{\mathcal{O}}_\delta),\\
\medskip
(\nabla_-\overline{\varphi}(x)-\nabla_+\overline{\varphi}(x))\cdot \nu(x)> 0 & \mathrm{on}~\partial\mathcal{O}_\delta\cup \partial \tilde{\mathcal{O}}_\delta,\\
  \nabla\overline{\varphi}\cdot \mathbf{n}\geq 0 & \mathrm{on}~\partial\Omega,  
  \end{array}\right.
 \end{equation}
 provided that $A$ is sufficiently large. Here  
 $\nu(x)$ is the outward unit normal vector of $\partial\mathcal{O}_\delta\cup \partial \tilde{\mathcal{O}}_\delta$ (so that $\nu(x)=\nu_1(x)$, $\forall  x\in \partial\mathcal{O}_\delta$), and $\nabla\overline{\varphi}_\pm\cdot \nu$ is defined as in \eqref{liu0621-3}.
 Then \eqref{liu-05}  follows from the comparison principle.

By same arguments as in Theorem \ref{liuprop1}, we may define $C^2$ function $\overline{\psi}(x)$ such that
\begin{equation*}
-\Delta\overline{\psi}-A\mathbf{b}\cdot\nabla\overline{\psi}+c(x)\overline{\psi}\geq \left[\min\limits_{1\leq i\leq k}\{c(x_i)\}-\epsilon\right]\overline{\psi},\quad \forall x\in \tilde{\mathcal{O}}_\delta,
\end{equation*}
for small $\delta>0$. 
 Let $ \overline{\phi}(x)$ be defined in Step 1 which satisfies \eqref{liu-04} and \eqref{liu0524-3}.  Define 
\begin{equation}\label{liu0524-4}
    \overline{\varphi}(x)=\overline{\phi}(x), \,\,\forall x\in \mathcal{O}_\delta,\\
    \quad\text{and}\quad \overline{\varphi}(x)=\overline{\psi}(x), \,\,\forall x\in \tilde{\mathcal{O}}_\delta.
\end{equation}
We then define $ \overline{\varphi}(x)$  in $\Omega\setminus(\mathcal{O}_\delta\cup  \tilde{\mathcal{O}}_\delta)$  as follows.  
Since ${\bf b}(x)\cdot {\bf n}(x)<0$ on $\partial \Omega$, and ${\bf b}(x)\cdot \nu_1(x)<0$ on $\partial \mathcal{O}_\delta$ (as assumed in \eqref{liu0729-2}), we find the orbits of \eqref{system}
remain in $\Omega\setminus(\mathcal{O}_\delta\cup  \tilde{\mathcal{O}}_\delta)$ only a finite time.
Together with \eqref{liu0524-3},
we can apply  \cite[Lemma 2.3]{DEF1974} and its proof to show that there exists some $C^2$ function $G>0$ such that $G=\overline{\varphi}$ and $\nabla G\cdot \nu<\nabla\overline{\phi}\cdot \nu$ on $\partial \mathcal{O}_\delta\cup \partial \tilde{\mathcal{O}}_\delta$, and  
\begin{equation}\label{liu0524-5}
{\bf b}(x)\cdot\nabla G(x)<0, \,\,\forall x\in \Omega\setminus(\mathcal{O}_\delta\cup  \tilde{\mathcal{O}}_\delta), \quad \text{and}\quad \nabla G\cdot\mathbf{n}(x)\geq 0,\,\,\forall x\in\partial\Omega.
\end{equation}
Then we define $\overline{\varphi}=G$ in $\Omega\setminus(\mathcal{O}_\delta\cup  \tilde{\mathcal{O}}_\delta)$.

Next, we verify the constructed  $\overline{\varphi}$ satisfies \eqref{liu0524-1}. Indeed, by \eqref{liu0524-4} and the definitions of $\overline{\phi}$ and $\overline{\psi}$, the first equation of \eqref{liu0524-1} holds in $\mathcal{O}_\delta\cup  \tilde{\mathcal{O}}_\delta$. On
$\Omega\setminus(\mathcal{O}_\delta\cup  \tilde{\mathcal{O}}_\delta)$, in view of  $\overline{\varphi}=G$, by \eqref{liu0524-5} we find the boundary condition $\nabla\overline{\varphi}\cdot\mathbf{n}\geq 0$ on $\partial\Omega$ holds, and we can choose $A$ large such that the first equation of \eqref{liu0524-1} holds due to  ${\bf b}\cdot\nabla \overline{\varphi}<0$ on $\Omega\setminus(\mathcal{O}_\delta\cup  \tilde{\mathcal{O}}_\delta)$ as given by \eqref{liu0524-5}.
Moreover, by definition we find
\begin{align*}
   (\nabla_-\overline{\varphi}(x)-\nabla_+\overline{\varphi}(x))\cdot \nu(x)&=\nabla\overline{\phi}(x)\cdot \mathbf{\nu}(x)
  -\nabla G(x)\cdot \nu(x)>0,\quad \forall x\in \partial \mathcal{O}_\delta\cup \partial \tilde{\mathcal{O}}_\delta,
\end{align*}
and thus \eqref{liu-05} follows. Part 1 is now complete.

\medskip

\noindent{\bf Part 2.} We establish \eqref{liu-05} for the general case that the limit cycle $\mathcal{C}$ is stable. The proof is divided into the following two steps.

{\bf Step 1}. We prove that for any $\ell\in(-\delta,\delta)$, there exists $\eta_\ell\in\mathbb{R}$ satisfying $\ell \eta_\ell\geq 0$  such that $|\eta_\ell|$ is increasing in $|\ell|$, and the solution $x_\ell(\tau)$ of the initial value problem
\begin{equation}\label{liu_04-0729}
\left\{\aligned
  &\,\,\frac{{\rm d} x(\tau)}{{\rm d} \tau}= ({\bf I}+\eta_\ell \mathcal{J})\mathbf{b}(x(\tau)),\qquad t>0,\\
  &\,\, x(0)=P(0)+\ell\mathcal{J}{\bf b}(P(0))
 \endaligned \right.
\end{equation}
is periodic, namely, $x_\ell(T_\ell)=x_\ell(0)$ for some $T_\ell>0$. We will establish this result for $\ell\geq 0$ and the proof for $\ell<0$ is rather similar.



To this end, we  assume without loss of generality that  the initial value in \eqref{liu_04-0729} satisfies $x_\ell(0)=(0,\ell)$.
For each $\eta\geq 0$ and $\ell\in(0,\delta)$, we define $x_\ell(\tau,\eta)$ as the solution of
\begin{equation}\label{liu_04}
  \frac{{\rm d} x(\tau)}{{\rm d} \tau}= ({\bf I}+\eta \mathcal{J})\mathbf{b}(x(\tau)),\qquad t>0,
\end{equation}
 with initial data $x(0)=(0, \ell)$.  We can  choose $\delta$ small if necessary and select some $\eta^*>0$ independent of $\delta$  such that for any $\eta\in[0,\eta^*]$ and $\ell\in[0,\delta]$, the
trajectory  $x_\ell(\tau,\eta)$ 
returns to $y$-axis in some time 
close to $T$, and intersects $y$-axis at $(0, e_\ell(\eta))$.

 \begin{figure}[http!!]
  \centering
\includegraphics[height=2.8in]{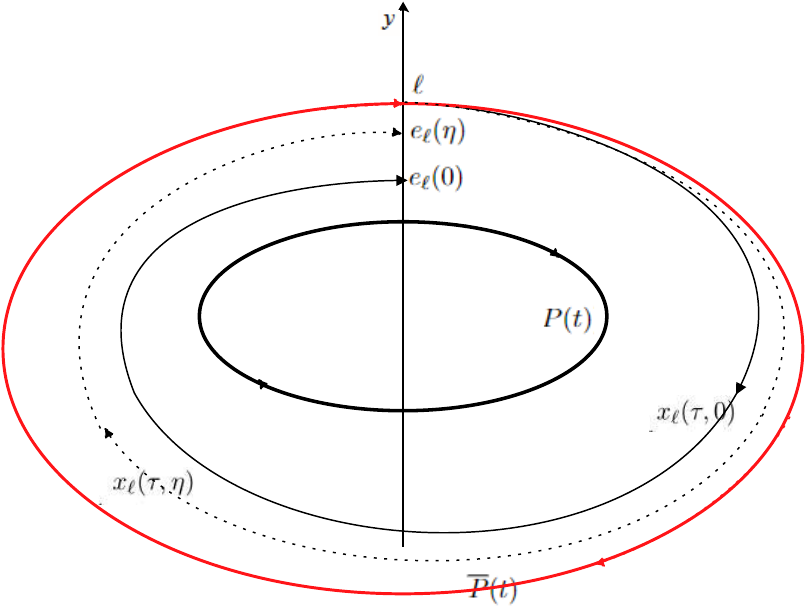}
  \caption{\small Illustration for the solution $x_\ell(\tau,\eta)$ of \eqref{liu_04} as well as the increasing function $e_\ell(\eta)$.
  }\label{figureliu1}
  \end{figure}

We first claim $e_\ell(\eta)$ is increasing in $\eta$ for fixed $\ell\in[0,\delta)$. Indeed, suppose on the contrary that $e_\ell(\eta_1)\geq e_\ell(\eta_2)$ for some  $0<\eta_1<\eta_2\leq \eta^*$.  In view of  $x_\ell(0,\eta_1)=x_\ell(0,\eta_2)$ and
$$\frac{{\rm d}x_\ell(\cdot,\eta_1)}{{\rm d}\tau}\Big |_{t=0}\cdot (0,1)<\frac{{\rm d}x_\ell(\cdot,\eta_2)}{{\rm d}\tau}\Big |_{t=0}\cdot (0,1),$$
we can find some $\tau_1,\tau_2>0$ such that $x_\ell(\tau_1,\eta_1)=x_\ell(\tau_2,\eta_2)=:x_*$, and the curves $\{x_\ell(\tau,\eta_1): \tau\in (0,\tau_1)\}$ and $\{x_\ell(\tau,\eta_2): \tau\in (0,\tau_2)\}$ do not intersect. Then there exist $c_1,c_2\geq 0$ such that
$$\frac{{\rm d} x_\ell(\cdot, \eta_1)}{{\rm d} \tau}\Big |_{\tau=\tau_1}=c_1 \frac{{\rm d} x_\ell(\cdot, \eta_2)}{{\rm d}\tau}\Big |_{\tau=\tau_2}+c_2\mathcal{J}\frac{{\rm d} x_\ell(\cdot, \eta_2)}{{\rm d} \tau}\Big |_{\tau=\tau_2},$$
which together with \eqref{liu_04} gives
$${\bf b}(x_*)+\eta_1 \mathcal{J}{\bf b}(x_*)=c_1{\bf b}(x_*)+c_1 \eta_2 \mathcal{J}{\bf b}(x_*)+c_2  \mathcal{J}{\bf b}(x_*)-c_2 \eta_2 {\bf b}(x_*).$$
This leads to $c_1-c_2\eta_2=1$ and $c_1\eta_2+c_2=\eta_1$, and thus $\eta_1-\eta_2=c_2(1+\eta_2^2)\geq 0$, which is a contradiction. Hence, $e_\ell(\eta)$ is  increasing in $\eta$.

We next prove that for any $\ell\in(0,\delta)$,  there exists some  $\eta_\ell\in (0,\eta^*]$ such that $e_\ell(\eta_\ell)=\ell$, i.e. the solution of \eqref{liu_04} with $\eta=\eta_\ell$ starting from $(0,\ell)$ is a periodic solution, and then the monotonicity of  $\eta_\ell$ in $\ell$ is a direct consequence of that of $e_\ell(\eta)$, which completes Step 1. Since the underlying limit cycle $\mathcal{C}$ is stable, it is easily seen that $e_\ell(0)<\ell$ for all $\ell\in (0,\delta)$. Thus,  by the monotonicity of $e_\ell(\eta)$ in $\eta$, it suffices to show $e_\ell(\eta^*)>\ell$.  Note that $x_0(t,0)$ is the periodic solution of \eqref{system}, i.e. $e_0(0)=0$, and thus $e_0(\eta^*)>0$ by the  monotonicity of $e_0(\eta)$. By continuity we can choose $\delta>0$ small if necessary such that $e_\ell(\eta^*)>\ell$ for all $\ell\in [0,\delta]$. Therefore, 
 \eqref{liu_04-0729} admits a periodic solution. Step 1 is complete.

{\bf Step 2}. We prove \eqref{liu-05}.  For any $\ell\in(-\delta,\delta)$, let $x_\ell(\tau)$ be the periodic solution satisfying \eqref{liu_04-0729}. We perform a $C^2$-smooth change of coordinate $x\mapsto(\tau, \ell)$ ($\tau\in[0,T_\ell)$) such that $x= x_\ell(\tau)$. Then we define the region
\begin{equation}\label{Odelta}
    \mathbb{O}_\delta:=\{(\tau, \ell): \ell\in (-\delta,\delta) \text{ and }\tau\in[0,T_\ell)\},
\end{equation}
 and let $\nu_2(x)=\pm \mathcal{J}\frac{{\rm d} x_{\pm \delta}(\tau)}{{\rm d} \tau}/|\frac{{\rm d} x_{\pm \delta}(\tau)}{{\rm d} \tau}|$ denote the outward unit normal vector  on $\partial\mathbb{O}_\delta$. Thus,
 by \eqref{liu_04-0729} we have
 \begin{equation}\label{liu0729-3}
    \begin{split}
       {\bf b}(x)\cdot \nu_2(x)&=\pm\frac{1}{|\frac{{\rm d} x_{\pm \delta}(\tau)}{{\rm d} \tau}|}{\bf b}(x_{\pm \delta}(\tau))\cdot \mathcal{J}({\bf I}+\eta_{\pm \delta} \mathcal{J})\mathbf{b}(x_{\pm\delta}(\tau))\\
     &=\mp\frac{\eta_{\pm \delta}}{|\frac{{\rm d} x_{\pm \delta}(\tau)}{{\rm d} \tau}|}|\mathbf{b}(x_{\pm\delta}(\tau))|^2<0, \quad \forall x\in  \partial\mathbb{O}_\delta.
    \end{split}
 \end{equation}

 Given any $\epsilon>0$, analogous to Step 1 of Part 1,  we shall define some $\overline{\phi}\in C^2(\mathbb{O}_{\delta})$ such that \eqref{liu-04}  holds on $\mathbb{O}_{\delta}$ for large  $A$, and similar to \eqref{liu0524-3},
\begin{equation}\label{liu0524-3-0729}
\nabla\overline{\phi}(x) \cdot \nu(x) >\varepsilon \delta\overline{\phi}(x), 
\quad \forall x\in \partial   \mathbb{O}_{\delta},
\end{equation}
where $\varepsilon>0$ will be determined later. Then based on \eqref{liu-04}, \eqref{liu0729-3}, and  \eqref{liu0524-3-0729}, we can apply  Step 2 of Part 1  to establish \eqref{liu-05} by constructing a positive super-solution satisfying \eqref{liu0524-1}.

To this end, we define
\begin{equation}\label{liu0729-4}
    \overline{\phi}(\tau,\ell):=\exp{\left\{\frac{1}{A}\left[\int_0^\tau c(x_\ell(s)){\rm d}s-\tau\bbint c(x_\ell (s)){\rm d}s\right]+\varepsilon \ell^2\right\}}, \quad \forall (\tau, \ell)\in\mathbb{O}_\delta.
\end{equation}
Note that $x_\ell(\tau)\to P(\tau)$ as $\ell\to 0$. It follows from \eqref{liu-01} that
$$\lim_{\ell\to 0}\partial_\ell x\cdot \mathcal{J}{\bf b}(x_\ell(\tau))=|{\bf b}(P(\tau))|,$$
and thus $\partial_\ell x\cdot \mathcal{J}{\bf b}(x_\ell(\tau))>0$ in $\mathbb{O}_\delta$ for small $\delta$.  Then by \eqref{liu_04-0729} and \eqref{liu0729-4} we have
\begin{align}\label{liu0729-5}
    \mathbf{b}(\tau, \ell)\cdot\nabla\overline{\phi}&= \mathbf{b}(x_\ell(\tau))\cdot\left[ \left(
  \begin{array}{cc}
  \smallskip
    \partial_\tau x_1 & \partial_\tau x_2 \\
   \partial_\ell x_1 & \partial_\ell x_2 \\
  \end{array}
\right)^{-1}\left(\begin{array}{c}
\smallskip
\partial_\tau \overline{\phi}\\
\partial_\ell \overline{\phi}
\end{array}
\right)\right] \notag\\
=&\frac{\mathbf{b}(x_\ell(\tau))}{\partial_\ell x\cdot \mathcal{J}{\bf b}+O(\eta_\ell)}\cdot\left[ \left(
  \begin{array}{cc}
  \smallskip
    \partial_\ell x_2 & -\partial_\tau x_2 \\
   -\partial_\ell x_1 &  \partial_\tau x_1\\
  \end{array}
\right)\left(\begin{array}{c}
\smallskip
\partial_\tau \overline{\phi}\\
\partial_\ell \overline{\phi}
\end{array}
\right)\right] \notag\\
=&\frac{\partial_\ell x\cdot \mathcal{J}{\bf b} }{\partial_\ell x\cdot \mathcal{J}{\bf b}+O(\eta_\ell)}\partial_\tau \overline{\phi}-\frac{\partial_\tau x\cdot \mathcal{J}{\bf b} }{\partial_\ell x\cdot \mathcal{J}{\bf b}+O(\eta_\ell)}\partial_\ell \overline{\phi}  \notag\\
=&(1+O(\eta_\ell))\partial_\tau \overline{\phi}-\frac{({\bf I}+\eta_\ell \mathcal{J})\mathbf{b}(x_\ell(\tau))\cdot \mathcal{J}{\bf b}((x_\ell(\tau)) }{\partial_\ell x\cdot \mathcal{J}{\bf b}+O(\eta_\ell)}\partial_\ell \overline{\phi}  \\
=&\frac{1+O(\eta_\ell))}{A}\left[c(x_\ell(\tau))-\bbint c(x_\ell (s)){\rm d}s\right] \overline{\phi}-\underbrace{\frac{2\varepsilon\ell\eta_\ell|\mathbf{b}(x_\ell(\tau))|^2 }{\partial_\ell x\cdot \mathcal{J}{\bf b}+O(\eta_\ell)}\overline{\phi}}_{\geq 0}  \notag\\
\leq &\frac{1+O(\eta_\ell))}{A}\left[c(x_\ell(\tau))-\bbint c(x_\ell (s)){\rm d}s\right] \overline{\phi},  \qquad \forall (\tau,\ell)\in \mathbb{O}_\delta.\notag
\end{align}
Note that $|\eta_\ell|\leq \max \{|\eta_\delta|,|\eta_{-\delta}|\}$ for $\ell\in[-\delta,\delta]$ and $\eta_{\pm \delta} \to 0$ as $\delta\searrow 0$ as proved in Step 1, and thus $O(\eta_\ell)\to 0$ as $\delta\searrow 0$. Hence, similar to \eqref{liu0729-6}, by \eqref{liu0729-4} and \eqref{liu0729-5} we deduce that
\begin{align*}
    &-\Delta\overline{\phi}-A\mathbf{b}(x)\cdot\nabla\overline{\phi}+c(x)\overline{\phi}\\
\geq&\left[O(1/A)+O(\varepsilon)\right]\overline{\phi}+(1+O(\eta_\ell))\left[-c(x_\ell(\tau))+\bbint c(x_\ell(s)){\rm d}s\right]\overline{\phi}+c(x_\ell(\tau))\overline{\phi}\\
\geq& \left[\bbint c(P(s)){\rm d}s+O(1/A)+O(\varepsilon)+O(\delta)\right]\overline{\phi},\qquad \forall (\tau,\ell)\in \mathbb{O}_\delta,
\end{align*}
where the last inequality holds due to $|x_\ell(\tau)-P(\tau)|\leq O(\delta)$. Then
we  can choose $\delta, \varepsilon$ small and $A$ large such that \eqref{liu-04} holds.

It remains to prove \eqref{liu0524-3-0729}. As in Step 1, we only verify \eqref{liu0524-3-0729} for $\ell=\delta$, as the verification for  $\ell=-\delta$ is similar. By virtue of \eqref{liu_04-0729} and \eqref{liu0729-4} we calculate that
\begin{equation}\label{liu_0730-6}
\begin{split}
&\nabla\overline{\phi}(x) \cdot \nu_2(x)\\
=& \frac{1}{\partial_\ell x\cdot \mathcal{J}{\bf b}+O(\eta_\delta)}\left[ \left(
  \begin{array}{cc}
  \smallskip
    \partial_\ell x_2 & -\partial_\tau x_2 \\
   -\partial_\ell x_1 &  \partial_\tau x_1\\
  \end{array}
\right)\left(\begin{array}{c}
\smallskip
\partial_\tau \overline{\phi}\\
\partial_\ell \overline{\phi}
\end{array}
\right)\right] \cdot \frac{\mathcal{J}\frac{{\rm d} x_{ \delta}(\tau)}{{\rm d} \tau}}{|\frac{{\rm d} x_{\delta}(\tau)}{{\rm d} \tau}|}\Bigg |_{\ell=\delta}\\
=&\frac{\mathcal{J}\partial_\tau x \cdot\mathcal{J}({\bf I}+\eta_{\delta}\mathcal{J}){\bf b}\partial_\ell \overline{\phi}+O(1/A)\overline{\phi}}{(\partial_\ell x\cdot \mathcal{J}{\bf b}+O(\eta_\delta))|\frac{{\rm d} x_{\delta}(\tau)}{{\rm d} \tau}|}\Bigg |_{\ell=\delta}\\
=&\frac{2\varepsilon \delta|({\bf I}+\eta_{\delta}\mathcal{J}){\bf b}|^2+O(1/A)}{(\partial_\ell x\cdot \mathcal{J}{\bf b}+O(\eta_\delta))(|{\bf b}(x_\delta(\tau))|+O(\eta_\delta))}\overline{\phi}.
\end{split}
\end{equation}
 Together with $\partial_\ell x\cdot \mathcal{J}{\bf b}\to |{\bf b}|$ and $\eta_\delta\to 0$ as $\delta\searrow 0$, this implies \eqref{liu0524-3-0729} holds by choosing $\delta$ small and $A$ large if necessary. Therefore, we have proved   \eqref{liu-04} and  \eqref{liu0524-3-0729}, so that \eqref{liu-05} can be   established by  Step 2 of Part 1.

\medskip

{\bf Upper bound estimate}: We next prove that the upper bound estimate
\begin{equation*}
  \limsup_{A\to\infty}\lambda(A)\leq \min\left\{\bbint c(P(t)) {\rm d}t,\,\, \min\limits_{1\leq i\leq k}\{c(x_i)\}\right\}.
\end{equation*}
The estimate $\limsup\limits_{A\to\infty}\lambda(A)\leq \min\limits_{1\leq i\leq k}\{c(x_i)\}$ can be proved by the same arguments as in Theorem \ref{liuprop1}. It remains to establish
\begin{equation}\label{liu-05-3}
  \limsup_{A\to\infty}\lambda(A)\leq \bbint c(P(t)) {\rm d}t.
\end{equation}
Let $\mathbb{O}_\delta$ be defined in \eqref{Odelta}.  Under the coordinate $x=(\tau,\ell)$ defined in Pert 2, we define
\begin{equation*}
    \underline{\phi}(\tau, \ell):=\exp{\left\{\frac{1}{A}\left[\int_0^\tau c(x_\ell(s)){\rm d}s-\tau\bbint c(x_\ell(s)){\rm d}s\right]+\varepsilon(\delta^2-\ell^2)\right\}},\quad \forall (\tau,\ell)\in\mathbb{O}_\delta.
\end{equation*}
Given any $\epsilon>0$, similar to Step 2 of Part 2 we can verify  $\overline{\phi}(x)$ satisfies
\begin{equation}\label{liu-04-1}
-\Delta\underline{\phi}-A\mathbf{b}\cdot\nabla\underline{\phi}+c(x)\underline{\phi}\leq \left[\bbint c(P(s)){\rm d}s+\epsilon\right]\underline{\phi},\quad \forall x\in \mathbb{O}_{\delta},
\end{equation}
for large  $A$ and small $\delta$, and
\begin{equation}\label{liu0524-3-1}
\nabla\underline{\phi}(x) \cdot \nu_2(x) <-\varepsilon\delta\underline{\phi}(x), \quad \forall x\in \partial   \mathbb{O}_{\delta}.
\end{equation}
Here $\nu_2$ denotes the outward unit normal vector  on $ \partial\mathcal{O}_{\delta}$.

We next choose some subset $U\subset\Omega$ such that {\rm (i)} $\mathbb{O}_\delta\subset U$; {\rm (ii)} $ {\bf b}(x)\cdot \nu_3(x)<0$, $\forall x\in\partial U$, with $\nu_3$ denoting the outward unit normal vector  on $\partial U$, and {\rm (iii)} there are no limit points of  system \eqref{system}  in $\overline{U}\setminus\mathbb{O}_\delta$. In view of  $\mathcal{C}\subset \mathbb{O}_\delta$, and ${\bf b}(x)\cdot \nu_1(x)<0$ on $\partial \mathbb{O}_\delta$, we may define $\underline{G}\in C^2(\overline{U}\setminus\mathbb{O}_\delta)$ such that $\underline{G}=\underline{\phi}$ on $\partial \mathbb{O}_\delta$,  $\underline{G}=0$ on $ \partial U$, and furthermore,
\begin{equation}\label{liu0524-5-1}
\underline{G}(x)>0 \quad\text{and}\quad {\bf b}(x)\cdot\nabla \underline{G}(x)>0, \,\,\quad \forall x\in \overline{U}\setminus\mathbb{O}_\delta.
\end{equation}
It thus follows that $\nabla\underline{G}(x)\cdot \nu_3(x)<0$, $\forall x\in\partial U$. Then we define $\underline{\varphi}\in C(\overline{U})$ such that
\begin{equation*}
    \underline{\varphi}(x):=\underline{\phi}(x), \,\,\forall x\in \mathbb{O}_\delta,
    \quad\text{and}\quad \underline{\varphi}(x)=\underline{G}(x), \,\,\quad \forall x\in \overline{U}\setminus\mathbb{O}_\delta.
\end{equation*}
Similar to Step 3, by \eqref{liu-04-1}, \eqref{liu0524-3-1}, and \eqref{liu0524-5-1}, we can choose $\delta$ small and $A$ large such that
\begin{equation*}
 \left\{\begin{array}{ll}
\medskip
-\Delta\underline{\varphi}-A\mathbf{b}\cdot\nabla\underline{\varphi}+c(x)\underline{\varphi}\leq\left[\bbint c(P(s)){\rm d}s+\epsilon\right] \underline{\varphi}&\mathrm{in} \,\, U\setminus\partial\mathbb{O}_\delta,\\
\medskip
(\nabla_-\underline{\varphi}(x)-\nabla_+\underline{\varphi}(x))\cdot \nu_1(x)<0 & \mathrm{on}~\partial\mathbb{O}_\delta,\\
  \underline{\varphi}=0,\,\,\nabla\underline{\varphi}\cdot \nu_3< 0 & \mathrm{on}~\partial U. 
  \end{array}\right.
 \end{equation*}
Then the upper bound estimate \eqref{liu-05-3} follows from the comparison principle and the Arbitrariness of $\epsilon$. This completes the proof.
\end{proof}


Our next result concerns the case when the limit set of system \eqref{system} contains an unstable limit cycle.
\begin{theorem}[Unstable case]\label{Limit-cycle-2}
Suppose that the limit set of system \eqref{system} consists of a finite number of hyperbolic fixed points and a 
{\bf unstable limit cycle 
} (periodic repeller) $\mathcal{C}$ inside $\Omega$ such that ${\bf b}\neq 0$ on $\mathcal{C}$. Let $\{x_1,\cdots,x_k\}$ denote the set of stable fixed points. Then
$$\lim_{A\to\infty}\lambda(A)= \min_{1\leq i\leq k}\{c(x_i)\}.$$
\end{theorem}
\begin{proof}
Note that the upper bound estimate $$\limsup\limits_{A\to\infty}\lambda(A)\leq \min\limits_{1\leq i\leq k}\{c(x_i)\}$$
can be established by the same arguments as Step 2 of Theorem \ref{liuprop1}. It suffices to show the lower bound estimate
\begin{equation}\label{liu0627}
  \liminf_{A\to\infty}\lambda(A)\geq \min\limits_{1\leq i\leq k}\{c(x_i)\}.
\end{equation}

{\bf Step 1}. For any small $\delta>0$, we shall define a region $\mathbb{O}_{\delta}\subset \Omega$ such that ${\bf b}(x)\cdot \nu(x)>0$, $\forall x\in\partial\mathbb{O}_{\delta}$, where  $\nu(x)$ is the outward unit normal vector of  $\partial \mathbb{O}_\delta$. Then we will construct some positive $\overline{\phi}\in C^2(\mathbb{O}_{\delta})$ such that
\begin{equation}\label{liu0627-1}
-\Delta\overline{\phi}-A\mathbf{b}\cdot\nabla\overline{\phi}+c(x)\overline{\phi}\geq \left[\min\limits_{1\leq i\leq k}\{c(x_i)\}\right]\overline{\phi},\quad \forall x\in \mathbb{O}_{\delta},
\end{equation}
provided that  $\delta$ is small, and
\begin{equation}\label{liu0628-1}
-2C\sqrt{\delta} <\nabla\overline{\phi}(x) \cdot \nu(x) <0, \quad \forall x\in \partial   \mathbb{O}_{\delta},
\end{equation}
 for some constant $C>0$ independent of $\delta$.

Set $\mathcal{C}:=\{P(t): t\in [0,T)\}$ for some $T$-periodic solution $P(t)$ of system \eqref{system}. Under the coordinate $x\mapsto(t,r)$ introduced by \eqref{liu-change},  for each  $(0,r)$ with $r\in[-\delta,\delta]$,  similar to Step 1 of Part 2 in the proof of Theorem \ref{Limit-cycle}, 
there exists some $\eta_\ell\in\mathbb{R}$ satisfying $\ell\eta_\ell\leq 0$ (in fact, it holds $\ell\eta_\ell<0$ if $r\neq 0$) such that the solution of \eqref{liu_04-0729}  
is a periodic solution with period $T_\ell$, which is denoted by $x_\ell(\tau)$. Here $\ell\eta_\ell\leq 0$ is due to the unstability of $\mathcal{C}$. As in Theorem \ref{Limit-cycle} we perform a $C^2$-smooth change of coordinate $x\mapsto(\tau, \ell)$ ($\tau\in[0,T_\ell)$) such that $x= x_\ell(\tau)$. We now define the region
$\mathbb{O}_\delta:=\{(\tau, \ell): \tau\in[0,T_\ell) \text{ and }  \ell\in[-\delta,\delta]\}$. Then the outward unit normal vector $\nu$ can be written as
 $\nu(x)=\pm \mathcal{J}\frac{{\rm d} x_{\pm \delta}(\tau)}{{\rm d} \tau}/|\frac{{\rm d} x_{\pm \delta}(\tau)}{{\rm d} \tau}|$, and thus by \eqref{liu_04-0729} we have
 \begin{equation}\label{liu0728-1}
    \begin{split}
       {\bf b}(x)\cdot \nu(x)
     >0, \quad \forall x\in\partial\mathbb{O}_\delta,
    \end{split}
 \end{equation}
which follows from the same arguments as in \eqref{liu0729-3}.

Next,  we define
\begin{equation}\label{liu-0727-1}
   \overline{\phi}(\tau, \ell):=1-\frac{\ell^2}{\sqrt{\delta}},\qquad  \forall(\tau, \ell)\in \mathbb{O}_\delta.
\end{equation}
By \eqref{liu_04-0729} and \eqref{liu-0727-1}, it is not difficult to verify that
\begin{equation}\label{liu-0728-7}
\begin{split}
     \mathbf{b}(\tau, \ell)\cdot\nabla\overline{\phi}=&-\frac{\ell|\nabla\overline{\phi}|}{|\ell|\left|\frac{{\rm d} x_\ell(\tau)}{{\rm d} \tau}\right|}\mathbf{b}(x_\ell(\tau) \cdot \mathcal{J}\frac{{\rm d} x_\ell(\tau)}{{\rm d} \tau}\\
     =&-\frac{\ell|\nabla\overline{\phi}|}{|\ell|\left|\frac{{\rm d} x_\ell(\tau)}{{\rm d} \tau}\right|}\mathbf{b}(x_\ell(\tau) \cdot \mathcal{J}({\bf I}+\eta_\ell \mathcal{J})\mathbf{b}(x_\ell(\tau))\\
     =&\frac{\ell\eta_\ell|\nabla\overline{\phi}|}{|\ell|\left|\frac{{\rm d} x_\ell(\tau)}{{\rm d} \tau}\right|}|\mathbf{b}(x_\ell(\tau)|^2\leq 0,\quad \forall (\tau, \ell)\in \mathbb{O}_\delta,
     \end{split}
\end{equation}
where the last inequality is due to $\ell\eta_\ell\leq 0$.
Hence, by \eqref{liu-0728-7} we can calculate that
\begin{align}\label{liu0628-4}
    &-\Delta\overline{\phi}-A\mathbf{b}\cdot\nabla\overline{\phi}+\left[c(x)-\min\limits_{1\leq i\leq k}\{c(x_i)\}\right]\overline{\phi}\notag\\
   \geq &-\Delta\overline{\phi}+\left[c(x)-\min\limits_{1\leq i\leq k}\{c(x_i)\}\right]\overline{\phi}\\
   =&\frac{2}{\sqrt{\delta}}\left(\left(\partial_{x_1} \ell\right)^2+\left(\partial_{x_2} \ell\right)^2\right)+\frac{2r}{\sqrt{\delta}}\left[\partial^2_{x_1} \ell+\partial^2_{x_2} \ell\right]\notag\\
   &+\left[c(x)-\min\limits_{1\leq i\leq k}\{c(x_i)\}\right]\overline{\phi}.\notag
\end{align}
In view of  $(\partial_{x_1}r)^2+(\partial_{x_2}r)^2>0$ and $|\frac{2\ell}{\sqrt{\delta}}|\leq 2 \sqrt{\delta}$ for all $\ell\in[-\delta,\delta]$,  by \eqref{liu0628-4} we can choose $\delta$ small such that \eqref{liu0627-1} holds.

It remains to verify \eqref{liu0628-1}. Indeed, by \eqref{liu-0727-1}, for any $x\in\partial\mathbb{O}_\delta$, it follows that
\begin{align*}
    \nabla\overline{\phi}(x) \cdot \nu(x)&=-|\nabla\overline{\phi}(x)|=-2\sqrt{\delta}\sqrt{(\partial_{x_1}\ell)^2+(\partial_{x_2}\ell)^2},
\end{align*}
which implies \eqref{liu0628-1} holds.  Step 1 is thus complete.

\smallskip

{\bf Step 2}. We establish \eqref{liu0627}. Let $\{x_1,\cdots,x_n\}$ be the set of fixed points of \eqref{system} with $n\geq k$ and set $\tilde{\mathcal{O}}_\delta:=\{x\in\Omega:\exists 1\leq i\leq n, \,\, |x-x_i|<\delta \}$.
Given any $\epsilon>0$, applying same arguments as in Theorem \ref{liuprop1}, we can construct $\overline{\psi}\in C^2(\tilde{\mathcal{O}}_\delta)$ such that
\begin{equation}\label{liu-010}
-\Delta\overline{\psi}-A\mathbf{b}\cdot\nabla\overline{\psi}+c(x)\overline{\psi}\geq \left[\min\limits_{1\leq i\leq k}\{c(x_i)\}-\epsilon\right]\overline{\psi},\quad \forall x\in \tilde{\mathcal{O}}_\delta,
\end{equation}
provided that $\delta>0$ is taken small. Let the region $\mathbb{O}_\delta$ and function $\overline{\phi}>0$ be defined in Step 1.
Similar to Step 2 of Part 1 in the proof of  Theorem \ref{Limit-cycle}, based on  \eqref{liu0627-1},  \eqref{liu0628-1}, and \eqref{liu0728-1}, we can define $\overline{\varphi}\in C(\overline{\Omega})\cap C^2(\Omega\setminus(\partial \mathbb{O}_\delta\cup \partial \tilde{\mathcal{O}}_\delta))$ such that
 $$
    \overline{\varphi}(x)=\overline{\phi}(x), \,\,\,\,\forall x\in \mathbb{O}_\delta,\\
    \quad\text{and}\quad \overline{\varphi}(x)=\overline{\psi}(x), \,\,\,\,\forall x\in \tilde{\mathcal{O}}_\delta,
$$
and for sufficiently large $A$, there holds
\begin{equation}\label{liu0524-1-0726}
 \left\{\begin{array}{ll}
-\Delta\overline{\varphi}-A\mathbf{b}\cdot\nabla\overline{\varphi}+c(x)\overline{\varphi}\geq\left[ \min\limits_{1\leq i\leq k}\{c(x_i)\}-\epsilon\right] \overline{\varphi}&\mathrm{in} \,\, \Omega\setminus(\partial \mathbb{O}_\delta\cup \partial \tilde{\mathcal{O}}_\delta),\\
\medskip
(\nabla_-\overline{\varphi}(x)-\nabla_+\overline{\varphi}(x))\cdot \nu(x)> 0 & \mathrm{on}~\partial\mathbb{O}_\delta\cup \partial \tilde{\mathcal{O}}_\delta,\\
  \nabla\overline{\varphi}\cdot \mathbf{n}\geq 0 & \mathrm{on}~\partial\Omega.  
  \end{array}\right.
 \end{equation}
Then \eqref{liu0627} can be derived by the comparison principle. The proof is complete.
\end{proof}


\begin{theorem}[Semi-stable case]\label{Limit-cycle-3}
Suppose that the limit set of system \eqref{system} consists of a finite number of hyperbolic fixed points and a {\bf semi-stable limit cycle} $\mathcal{C}$ inside $\Omega$. 
Then
$$\lim_{A\to\infty}\lambda(A)= \min_{1\leq i\leq k}\{c(x_i)\},$$
where $\{x_1,\cdots,x_k\}$ denote the set of stable fixed points.
\end{theorem}
\begin{proof}
As in the proof of Theorem \ref{Limit-cycle-2}, it suffices to prove the lower bound estimate
\begin{equation}\label{liu0726-1}
  \liminf_{A\to\infty}\lambda(A)\geq \min\limits_{1\leq i\leq k}\{c(x_i)\},
\end{equation}
since the upper bound estimate can be established as in Step 2 of Theorem \ref{liuprop1}. To this end, we set $\mathcal{C}:=\{P(t): t\in [0,T)\}$ for some periodic solution $P(t)$ of system \eqref{system}. Under coordinate $x\mapsto(t,r)$ introduced by \eqref{liu-change},
 for each  $(0, \ell)$ with $\ell\in[-\delta,\delta]$,  we 
  can find $\eta_\ell\in\mathbb{R}$  such that the solution  denoted by $x_\ell(\tau)$ of \eqref{liu_04-0729} is a $T_\ell$-periodic.
Noting that $\mathcal{C}$ is semi-stable, we may assume without loss of generality that $\eta_\ell\leq 0$.
Then we define the region
$\mathbb{O}_\delta:=\{(\tau, \ell): \ell\in[-\delta,\delta] \text{ and }\tau\in[0,T_\ell)\}$ as in Step 1 of Theorem \ref{Limit-cycle-2}. Similar to \eqref{liu0729-3}, due to $\eta_\ell<0$ for $\ell\neq 0$ we verify that
 \begin{equation}\label{liu0728-2}
       {\bf b}(x)\cdot \nu(x)\Big |_{\ell=\delta}>0\quad\text{and}\quad  {\bf b}(x)\cdot \nu(x)\Big |_{\ell=-\delta}<0.
 \end{equation}

Define
\begin{equation}\label{liu0726-3}
   \overline{\phi}(\tau, \ell):=1-4\sqrt{\delta}\ell-\frac{\ell^2}{\sqrt{\delta}}, \qquad \forall (\tau, \ell)\in \mathbb{O}_{\delta}.
\end{equation}
Let $\nu(x)$ denote the outward unit normal vector on  $\partial \mathbb{O}_\delta$. By definitions we have
$$\nabla\overline{\phi} \cdot \nu\Big|_{\ell=\pm \delta}=\mp |\nabla\overline{\phi}|=\mp |\partial_\ell\overline{\phi}|\sqrt{(\partial_{x_1}\ell)^2+(\partial_{x_2}\ell)^2}.$$
Due to  $(\partial_{x_1}\ell)^2+(\partial_{x_2}\ell)^2>0$, 
this implies
\begin{equation}\label{liu0726-4}
\begin{split}
-2C\sqrt{\delta}<&\nabla\overline{\phi} \cdot \nu\Big |_{\ell=\delta} <0,\\
0<&\nabla\overline{\phi} \cdot \nu\Big |_{\ell=-\delta} <2 C\sqrt{\delta},
\end{split}
\end{equation}
for some $C>0$ independent of $\delta$.

We next verify \eqref{liu0627-1} for small  $\delta$.  As in \eqref{liu-0728-7}, in view $\eta_\ell\leq 0$, by \eqref{liu_04-0729} and \eqref{liu0726-3} we calculate that
  \begin{equation*}
\begin{split}
     \mathbf{b}(\tau, \ell)\cdot\nabla\overline{\phi}=&-\frac{|\nabla\overline{\phi}|}{\left|\frac{{\rm d} x_\ell(\tau)}{{\rm d} \tau}\right|}\mathbf{b}(x_\ell(\tau) \cdot \mathcal{J}\frac{{\rm d} x_\ell(\tau)}{{\rm d} \tau}\\
     =&\frac{\eta_\ell|\nabla\overline{\phi}|}{\left|\frac{{\rm d} x_\ell(\tau)}{{\rm d} \tau}\right|}|\mathbf{b}(x_\ell(\tau)|^2\leq 0,\quad \forall (\tau, \ell)\in \mathbb{O}_\delta.
     \end{split}
\end{equation*}
Hence,  similar to \eqref{liu0628-4} we can deduce that
\begin{align*}
    &-\Delta\overline{\phi}-A\mathbf{b}\cdot\nabla\overline{\phi}+\left[c(x)-\min\limits_{1\leq i\leq k}\{c(x_i)\}\right]\overline{\phi}\notag\\
   \geq &-\Delta\overline{\phi}+\left[c(x)-\min\limits_{1\leq i\leq k}\{c(x_i)\}\right]\overline{\phi}\\
   =&\frac{2}{\sqrt{\delta}}\left[\left(\partial_{x_1}\ell\right)^2+\left(\partial_{x_2}\ell\right)^2\right]+(4\sqrt{\delta}+\frac{2r}{\sqrt{\delta}})\left(\partial^2_{x_1}\ell+\partial^2_{x_2}\ell\right)\notag\\
   &+\left[c(x)-\min\limits_{1\leq i\leq k}\{c(x_i)\}\right]\overline{\phi}\geq 0, \quad \forall (\tau, \ell)\in\mathbb{O}_\delta,
\end{align*}
by choosing $\delta$ small if necessary. This implies  \eqref{liu0627-1}. 

Next, we can apply the same arguments as in the Step 2 of Theorem \ref{Limit-cycle-2} to define  $\overline{\varphi}\in C(\overline{\Omega})$ satisfying
 $\overline{\varphi}(x)=\overline{\phi}(x)$
 on $\mathbb{O}_\delta$, such that \eqref{liu0524-1-0726} holds for large $A$. 
 Then the lower bound estimate \eqref{liu0726-1} can be derived by the comparison principle. This completes the proof.
\end{proof}

Combining  Theorems \ref{Limit-cycle}-\ref{Limit-cycle-3}, we conclude this section by the following
\begin{theorem}\label{Limit-cycle-4}
Suppose that the limit set of system \eqref{system} consists of a finite number of limit cycles and hyperbolic fixed points in $\Omega$, whereas the stable fixed points are denoted by $\{x_1,\cdots,x_k\}$ and stable limit cycles are  denoted by $\{\mathcal{C}_1,\cdots,\mathcal{C}_m\}$ with  $\mathcal{C}_i:=\{P_i(t): t\in [0,T_i)\}$ for periodic solutions $P_i$ of \eqref{system} with period $T_i$. Then
$$\lim_{A\to\infty}\lambda(A)=\min\left\{\min_{1\leq i\leq m}\left\{\frac{1}{T_i}\int_0^{T_i} c(P_i(t)) {\rm d}t\right\}, \,\,\, \min_{1\leq i\leq k}\{c(x_i)\}\right\}.$$
\end{theorem}

 Theorem \ref{Limit-cycle-4} can be proved by combining the proofs of Theorems \ref{Limit-cycle}-\ref{Limit-cycle-3} and the details are omitted.




\section{\bf Case of saddle-note homoclinic orbits}\label{S5}

In this section, we assume the the limit set of system \eqref{system} includes homoclinic orbits. Under Hypothesis \ref{assum1}, we first consider a stable combination of two  homoclinic orbits  connected by a hyperbolic saddle.

\begin{theorem}\label{theorem20220722}
Suppose that the limit set of system \eqref{system} consists of two unstable fixed points and two  homoclinic orbits $\mathcal{C}_+$ and $\mathcal{C}_-$ connected by a hyperbolic saddle  $x_*$.  Assume the union of  homoclinic orbits $\mathcal{C}_+\cup\mathcal{C}_-$ is stable. Then
 $$\lim_{A\to\infty}\lambda(A)=c(x_*).$$
\end{theorem}

 \begin{figure}[http!!]
  \centering
\includegraphics[height=2.3in]{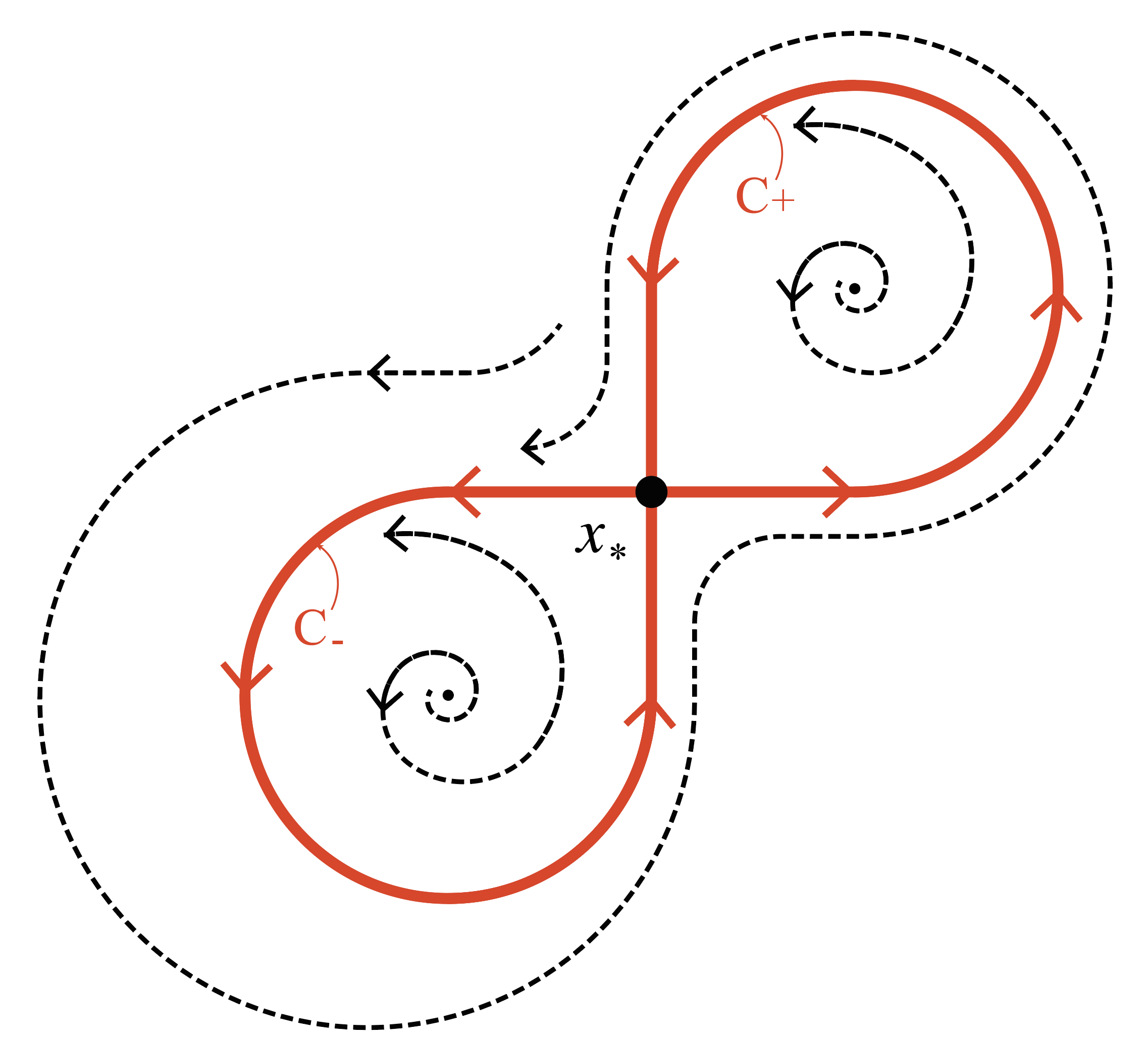}
  \caption{\small Illustration for the phase-portrait of \eqref{system} near homoclinic orbits as assumed in Theorem \ref{theorem20220722}.}\label{figure1-1-1}
  \end{figure}

\begin{proof}
We shall prove the lower bound estimate
\begin{equation}\label{liu20220722-2}
   \liminf_{A\to\infty}\lambda(A)\geq c(x_*),
\end{equation}
then the upper  bound estimate can be established by the similar arguments.

Given any $\epsilon>0$, we first choose $\delta>0$ small such that
\begin{equation}\label{liu20220722-1}
    |c(x)-c(x_*)|\leq \epsilon, \quad \forall x\in B_{4\delta}:=\{x\in\Omega: |x-x_*|<4\delta\}.
\end{equation}
Denote $\mathcal{C}:= \mathcal{C}_+\cup\mathcal{C}_-$ as the stable saddle-note homoclinic orbit.
Similar to Part 2 in the proof of Theorem \ref{Limit-cycle}, we can define a domain $O_\delta$ such that ${d_\mathcal{H}}(x,\mathcal{C})\leq \delta$, $\forall x\in O_\delta$, and ${\bf b}(x)\cdot \nu(x)<0$, $\forall x\in \partial O_\delta$, where $\nu(x)$ is the outward unit normal vector of $\partial O_\delta$. We define a curve $\Gamma\subset O_\delta$ such that $\Gamma\cap \mathcal{C}_+=\Gamma\cap \mathcal{C}_-=\{x_*\}$ and
$${d_\mathcal{H}}(x, \mathcal{C}_+)={d_\mathcal{H}}(x, \mathcal{C}_-), \qquad \forall x\in\Gamma.$$
Next, we will construct a positive super-solution $\overline{\varphi}\in C(\overline{\Omega})$ such that
\begin{equation}\label{liu20220722-3}
 \left\{\begin{array}{ll}
\medskip
-\Delta\overline{\varphi}-A\mathbf{b}\cdot\nabla\overline{\varphi}+c(x)\overline{\varphi}\geq\left(c(x_*)-4\epsilon\right) \overline{\varphi}&\mathrm{in} \,\, \Omega\setminus(\partial O_\delta\cup\Gamma),\\
\medskip
(\nabla_-\overline{\varphi}(x)-\nabla_+\overline{\varphi}(x))\cdot \nu(x)> 0 & \mathrm{on}~\partial O_\delta\cup \Gamma,\\
  \nabla\overline{\varphi}\cdot \mathbf{n}\geq 0 & \mathrm{on}~\partial\Omega,  
  \end{array}\right.
 \end{equation}
 provided that $A$ is sufficiently large, where for any $x\in\Gamma$, $\nu(x)$ denotes a unit vector being orthogonal to the curve $\Gamma$. Then \eqref{liu20220722-2} follows from the comparison principle and the arbitrariness of $\epsilon$. The construction can be completed by the following two steps.

 {\bf Step 1.} The construction of $\overline{\varphi}$ on the region $O_\delta$. Assume the domain $O_\delta$ is divided into two parts $O_{\delta+}$ and $O_{\delta-}$ by $\Gamma$ such that $\mathcal{C}_{\pm}\subset \overline{O}_{\delta\pm}$ and $\overline{O}_{\delta+}\cap\overline{O}_{\delta-}=\Gamma$. We perform the change of coordinate $x=X_\pm(y, z)$  in domains $O_{\delta\pm}$ with $y$ representing the arc length along the $\mathcal{C}_{\pm}$ and $z$ being the coordinate in the hyperplanes orthogonal to $\mathcal{C}_{\pm}$, such that $x_*=(0,0)$.
 This change of coordinate is $C^2$ smooth in $O_\delta\setminus B_{\rho}$  for any $\rho>0$, and
 $$x-\zeta(t(y))=z \frac{\mathcal{J}{\bf b}(y,0)}{|{\bf b}(y,0)|}, \quad \forall x\in O_\delta\setminus B_{\rho},$$
 where the rotation matrix $\mathcal{J}$ is defined below  \eqref{liu-change}, and $\zeta(t)$ and $t(y)$ satisfy
 $$ \frac{{\rm d}\zeta(t)}{{\rm d} t} =\mathbf{b}(\zeta(t))\quad\text{and}\quad\int_0^{t(y)}|{\bf b}(\zeta(s))|{\rm d}s=y.$$
 Thus, it holds
 \begin{equation}\label{liu-0728-22}
    \partial_z x=\frac{\mathcal{J}{\bf b}(y,0)}{|{\bf b}(y,0)|}.
 \end{equation}
Similar to Part 2 of Theorem \ref{Limit-cycle}, due to the stability of $\mathcal{C}$,  we can derive that for any $\ell\in(-\delta,\delta)$, there exists $\eta_\ell\in\mathbb{R}$ such that  the solution $x_\ell(\tau)$ of
\begin{equation}\label{liu_0730-1}
\left\{\aligned
  &\,\,\frac{{\rm d} x(\tau)}{{\rm d} \tau}= ({\bf I}+\eta_\ell \mathcal{J})\mathbf{b}(x(\tau)),\qquad \tau>0,\\
  &\,\, x(0)=X_+(4\delta,0)+\ell\mathcal{J}{\bf b}(X_+(4\delta,0))
 \endaligned \right.
\end{equation}
is periodic, $\ell \eta_\ell\geq 0$, and $|\eta_\ell|$ is  increasing in $ |\ell|$, where  $X_+(4\delta,0)\in\mathcal{C}_+$ satisfying $y=4\delta$. For any $\ell\in [-\delta,\delta]$, we let $y$ denote the arc length along the solution  $x_\ell(\tau)$, namely,
\begin{equation}\label{liu_0730-2}
   \int_0^{\tau(y)}|({\bf I}+\eta_\ell \mathcal{J})\mathbf{b}(x_\ell(s))|{\rm d}s=y.
 \end{equation}
In what follows, we will complete the proof under the coordinate $x\mapsto(y,\ell)$ such that $x=x_\ell(\tau(y))$, which is $C^2$ smooth in $O_\delta\setminus B_{\rho}$  for any $\rho>0$.

 We define
\begin{equation}\label{liu-0722-7}
\overline{\varphi}(y,\ell):=\exp\left\{\frac{\overline{\phi}(y,\ell)}{A}+\varepsilon\ell^2\right\},\qquad \forall (y,\ell)\in O_\delta\setminus B_\rho.
\end{equation}
Here function $\overline{\phi}$ and  constant $\rho\in (0, 4\delta)$ will be defined as follows.

{\bf (1)} The construction of $\overline{\varphi}(x)$ on the region $O_\delta\setminus B_{4\delta}$. We define $\overline{\phi}(y,\ell)\equiv\overline{\phi}(y)$ on $O_\delta\setminus B_{4\delta}$ such that
 \begin{equation}\label{liu-0722-5}
     -\frac{{\rm d}\overline{\phi}(y)}{{\rm d}y}\geq \frac{|c(x)-c(x_*)|}{|{\bf b}(y,\ell)|},\quad \forall x=(y,\ell)\in O_\delta\setminus B_{4\delta}.
 \end{equation}
 In virtue of \eqref{liu_0730-1} and  \eqref{liu_0730-2}, we have
 \begin{equation}\label{liu_0730-3}
     \partial_y x(y,\ell)=\frac{({\bf I}+\eta_\ell \mathcal{J})\mathbf{b}(y,\ell)}{|({\bf I}+\eta_\ell \mathcal{J})\mathbf{b}(y,\ell)|}, \quad \forall (y,\ell)\in O_\delta\setminus B_{4\delta}.
 \end{equation}
Similar to \eqref{liu0729-5}, by \eqref{liu_0730-3}  we calculate that
 \begin{equation}\label{liu-0728-23}
\begin{split}
    \mathbf{b}(y,\ell)\cdot\nabla\overline{\varphi}&= \mathbf{b}(y,\ell)\cdot\left[ \left(
  \begin{array}{cc}
  \smallskip
    \partial_y x_1 & \partial_y x_2 \\
   \partial_\ell x_1 & \partial_\ell x_2 \\
  \end{array}
\right)^{-1}\left(\begin{array}{c}
\smallskip
\partial_y \overline{\varphi}\\
\partial_z \overline{\varphi}
\end{array}
\right)\right]\\
=&\frac{|({\bf I}+\eta_\ell \mathcal{J})\mathbf{b}|}{\partial_\ell x\cdot(\mathcal{J}\mathbf{b}-\eta_\ell \mathbf{b} )} {\bf b}(y,\ell)\cdot\left[ \left(
  \begin{array}{cc}
  \smallskip
    \partial_\ell x_2 & -\partial_y x_2 \\
   -\partial_\ell x_1 &  \partial_y x_1\\
  \end{array}
\right)\left(\begin{array}{c}
\smallskip
\partial_y \overline{\varphi}\\
\partial_\ell \overline{\varphi}
\end{array}
\right)\right] \\
=&\frac{|({\bf I}+\eta_\ell \mathcal{J})\mathbf{b}|}{\partial_\ell x\cdot(\mathcal{J}\mathbf{b}-\eta_\ell \mathbf{b} )}\left(\partial_\ell x\cdot\mathcal{J}{\bf b}\partial_y \overline{\varphi}-\partial_y x\cdot\mathcal{J}{\bf b}\partial_\ell \overline{\varphi}\right)  \\
=&\frac{|({\bf I}+\eta_\ell \mathcal{J})\mathbf{b}|}{\partial_\ell x\cdot(\mathcal{J}\mathbf{b}-\eta_\ell \mathbf{b} )}\left(\partial_\ell x\cdot\mathcal{J}{\bf b}\partial_y \overline{\varphi}-\frac{2\varepsilon \ell\eta_\ell |{\bf b}|^2}{|({\bf I}+\eta_\ell \mathcal{J})\mathbf{b}|} \overline{\varphi}\right)  \\
\leq&\frac{|({\bf I}+\eta_\ell \mathcal{J})\mathbf{b}|}{\partial_\ell x\cdot(\mathcal{J}\mathbf{b}-\eta_\ell \mathbf{b} )}\partial_\ell x\cdot\mathcal{J}{\bf b}\partial_y \overline{\varphi},
\end{split}
\end{equation}
where the last inequality holds since $\ell\eta_\ell\geq 0$ and 
$$\lim_{\ell\to 0}\partial_\ell x\cdot(\mathcal{J}\mathbf{b}(y,\ell)-\eta_\ell \mathbf{b}(y,\ell))=|{\bf b}(y,0)|>4\delta, \quad \forall (y,\ell)\in O_\delta\setminus B_{4\delta},$$
due to \eqref{liu-0728-22}.
Hence,  by \eqref{liu-0722-7} and \eqref{liu-0722-5}  we deduce from \eqref{liu-0728-23} that
 \begin{align*}
    -\Delta\overline{\varphi}-A\mathbf{b}(x)\cdot\nabla\overline{\varphi}+c(x)\overline{\varphi}
\geq&\left[O(1/A)+O(\varepsilon)\right]\overline{\varphi}+(1+O(\delta))|c(x)-c(x_*)|\overline{\varphi}+c(x)\overline{\varphi}\\
\geq& (c(x_*)-\epsilon)\overline{\varphi}, \qquad \forall x\in O_\delta\setminus B_{4\delta},
\end{align*}
by choosing $\delta$ small and then $A$ large. Hence, the constructed $\overline{\varphi}(x)$ in \eqref{liu-0722-7} and \eqref{liu-0722-5} satisfies \eqref{liu20220722-3} in $O_\delta\setminus B_{4\delta}$.

  \begin{figure}[http!!]
  \centering
\includegraphics[height=2.6in]{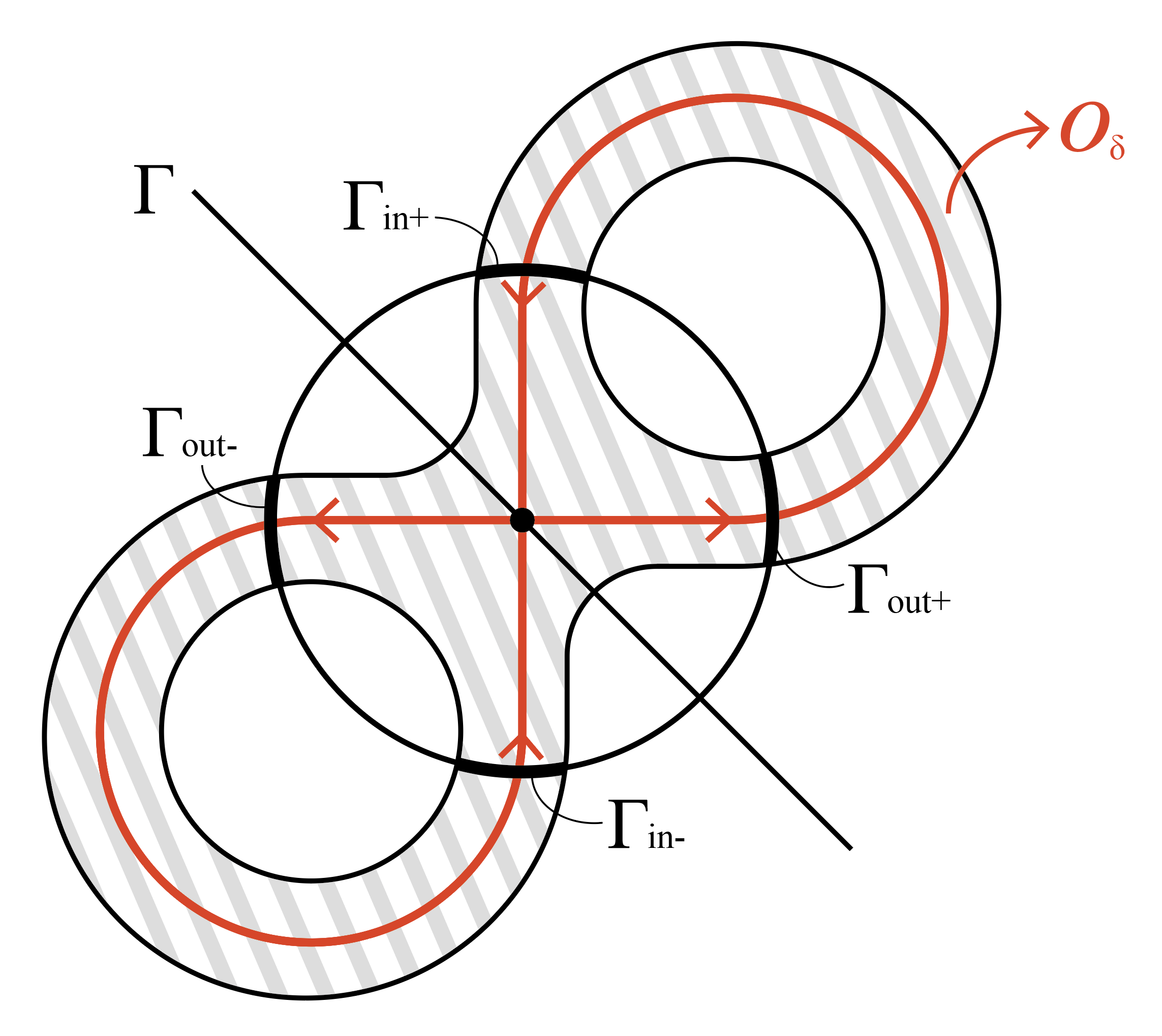}
  \caption{\small Illustration for some notations in the construction of $\overline{\varphi}$.}\label{figure1-1}
  \end{figure}

{\bf (2)} The construction of $\overline{\varphi}(x)$ on the region $O_{\delta+}\cap B_{4\delta}$. We denote $\partial B_{4\delta}\cap O_{\delta+}=\Gamma_{{\rm out}+}\cup \Gamma_{{\rm in}+}$ such that ${{\bf b}}(x)\cdot \nu(x)>0$ for all $x\in \Gamma_{{\rm out}+}$ and ${{\bf b}}(x)\cdot \nu(x)<0$ for all $x\in \Gamma_{{\rm in}+}$, see Fig. \ref{figure1-1} for some illustrations.  Let $L$ be the arc length of $\mathcal{C}_+$.
For any $x\in \Gamma_{{\rm in}+}$, 
it is easily seen that  $x=(y,\ell)$ satisfies $y\approx L-4\delta$, and similarly, each $x\in \Gamma_{{\rm out}+}$ can parameterized by $x=(y,\ell)$ with  $y\approx 4\delta$. By definitions \eqref{liu-0722-7} and \eqref{liu-0722-5}, it can be verified that there exists some ${\rm C}>0$ independent of $A$ such that for large $A$,
\begin{equation}\label{liu-0722-8}
    \max_{x\in \Gamma_{{\rm out}+}}\overline{\varphi}(y,\ell)-\min_{x\in \Gamma_{{\rm in}+}}\overline{\varphi}(y,\ell)\leq {\rm C}/A.
\end{equation}
Since $x_*$ is a hyperbolic saddle, we may assume without loss of generality that
\begin{equation}\label{liu0722-4}
    D{\bf b}(x_*)=\left(\begin{array}{cc}
    \lambda_1 & 0  \\
   0  & -\lambda_2
\end{array}\right)
\end{equation}
for some $\lambda_1,\lambda_2>0$. Then  the change of coordinate $x\mapsto(y,\ell)$ is $C^2$ smooth in $O_\delta\setminus B_{1/\sqrt{\lambda_2 A}}$, where $1/\sqrt{\lambda_2 A}<2\delta$ by letting $A$ be large if necessary.



For any $x=(y,\ell)\in O_{\delta+}\cap B_{4\delta}$ satisfying $y\leq L-1/\sqrt{\lambda_2 A}$, we shall define $\overline{\varphi}(x)$ as in \eqref{liu-0722-7}, in which the $C^2$ function $\overline{\phi}(y,\ell)$ is chosen such that  
\begin{equation}\label{liu-0722-9}
          \left|\partial_y\overline{\phi}(y,0)\right|\leq \frac{\epsilon}{\lambda_2(L-y)},\,\, \quad \forall y\in [L-4\delta,\,\,L-1/\sqrt{\lambda_2 A}],
\end{equation}
\begin{equation}\label{liu-0722-9-3}
          \left|\Delta\overline{\phi}(y,\ell)\right|\leq \frac{\epsilon}{\lambda_2(L-y)^2},\,\, \quad \forall y\in [L-4\delta,\,\,L-1/\sqrt{\lambda_2 A}],
\end{equation}
\begin{equation}\label{liu-0722-9-1}
 \overline{\varphi}\left(L-\frac{1}{\sqrt{\lambda_2 A}},\ell\right)\geq \max_{x\in \Gamma_{{\rm out}+}}\overline{\varphi}(x), \quad \forall \ell\in (-\delta,\delta).
\end{equation}
Here the inequality \eqref{liu-0722-9-1} is possible since we can choose $\overline{\phi}$ such that $\partial_y\overline{\phi}(y,0)= \frac{\epsilon}{\lambda_2(L-y)}$ for $y\in[L-2\delta, L-2/\sqrt{\lambda_2 A}] $ (which implies such $(y,\ell)$ belongs to the interior of $( O_{\delta+}\cap B_{4\delta})\setminus B_{1/\sqrt{\lambda_2 A}}$), and thus by \eqref{liu-0722-7}, we have
\begin{align*}
    \overline{\varphi}\left(L-\frac{1}{\sqrt{\lambda_2 A}},0\right)-\overline{\varphi}\left(L-2\delta,0\right)\geq& \int_{L-2\delta}^{L-\frac{2}{\sqrt{\lambda_2 A}}} \partial_y \overline{\varphi}(y,0) {\rm d}y=\frac{\epsilon \overline{\varphi}}{A\lambda_2 } \int_{L-2\delta}^{L-\frac{2}{\sqrt{\lambda_2 A}}} \frac{1}{L-y} {\rm d}y\\
    =&\frac{\epsilon \overline{\varphi}}{A\lambda_2 }\left[\ln (2\delta)+\ln(\sqrt{\lambda_2 A}/2)\right]>\frac{{\rm C}}{A},
\end{align*}
provided that $A$ is chosen large if necessary. By  continuity, we may derive from \eqref{liu-0722-7} that $ \overline{\varphi}\left(L-\frac{1}{\sqrt{\lambda_2 A}},\ell\right)-\overline{\varphi}\left(y,\ell\right)\geq C/A$ holds for all $(y,\ell)\in\Gamma_{{\rm in}+}$. 
This together with \eqref{liu-0722-8} yields that the choice of $\overline{\varphi}(x)$ satisfying \eqref{liu-0722-7},  \eqref{liu-0722-9}, and \eqref{liu-0722-9-1} is achievable.

For the  remaining $x=(y,\ell)\in O_{\delta+}\cap B_{4\delta}$ satisfying $y>L-1/\sqrt{\lambda_2 A}$, by \eqref{liu-0722-9}, we can construct $\overline{\varphi}\in C^2(O_{\delta+}\cap B_{4\delta})$ such that
\begin{equation}\label{liu-0722-9-2}
         {\bf b}\cdot\nabla\overline{\varphi}\leq 0\quad \text{and}\quad\left|\Delta\overline{\varphi}\right|\leq O(1/A)\overline{\varphi}, \quad \forall x\in O_{\delta+}\cap B_{4\delta} \text{ and } y>L-1/\sqrt{\lambda_2 A}.
\end{equation}

We next verify that the constructed $\overline{\varphi}$ above  satisfies \eqref{liu20220722-3} in $O_{\delta+}\cap B_{4\delta}$. First, analogue to \eqref{liu-0728-23}, by 
\eqref{liu-0722-7} 
we can calculate that
\begin{equation}\label{liu_0730-5}
\begin{split}
     A\mathbf{b}(y,\ell)\cdot\nabla\overline{\varphi}
    =&\frac{A|({\bf I}+\eta_\ell \mathcal{J})\mathbf{b}|}{\partial_\ell x\cdot(\mathcal{J}\mathbf{b}-\eta_\ell \mathbf{b} )}\left(\partial_\ell x\cdot\mathcal{J}{\bf b}\partial_y \overline{\varphi}-\frac{\eta_\ell |{\bf b}|^2}{|({\bf I}+\eta_\ell \mathcal{J})\mathbf{b}|} \partial_\ell \overline{\varphi}\right)\\
    =&\frac{A|({\bf I}+\eta_\ell \mathcal{J})\mathbf{b}|}{\partial_\ell x\cdot(\mathcal{J}\mathbf{b}-\eta_\ell \mathbf{b} )}\left(\partial_\ell x\cdot\mathcal{J}{\bf b}\partial_y \overline{\varphi}-\frac{\eta_\ell |{\bf b}|^2\partial_\ell \overline{\phi}}{A|({\bf I}+\eta_\ell \mathcal{J})\mathbf{b}|}  \overline{\varphi}-\frac{2\varepsilon \ell\eta_\ell |{\bf b}|^2}{|({\bf I}+\eta_\ell \mathcal{J})\mathbf{b}|} \overline{\varphi}\right)  \\
\leq &\left[\frac{|({\bf I}+\eta_\ell \mathcal{J})\mathbf{b}|}{\partial_\ell x\cdot(\mathcal{J}\mathbf{b}-\eta_\ell \mathbf{b} )}\partial_\ell x\cdot\mathcal{J}{\bf b}\partial_y \overline{\phi}-\frac{\eta_\ell |{\bf b}|^2\partial_\ell \overline{\phi}}{\partial_\ell x\cdot(\mathcal{J}\mathbf{b}-\eta_\ell \mathbf{b})} \right]\overline{\varphi}=: F(\ell).
\end{split}
\end{equation}
For any $x=(y,\ell)\in O_{\delta+}\cap B_{4\delta}$ and $y\in (L-4\delta,L-1/\sqrt{\lambda_2 A})$,  by choosing $\delta$ small if necessary, we find  ${\bf b}\approx (\lambda_1 x_1,-\lambda_2 x_2)$ and $x\approx (\ell, L-y)$.
 In view of $\eta_0=0$, by \eqref{liu-0722-9} we derive
\begin{align*}
   F(0)= |\mathbf{b}|\partial_y \overline{\phi}(y,0)\overline{\varphi}\leq (L-y)|\overline{\phi}(y,0)|\overline{\varphi}\leq \epsilon \overline{\varphi},
\end{align*}
and thus by \eqref{liu_0730-5} we conclude that $A{\bf b}\cdot \nabla\overline{\varphi}\leq F(\ell)\leq 2\epsilon \overline{\varphi}$ for small $\delta>0$.
Accordingly, for any $(y,\ell)\in O_{\delta+}\cap B_{4\delta}$ satisfying $y\in (L-4\delta, L-1/\sqrt{\lambda_2 A})$, by \eqref{liu-0722-7} and \eqref{liu-0722-9-3}  we have
\begin{align}\label{liu-0723-1}
    &-\Delta\overline{\varphi}-A\mathbf{b}(x)\cdot\nabla\overline{\varphi}+c(x)\overline{\varphi}\notag\\
\geq&\left[-\frac{1}{ A } \frac{\epsilon}{\lambda_2 (\ell-y)^2 }+O(\varepsilon)\right]\overline{\varphi}+(c(x_*)-2\epsilon)\overline{\varphi}\\
\geq&(c(x_*)-4\epsilon)\overline{\varphi}, \qquad \forall y\in (\ell-4\delta,\ell-1/\sqrt{\lambda_2 A}), \notag
\end{align}
for small $\delta$ and large $A$. Next, for $x=(y,\ell)\in O_{\delta+}\cap B_{4\delta}$ such that $y\geq \ell-1/\sqrt{\lambda_2 A}$, 
by \eqref{liu-0722-9-2} we can  choose $A$ large such that
\begin{align*}
    -\Delta\overline{\varphi}-A\mathbf{b}(x)\cdot\nabla\overline{\varphi}+c(x)\overline{\varphi}\geq &\left[c(x)+O(1/A)\right]\overline{\varphi}
\geq (c(x_*)-2\epsilon)\overline{\varphi}, \quad \forall y \geq \ell-1/\sqrt{\lambda_2 A}.
\end{align*}
 This together with  \eqref{liu-0723-1} verifies  \eqref{liu20220722-3} in $O_{\delta+}\cap B_{4\delta}$.

{\bf (3)} The super-solution $\overline{\varphi}(y)$ on the region $O_{\delta-}\cap B_{4\delta}$ can be  constructed symmetrically  as in  {\bf (2)}. Then we can verify the first equation in  \eqref{liu20220722-3} holds on  $O_{\delta-}\cap B_{4\delta}$. Furthermore, by \eqref{liu20220722-3} we can choose $A$ large and $\delta$ small  such that
\begin{equation}\label{liu-0723-3-2}
    \nabla_-\overline{\varphi}(x)\cdot \nu(x)>0 \quad\text{and} \quad \nabla_+\overline{\varphi}(x)\cdot \nu(x)<0, \quad \forall x\in\Gamma,
\end{equation}
which implies the boundary conditions on $\Gamma$ given in \eqref{liu20220722-3} is satisfied.

Combining with the constructions in {\bf (1)}-{\bf (3)}, we conclude the constructed $\overline{\varphi}$ in \eqref{liu-0722-7} satisfies \eqref{liu20220722-3} in $O_\delta$, which completes Step 1.


 {\bf Step 2.} The construction of $\overline{\varphi}$ on the region $\Omega\setminus O_\delta$.
Let $\Gamma\cap \partial O_\delta=\{x_+,x_-\}$. Notice from \eqref{liu20220722-3} that $\overline{\varphi}$ is increasing in $|x-x_*|$ for any $x\in\Gamma$. 
We deduce that  $\nabla_+\overline{\varphi}(x)\cdot \nu(x)\geq \eta \overline{\varphi}(x)$, $\forall x\in \partial O_\delta \cap B_\eta(\{x_+,x_-\})$, for some $\eta$ small depending upon $\varepsilon$, which is independent of $A$.
Furthermore, 
similar to \eqref{liu_0730-6}, we derive from \eqref{liu-0722-7} and \eqref{liu_0730-3}  that
\begin{equation}\label{liu-0723-3-1}
\begin{split}
&\nabla\overline{\varphi}(x) \cdot \nu(x)\\
=& \frac{|({\bf I}+\eta_\delta \mathcal{J})\mathbf{b}|}{\partial_\ell x\cdot(\mathcal{J}\mathbf{b}-\eta_\delta \mathbf{b})}\left[ \left(
  \begin{array}{cc}
  \smallskip
    \partial_\ell x_2 & -\partial_y x_2 \\
   -\partial_\ell x_1 &  \partial_y x_1\\
  \end{array}
\right)\left(\begin{array}{c}
\smallskip
\partial_y \overline{\varphi}\\
\partial_\ell \overline{\varphi}
\end{array}
\right)\right] \cdot \frac{\mathcal{J}\frac{{\rm d} x_{ \delta}(y)}{{\rm d} y}}{|\frac{{\rm d} x_{\delta}(y)}{{\rm d} y}|}\Bigg |_{\ell=\delta}\\
=&\frac{\mathcal{J}\partial_y x \cdot\mathcal{J}({\bf I}+\eta_{\delta}\mathcal{J}){\bf b}\partial_\ell \overline{\varphi}+O(1/A)\overline{\varphi}}{\partial_\ell x\cdot(\mathcal{J}\mathbf{b}-\eta_\delta \mathbf{b})}\Bigg |_{\ell=\delta}\\
=&\frac{2\varepsilon \delta|({\bf I}+\eta_{\delta}\mathcal{J}){\bf b}|+O(1/A)}{\partial_\ell x\cdot(\mathcal{J}\mathbf{b}-\eta_\delta \mathbf{b})}\overline{\varphi}\\
=&2\varepsilon\delta(1+O(\delta))|{\bf b}(y,0)|, \quad \forall x\in \partial O_\delta \setminus B_\eta(\{x_+,x_-\}).
\end{split}
\end{equation}
Thus we have verify that
\begin{equation*}
\nabla\overline{\varphi}(x) \cdot \nu(x)\geq \rho(\eta,\delta,\varepsilon)\overline{\varphi}, \quad \forall x\in \partial O_\delta \setminus \{x_+,x_-\},
\end{equation*}
with constant $\rho>0$ being independent of $A$.


 Since ${\bf b}(x)\cdot {\bf n}(x)<0$ on $\partial \Omega$  and ${\bf b}(x)\cdot \nu(x)<0$ on $\partial O_\delta$ as constructed,  the orbits of \eqref{system}
remain in $\Omega\setminus O_\delta$ only a finite time.
We then apply  \cite[Lemma 2.3]{DEF1974}  to show that there exists some $C^2$ function $G>0$ such that $G(x)=\overline{\varphi}(x)$ on $\partial B_{4\delta}\cap O_\delta$ and
\begin{equation}\label{liu-0723-4}
{\bf b}(x)\cdot\nabla G(x)<0, \quad \forall x\in \Omega\setminus O_\delta,
\end{equation}
\begin{equation}\label{liu-0723-5}
\nabla G(x)\cdot \nu(x)<\rho G(x), \quad \forall x\in \partial O_\delta, \quad \text{and}\quad \nabla G\cdot\mathbf{n}(x)\geq 0,\,\,\forall x\in\partial\Omega.
\end{equation}
We next define $\overline{\varphi}=G$ in $\Omega\setminus O_\delta$. Then by \eqref{liu-0723-4}, we can choose $A$ large such that the first equation in  \eqref{liu20220722-3} holds on  $\Omega\setminus O_{\delta}$. Finally, the boundary conditions in \eqref{liu20220722-3} can be verified by \eqref{liu-0723-3-2} and \eqref{liu-0723-5}.

Summarily, by Steps 1 and 2, we have constructed a positive super-solution $\overline{\varphi}\in C(\overline{\Omega})$ such that \eqref{liu20220722-3} holds. Then \eqref{liu20220722-2} can be deduced by the comparison principle. The proof of Theorem \ref{theorem20220722} is now complete.
\end{proof}

For the next result, we assume the union of two homoclinic orbits connected by a saddle is unstable, where the surrounding orbits are gradually farther from the homoclinic connections.

\begin{theorem}[Unstable case]\label{theorem20220804}
 Suppose that the limit set of system \eqref{system} consists of a finite number of hyperbolic fixed points and a unstable union composed of two homoclinic orbits 
 connected by a hyperbolic saddle.
 Then
$$\lim_{A\to\infty}\lambda(A)= \min_{1\leq i\leq k}\{c(x_i)\},$$
where $\{x_1,\cdots,x_k\}$ denotes the set of stable fixed points.
\end{theorem}
\begin{proof}
As in the proof of Theorem \ref{Limit-cycle-2}, the upper bound estimate $$\limsup\limits_{A\to\infty}\lambda(A)\leq \min\limits_{1\leq i\leq k}\{c(x_i)\}$$
can be established by  Step 2 of Theorem \ref{liuprop1}. It suffices to show the lower bound estimate
\begin{equation}\label{liu0804-1}
  \liminf_{A\to\infty}\lambda(A)\geq \min\limits_{1\leq i\leq k}\{c(x_i)\}.
\end{equation}

Denote by $\mathcal{C}:=\mathcal{C}_+\cup\mathcal{C}_-$ the two connected  homoclinic orbits $\mathcal{C}_+$ and $\mathcal{C}_-$.  We first define  region $O_\delta$ (as in Part 2 in the proof of Theorem \ref{Limit-cycle}) such that ${d_\mathcal{H}}(x,\mathcal{C})\leq \delta$, $\forall x\in O_\delta$, and ${\bf b}(x)\cdot \nu(x)<0$, $\forall x\in \partial O_\delta$, with $\nu(x)$ being the outward unit normal vector of $\partial O_\delta$. The proof of \eqref{liu0804-1} is divided into the following two steps.

{\bf Step 1}.  We shall  construct the positive super-solution $\overline{\varphi}\in C(O_{\delta})$  such that
\begin{equation}\label{liu0807-2}
-\Delta\overline{\varphi}-A\mathbf{b}\cdot\nabla\overline{\varphi}+c(x)\overline{\varphi}\geq \left[\min\limits_{1\leq i\leq k}\{c(x_i)\}\right]\overline{\varphi},\quad \forall x\in O_{\delta},
\end{equation}
for small $\delta$ and large $A$, and
\begin{equation}\label{liu0807-3}
\nabla\overline{\varphi}(x) \cdot \nu(x) \geq-3\sqrt{\delta} \overline{\varphi}(x), 
\quad \forall x\in \partial   O_{\delta}.
\end{equation}

Let $B_{4\delta}$ be defined by \eqref{liu20220722-1} and $\lambda_2>0$ be defined in \eqref{liu-0722-8}. Under the coordinate $x\mapsto(y,\ell)$ introduced in  the proof of Theorem \ref{theorem20220722}, we define
 \begin{equation}\label{liu0807-4}
\overline{\varphi}(y,\ell):=\exp\left\{\frac{\overline{\phi}(y)}{A}-\frac{\ell^2}{\sqrt{\delta}}\right\},\qquad \forall (y,\ell)\in O_\delta\setminus B_{1/\sqrt{\lambda_2 A}},
\end{equation}
and
\begin{equation}\label{liu0807-7}
\overline{\varphi}(x):=\exp\left\{-D\|x-x_*\|^2
+\frac{\rho(x)}{A}\right\},\qquad \forall x\in B_{1/\sqrt{\lambda_2 A}},
\end{equation}
where the functions $\overline{\phi}$ and $\rho$ are chosen such that $\varphi\in C^2(O_\delta)$, and the  constant  $D>0$ will be defined as follows.

{\bf (1)} The verification of \eqref{liu0807-2} on  $O_\delta\setminus B_{1/\sqrt{\lambda_2 A}}$.
By the unstablity of $\mathcal{C}$, it easily seen that $\ell\eta_\ell\leq0$, where $\eta_\ell$ is chosen such that \eqref{liu_0730-1} admits a periodic solution $x_\ell(\tau)$. Hence, as in \eqref{liu-0728-23}, we can calculate that
  \begin{equation}\label{liu0815-6}
\begin{split}
    \mathbf{b}(y,\ell)\cdot\nabla\overline{\varphi}
=&\frac{|({\bf I}+\eta_\ell \mathcal{J})\mathbf{b}|}{\partial_\ell x\cdot(\mathcal{J}\mathbf{b}-\eta_\ell \mathbf{b} )}\left(\partial_\ell x\cdot\mathcal{J}{\bf b}\partial_y \overline{\varphi}-\partial_y x\cdot\mathcal{J}{\bf b}\partial_\ell \overline{\varphi}\right)  \\
=&\frac{|({\bf I}+\eta_\ell \mathcal{J})\mathbf{b}|}{\partial_\ell x\cdot(\mathcal{J}\mathbf{b}-\eta_\ell \mathbf{b} )}\left(\partial_\ell x\cdot\mathcal{J}{\bf b}\partial_y \overline{\varphi}+\frac{2 \ell\eta_\ell |{\bf b}|^2}{\sqrt{\delta}|({\bf I}+\eta_\ell \mathcal{J})\mathbf{b}|} \overline{\varphi}\right)  \\
\leq&\frac{|({\bf I}+\eta_\ell \mathcal{J})\mathbf{b}|}{\partial_\ell x\cdot(\mathcal{J}\mathbf{b}-\eta_\ell \mathbf{b} )}\partial_\ell x\cdot\mathcal{J}{\bf b}\partial_y \overline{\varphi}\\
=&\frac{1}{A}\frac{|({\bf I}+\eta_\ell \mathcal{J})\mathbf{b}| \partial_\ell x\cdot\mathcal{J}{\bf b}}{\partial_\ell x\cdot(\mathcal{J}\mathbf{b}-\eta_\ell \mathbf{b} )}\frac{{\rm d}\overline{\phi}(y)}{{\rm d}y} \overline{\varphi}.
\end{split}
\end{equation}
In view of $\eta_\ell\to 0$ as $\ell\to 0$,  we can derive from \eqref{liu0807-4} and \eqref{liu0815-6}  that
 \begin{equation}\label{liu0815-7}
     \begin{split}
          &-\Delta\overline{\varphi}-A\mathbf{b}(x)\cdot\nabla\overline{\varphi}+c(x)\overline{\varphi}\\
\geq&\left[\frac{2}{\sqrt{\delta}}+O(1/A)+O(\sqrt{\delta})\right]\overline{\varphi}+(|{\bf b}(x)|+O(\delta))\left|\frac{{\rm d}\overline{\phi}(y)}{{\rm d}y}\right|\overline{\varphi}+c(x)\overline{\varphi}\\
\geq& \left[\min\limits_{1\leq i\leq k}\{c(x_i)\}\right]\overline{\varphi}, \qquad \forall x\in O_\delta\setminus B_{1/\sqrt{\lambda_2 A}},
     \end{split}
 \end{equation}
by choosing $\delta$ small and then $A$ large. Hence,  \eqref{liu0807-2} holds on $O_\delta\setminus B_{1/\sqrt{\lambda_2 A}}$. 

 {\bf (2)} For any $x\in B_{1/\sqrt{\lambda_2 A}}$,
by \eqref{liu0722-4} and \eqref{liu0807-7},  direct calculation yields
\begin{align*}
\mathbf{b}\cdot\nabla\overline{\varphi}(x)=&(x-x_*)^{\rm T}\left(\begin{array}{cc}
    \lambda_1 & 0  \\
   0  & -\lambda_2
\end{array}\right)\nabla\overline{\varphi}(x)+O(\|x-x_*\|^3)\\
\leq& 
D\lambda_2\|x-x_*\|^2+O(\|x-x_*\|^3)\\
\leq&D/A+O(1/A^{3/2}), \qquad\forall x\in B_{1/\sqrt{\lambda_2 A}},
\end{align*}
whence we can calculate that for any $x\in B_{1/\sqrt{\lambda_2 A}}$,
\begin{align*}
&-\Delta\overline{\varphi}-A\mathbf{b}\cdot\nabla\overline{\varphi}+(c(x)-\min\limits_{1\leq i\leq k}\{c(x_i)\})\overline{\varphi}\\
\geq&\left[3D+O(1/\sqrt{A})+(c(x)-\min\limits_{1\leq i\leq k}\{c(x_i)\})\right]\overline{\varphi}.
\end{align*}
Therefore, we may choose $D$ large such that \eqref{liu0807-2} holds in $B_{1/\sqrt{\lambda_2 A}}$.

It remains to verify \eqref{liu0807-3}. By \eqref{liu-0723-3-1}, we calculate from \eqref{liu0807-4} that
\begin{equation}\label{liu0821-1}
\begin{split}
\nabla\overline{\varphi}(x) \cdot \nu(x)
=&\frac{\mathcal{J}\partial_y x \cdot\mathcal{J}({\bf I}+\eta_{\delta}\mathcal{J}){\bf b}\partial_\ell \overline{\varphi}+O(1/A)\overline{\varphi}}{\partial_\ell x\cdot(\mathcal{J}\mathbf{b}-\eta_\delta \mathbf{b})}\Bigg |_{\ell=\delta}\\
=&-\frac{2\sqrt{\delta}|({\bf I}+\eta_{\delta}\mathcal{J}){\bf b}|+O(1/A)}{\partial_\ell x\cdot(\mathcal{J}\mathbf{b}-\eta_\delta \mathbf{b})}\overline{\varphi}\\
\geq&-2\sqrt{\delta}(1+O(\delta))\overline{\varphi}, \quad \forall x\in \partial O_\delta,
\end{split}
\end{equation}
and thus \eqref{liu0807-3} follows. Step 1 is complete.

{\bf Step 2}. We prove \eqref{liu0804-1}. Let $\{x_1,\cdots,x_n\}$ be the set of fixed points of \eqref{system} with $n\geq k$. Recall  $\tilde{O}_\delta:=\{x\in\Omega:\exists 1\leq i\leq n, \,\, |x-x_i|<\delta \}$ defined in Step 2 of  Theorem \ref{Limit-cycle-2}.
Given any $\epsilon>0$, we can apply the arguments in the proof of Theorem \ref{liuprop1} to construct $\overline{\varphi}\in C^2(\tilde{O}_\delta)$ such that \eqref{liu-010} holds for 
small $\delta>0$.
Notice that there are no limit points of  system \eqref{system}  in $\Omega\setminus(O_\delta\cup\tilde{O}_\delta)$. In view of  ${\bf b}(x)\cdot \nu(x)>0$ on $\partial O_\delta$, by \eqref{liu0807-3} we can define $\overline{\varphi}\in C^2(\Omega\setminus(O_\delta\cup\tilde{O}_\delta))$ such that ${\bf b}(x)\cdot\nabla \overline{\varphi}(x)<0$ on $\Omega\setminus(O_\delta\cup\tilde{O}_\delta)$ and
$$(\nabla_-\overline{\varphi}(x)-\nabla_+\overline{\varphi}(x))\cdot \nu(x)> 0, \qquad \forall x\in\partial O_\delta\cup \partial \tilde{O}_\delta.$$
Hence, \eqref{liu-010} holds on  $\Omega\setminus(O_\delta\cup\tilde{O}_\delta)$ by choosing $A$ large. Combining with \eqref{liu0807-2}, we have  constructed a positive super-solution $\overline{\varphi}\in C(\overline{\Omega})$ such that
\begin{equation*}
 \left\{\begin{array}{ll}
\medskip
-\Delta\overline{\varphi}-A\mathbf{b}\cdot\nabla\overline{\varphi}+c(x)\overline{\varphi}\geq\left[\min\limits_{1\leq i\leq k}\{c(x_i)\}\right] \overline{\varphi}&\mathrm{in} \,\, \Omega\setminus(\partial O_\delta\cup\partial \tilde{O}_\delta),\\
\medskip
(\nabla_-\overline{\varphi}(x)-\nabla_+\overline{\varphi}(x))\cdot \nu(x)> 0 & \mathrm{on}~\partial O_\delta\cup\partial \tilde{O}_\delta,\\
  \nabla\overline{\varphi}\cdot \mathbf{n}\geq 0 & \mathrm{on}~\partial\Omega,  
  \end{array}\right.
 \end{equation*}
for sufficiently large $A$. Then \eqref{liu0804-1} can be derived from the comparison principle. This concludes the proof. 
\end{proof}

\begin{theorem}\label{theorem20220809}
Suppose that the limit set of system \eqref{system} consists of a finite number of  fixed points and a semi-stable union of two homoclinic orbits $\mathcal{C}_+\cup\mathcal{C}_-$ connected by a hyperbolic saddle. Assume $\mathcal{C}_+\cup\mathcal{C}_-$ is stable from the outside, and  $\mathcal{C}_+$ or/and $\mathcal{C}_-$ are unstable from the inside, then
 $$\lim_{A\to\infty}\lambda(A)=\min_{1\leq i\leq k}\{c(x_i)\},$$
where $\{x_1,\cdots,x_k\}$ denotes the set of stable fixed points.
\end{theorem}
\begin{proof}
Similar to Theorems \ref{Limit-cycle-2} and \ref{theorem20220804},  it suffices to show the lower bound estimate
\begin{equation}\label{liu0815-1}
  \liminf_{A\to\infty}\lambda(A)\geq \min\limits_{1\leq i\leq k}\{c(x_i)\},
\end{equation}
as the upper bound estimate
can be established by  Step 2 of Theorem \ref{liuprop1}. For clarity, we divide the proof into the following two parts.

\medskip

{\bf Part 1}. We first establish \eqref{liu0815-1} in the case that both homoclinic orbits  are  unstable from the inside.   
Set  $\mathcal{C}:= \mathcal{C}_+\cup\mathcal{C}_-$ and let $x_*$ be the hyperbolic saddle connecting $\mathcal{C}_+$ and $\mathcal{C}_-$. For any $\delta>0$, we choose $O_\delta$ as in Part 2  of Theorem \ref{Limit-cycle} such that ${d_\mathcal{H}}(x,\mathcal{C})\leq \delta$, $\forall x\in O_\delta$, and ${\bf b}(x)\cdot \nu(x)<0$, $\forall x\in \partial O_\delta$, with $\nu(x)$ being the outward unit normal vector on $\partial O_\delta$. Denote $\{x_1,\cdots,x_n\}$ as the set of fixed points of \eqref{system} with $n\geq k$. Recall  $\tilde{O}_\delta:=\{x\in\Omega:\exists 1\leq i\leq n, \,\, |x-x_i|<\delta \}$.
Given any $\epsilon>0$, inspired from the proof of Theorem \ref{theorem20220722},
we next  construct a positive super-solution $\overline{\varphi}\in C(\overline{\Omega})\cap C^2(\Omega\setminus(\partial O_\delta\cup\partial \tilde{O}_\delta))$ such that
\begin{equation}\label{liu0815-2}
 \left\{\begin{array}{ll}
-\Delta\overline{\varphi}-A\mathbf{b}\cdot\nabla\overline{\varphi}+c(x)\overline{\varphi}\geq\left[\min\limits_{1\leq i\leq k}\{c(x_i)\}-\epsilon\right] \overline{\varphi}&\mathrm{in} \,\, \Omega\setminus(\partial O_\delta\cup\partial \tilde{O}_\delta),\\
\medskip
(\nabla_-\overline{\varphi}(x)-\nabla_+\overline{\varphi}(x))\cdot \nu(x)> 0 & \mathrm{on}~\partial O_\delta\cup\partial \tilde{O}_\delta,\\
  \nabla\overline{\varphi}\cdot \mathbf{n}\geq 0 & \mathrm{on}~\partial\Omega  
  \end{array}\right.
 \end{equation}
for sufficiently large $A$ and small $\delta$. Then \eqref{liu0815-1} can be derived by the comparison principle.

First,
applying the same arguments as in Theorem \ref{liuprop1}, we can construct $\overline{\varphi}\in C^2(\tilde{O}_\delta)$ such that the first inequality in \eqref{liu0815-2} holds on $\tilde{O}_\delta$.  The construction of $\overline{\varphi}$ on $\Omega\setminus \tilde{O}_\delta$ can be preformed in the following two steps.

{\bf Step 1}. The construction of $\overline{\varphi}$ on $O_\delta$.
We first recall some notations. 
Let $B_{4\delta}$ be defined by \eqref{liu20220722-1}.
We define the curve $\Gamma\subset\Omega$ such that $\Gamma\cap \mathcal{C}_+=\Gamma\cap \mathcal{C}_-=\{x_*\}$ and
${d_\mathcal{H}}(x, \mathcal{C}_+)={d_\mathcal{H}}(x, \mathcal{C}_-)$,  $\forall x\in\Gamma$. Then the region $O_\delta$ is assumed to be divided into two parts $O_{\delta+}$ and $O_{\delta-}$ by $\Gamma$ such that $\mathcal{C}_{\pm}\subset \overline{O}_{\delta\pm}$ and $\overline{O}_{\delta+}\cap\overline{O}_{\delta-}=\Gamma$. We only complete the construction on $O_{\delta+}$, as that on $O_{\delta-}$ is
symmetrical.

As in the proof of Theorem \ref{theorem20220722},
we introduce the coordinate $x\mapsto(y,\ell)$ such that $x=x_\ell(\tau(y))$, where $x_\ell$ denotes the periodic solution of \eqref{liu_0730-1} by choosing suitable $\eta_\ell$, and  $y$ denotes the arc length along the solution  $x_\ell(\tau)$ as in \eqref{liu_0730-2}. Since the  homoclinic orbit $\mathcal{C}_+$ is semi-stable, it turns out that  $ \eta_\ell\leq 0$ and $|\eta_\ell|\searrow  0$ as $\ell\to 0$. 
The ideas of the construction is similar to Step 1 of Theorem \ref{theorem20220804}. Indeed,  we define \begin{equation}\label{liu0815-3}
\overline{\varphi}(y,\ell):=\exp\left\{\frac{\overline{\phi}(y)}{A}-4\sqrt{\delta}\ell-\frac{\ell^2}{\sqrt{\delta}}\right\},\qquad \forall (y,\ell)\in O_\delta\setminus B_{1/\sqrt{\lambda_2 A}},
\end{equation}
and
\begin{equation}\label{liu0815-4}
\overline{\varphi}(x):=\exp\left\{
-(x-x_*)^{\rm T}\left(\begin{array}{cc}
   D & {\rm K}/\lambda_1  \\
   {\rm K}/\lambda_2  & D
\end{array}\right)(x-x_*)
+\frac{\rho(x)}{A}\right\},\qquad \forall x\in B_{1/\sqrt{\lambda_2 A}}.
\end{equation}
Here $\lambda_1, \lambda_2>0$ are defined in \eqref{liu0722-4},  functions $\overline{\phi}$ and $\rho$ are chosen such that $\overline{\varphi}\in C^2(O_\delta)$, and the constants  $D,{\rm K}>0$ will be defined as follows.

{\bf (1)} For any $(y,\ell)\in O_\delta\setminus B_{1/\sqrt{\lambda_2 A}}$,  similar to \eqref{liu0815-6}, in light of  $\eta_\ell\leq 0$,
by \eqref{liu_0730-3} and \eqref{liu0815-3}  we can calculate that
  \begin{equation*}
\begin{split}
    \mathbf{b}(y,\ell)\cdot\nabla\overline{\varphi}
=&\frac{|({\bf I}+\eta_\ell \mathcal{J})\mathbf{b}|}{\partial_\ell x\cdot(\mathcal{J}\mathbf{b}-\eta_\ell \mathbf{b} )}\left(\partial_\ell x\cdot\mathcal{J}{\bf b}\partial_y \overline{\varphi}+\frac{2 (2\sqrt{\delta}+\ell/\sqrt{\delta})\eta_\ell |{\bf b}|^2}{\sqrt{\delta}|({\bf I}+\eta_\ell \mathcal{J})\mathbf{b}|} \overline{\varphi}\right)  \\
\leq&\frac{1}{A}\frac{|({\bf I}+\eta_\ell \mathcal{J})\mathbf{b}| \partial_\ell x\cdot\mathcal{J}{\bf b}}{\partial_\ell x\cdot(\mathcal{J}\mathbf{b}-\eta_\ell \mathbf{b} )}\frac{{\rm d}\overline{\phi}(y)}{{\rm d}y} \overline{\varphi},
\end{split}
\end{equation*}
whence by \eqref{liu0815-7} we can choose $A$ large and $\delta$ small such that the first inequality in \eqref{liu0815-2} holds on $O_\delta\setminus B_{1/\sqrt{\lambda_2 A}}$.

{\bf (2)} For any $x\in B_{1/\sqrt{\lambda_2 A}}$, we derive from \eqref{liu0722-4} and \eqref{liu0815-4} that
\begin{align*}
&\mathbf{b}(x)\cdot\nabla\overline{\varphi}(x)\\
=&(x-x_*)^{\rm T}\left(\begin{array}{cc}
    \lambda_1 & 0  \\
   0  & -\lambda_2
\end{array}\right)\nabla\overline{\varphi}(x)+O(\|x-x_*\|^3)\\
=& (x-x_*)^{\rm T}\left[\left(\begin{array}{cc}
    -\lambda_1 D & -\ {\rm K} \\
     {\rm K}  & \lambda_2 D
\end{array}\right)(x-x_*)+\left(\begin{array}{cc}
    \lambda_1 & 0  \\
   0  & -\lambda_2
\end{array}\right)\frac{\nabla\rho(x)}{A}\right]\overline{\varphi}+O(\|x-x_*\|^3)\\
\leq&D/A+O(1/A^{3/2}), \qquad\forall x\in B_{1/\sqrt{\lambda_2 A}}.
\end{align*}
Hence by \eqref{liu0815-4} it holds that 
\begin{align*}
&-\Delta\overline{\varphi}-A\mathbf{b}\cdot\nabla\overline{\varphi}+(c(x)-\min\limits_{1\leq i\leq k}\{c(x_i)\})\overline{\varphi}\\
=&\left[3D-\frac{\Delta \rho(x)}{A}+O(1/\sqrt{A})+(c(x)-\min\limits_{1\leq i\leq k}\{c(x_i)\})\right]\overline{\varphi}\geq 0
\end{align*}
by choosing $A$ large, so that  \eqref{liu0815-2} holds on $B_{1/\sqrt{\lambda_2 A}}$.

Furthermore, 
similar to \eqref{liu_0730-6} and \eqref{liu0807-3}, by  \eqref{liu0815-3} we can deduce  that for any $x\in\partial O_\delta$,
\begin{equation}\label{liu0815-8}
-7\sqrt{\delta} \overline{\varphi}(x)\leq \nabla\overline{\varphi}(x) \cdot \nu(x)\Big |_{\ell=\delta}\leq 0 \quad\text{and}\quad 0\leq \nabla\overline{\varphi}(x) \cdot \nu(x)\Big |_{\ell=-\delta}\leq 3\sqrt{\delta} \overline{\varphi}(x).
\end{equation}
 Step 1 is complete.

{\bf Step 2}. The construction of $\overline{\varphi}$ on $\Omega\setminus(O_\delta\cup\tilde{O}_\delta)$. By our assumption, there are no limit points of  system \eqref{system}  in $\Omega\setminus(O_\delta\cup\tilde{O}_\delta)$. Noting that  ${\bf b}(x)\cdot \nu(x)>0$ on $\partial O_\delta$, according to \eqref{liu0815-8} we can define positive $\overline{\varphi}\in C^2(\Omega\setminus(O_\delta\cup\tilde{O}_\delta))$ such that ${\bf b}\cdot\nabla \overline{\varphi}<0$ on $\Omega\setminus(O_\delta\cup\tilde{O}_\delta)$ and
$$(\nabla_+\overline{\varphi}(x)-\nabla_-\overline{\varphi}(x))\cdot \nu(x)> 0, \qquad \forall x\in\partial O_\delta\cup \partial \tilde{O}_\delta,$$
which means the boundary conditions in \eqref{liu0815-2} hold. The existence of such $\overline{\varphi}$ can be established by the same arguments in  \cite[Lemma 2.3]{DEF1974}. Moreover, in light of  ${\bf b}\cdot\nabla \overline{\varphi}<0$ on $\Omega\setminus(O_\delta\cup\tilde{O}_\delta)$, we can choose $A$ large so that the first inequality in  \eqref{liu0815-2} also holds.

Combining with Step 1,  we have  constructed a super-solution $\overline{\varphi}\in C(\overline{\Omega})\cap C^2(\Omega\setminus(\partial O_\delta\cup\partial \tilde{O}_\delta))$ satisfying \eqref{liu0815-2}, whence \eqref{liu0815-1} holds.

\medskip

{\bf Part 2}. We next prove \eqref{liu0815-1} in the case that   the homoclinic orbit $\mathcal{C}_+$ is unstable from the inside while $\mathcal{C}_-$ is stable from the inside, see Fig. \ref{figure1-1liu-1}. Recall that $O_\delta=O_{\delta+}\cup O_{\delta-}$ with $\mathcal{C}_+\subset O_{\delta+}$ and $\mathcal{C}_-\subset O_{\delta-}$. Denote $L>0$ as the arc length of the homoclinic orbit $\mathcal{C}_-$.

 \begin{figure}[http!!]
  \centering
\includegraphics[height=2.3in]{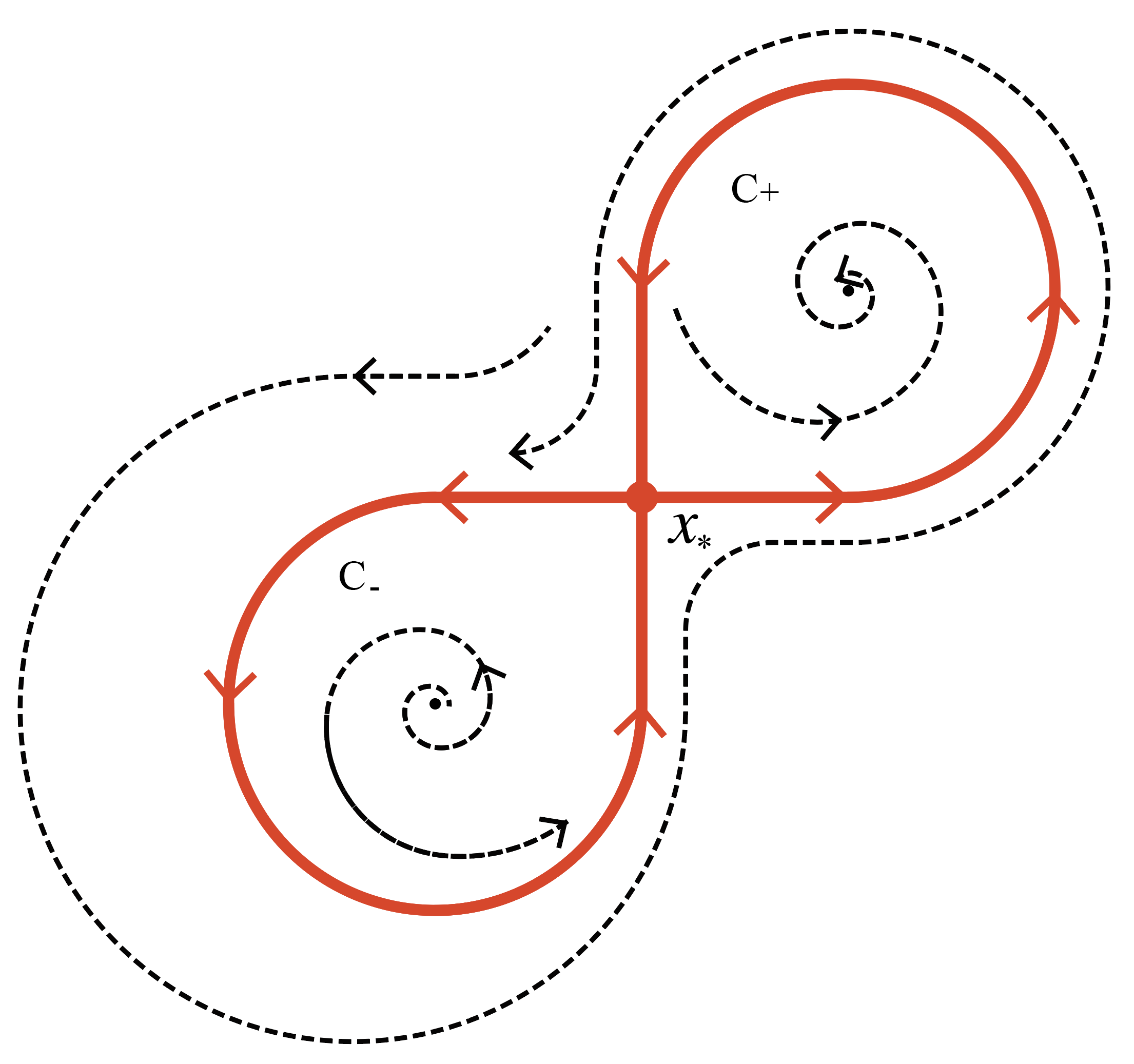}
  \caption{\small Illustration for the phase-portrait of \eqref{system} in the case that the union of two homoclinic orbits $\mathcal{C}_+\cup \mathcal{C}_+$ is semi-stable.}\label{figure1-1liu-1}
  \end{figure}

{\bf Step 1}. Before proving \eqref{liu0815-1}, we first consider the one-dimensional eigenvalue problem
\begin{equation}\label{liu0816-1}
 \left\{\begin{array}{ll}
\medskip
-\phi_{yy}-A|\mathbf{b}(y,0)|\phi_y+c(y,0)\phi=\lambda\phi, \quad  x\in(0,L),\\
\phi_y(0)+\min\{\alpha,\epsilon\}\phi(0)=0, \quad \phi_y(L)+\alpha \phi(L)=0.
  \end{array}\right.
 \end{equation}
Given any $A,\alpha>0$, let $\underline{\lambda}(A,\alpha)$ be the principal eigenvalue of problem \eqref{liu0816-1} and $\phi_{A,\alpha}>0$ be the corresponding eigenfunction. We shall claim that for each  $A$ large, there exists some $\alpha_A>0$ such that $\phi_{A,\alpha_A}(0)=\phi_{A,\alpha_A}(L)$ and it holds $\underline{\lambda}(A,\alpha_A)\to\infty$ as $A\to\infty$.

On the one hand, since $|{\bf b}(y,0)|>0$ in $(0,L)$,  if $\alpha=0$ in \eqref{liu0816-1}, it is proved by Lemmas 3.1 and 3.2 in \cite{CL2008} that  $\phi_{A,0}(y)\rightharpoonup \delta(L)$  weakly in $L^2((0,L))$, where $\delta(L)$ is the Dirac measure centered at $x=L$. This implies that $\phi_{A,0}(0)<\phi_{A,0}(L)$ for large $A$. On the other hand, given any $A>0$, by virtue of the variational characterization (see e.g. \cite{CL2008,PZ2018})
\begin{equation*}
    \begin{split}
        \underline{\lambda}(A,\alpha)=\min_{\int_0^Le^{A\int_0^y|{\bf b}|(z,0)|{\rm d}z}\phi^2{\rm d}y=1}\Big\{&-\min\{\alpha,\epsilon\}\phi^2(0)+\alpha\phi^2(L)\\
        &+\int_0^Le^{A\int_0^y|{\bf b}|(z,0)|{\rm d}z}(|\phi_y|^2+c(y,0)\phi^2){\rm d}y\Big\},
    \end{split}
\end{equation*}
we can derive 
$\underline{\lambda}(A,\alpha)\to \underline{\lambda}_\infty(A)$ as $\alpha\to\infty$ with $\underline{\lambda}_\infty(A)$ being the principal eigenvalue of 
 \begin{equation*}
 \left\{\begin{array}{ll}
\medskip
-\phi_{yy}-A|\mathbf{b}(y,0)|\phi_y+c(y,0)\phi=\underline{\lambda}_\infty\phi, \quad  x\in(0,L),\\
\phi_y(0)+\epsilon\phi(0)=0, \quad \phi(L)=0,
  \end{array}\right.
 \end{equation*}
for which  the principal eigenfunction  is denoted by $\phi_{A,\infty}$. By the elliptic regularity theory, it can be proved that $\phi_{A,\alpha}\to\phi_{A,\infty}$ (up to some  multiplication) uniformly in $[0,L]$ as $\alpha\to\infty$. In view $\phi_{A,\infty}(L)=0$ and $\phi_{A,\infty}(0)>0$, 
we have $\phi_{A,\alpha}(0)>\phi_{A,\alpha}(L)$ for large $\alpha$, which together with $\phi_{A,0}(0)<\phi_{A,0}(L)$ proves the existence of $\alpha_A$.

It remains to show $\underline{\lambda}(A,\alpha_A)\to\infty$ as $A\to\infty$. Let $\alpha=\alpha_A$ in \eqref{liu0816-1}. Multiply both sides of \eqref{liu0816-1} by $\phi_{A,\alpha_A}$ and integrate over $[0,L]$, then we find
\begin{equation}\label{liu0818-2}
\begin{split}
    \underline{\lambda}(A,\alpha_A)&\geq \min_{\int_0^Le^{A\int_0^y|{\bf b}|(z,0)|{\rm d}z}\phi^2{\rm d}y=1  \atop \phi(0)=\phi(L)}\int_0^Le^{A\int_0^y|{\bf b}|(z,0)|{\rm d}z}(|\phi_y|^2+c(y,0)\phi^2){\rm d}y\\
    &=\min_{\int_0^L w^2{\rm d}y=1  \atop w(0)=e^{A\int_0^L|{\bf b}|(z,0)|{\rm d}z}w(L)}\int_0^L\Big[\Big(w_y-\frac{A}{2}|{\bf b}(y,0)|w\Big)^2+c(y,0)w^2\Big]{\rm d}y.
    \end{split}
\end{equation}
Suppose on the contrary that $\underline{\lambda}(A_n,\alpha_{A_n})$ is uniformly bounded for some sequence $\{A_n\}$ such that $A_n\to\infty$ as $n\to\infty$. It follows from \eqref{liu0818-2} that there exists sequence $\{w_n\}_{n=1}^\infty$ satisfying $\int_0^L w_n^2{\rm d}y=1$ and $ w_n(0)=e^{A_n\int_0^L|{\bf b}|(z,0)|{\rm d}z}w_n(L)$
such that
$$\int_0^L\Big[\Big((w_n)_y-\frac{A_n}{2}|{\bf b}(y,0)|w_n\Big)^2+c(y,0)w_n^2\Big]{\rm d}y\leq C,\quad \forall n\geq 1$$
for some $C$ independent of $n$.
We may assume $w_n^2\rightharpoonup \mu$ weakly for some Radon measure $\mu$, and thus $\mu([0,L])=1$.  In view of $|{\bf b}(y,0)|>0$ in $(0,L)$, Lemmas 3.1 and 3.2 in \cite{CL2008} imply $\mu([0,L))=0$. Noting that $ w_n(0)=e^{A_n\int_0^L|{\bf b}|(z,0)|{\rm d}z}w_n(L)$, there holds $\mu(\{L\})=0$, which contradicting  $\mu([0,L])=1$. Hence, we have $\underline{\lambda}(A,\alpha_A)\to\infty$ as $A\to\infty$. Step 1 is complete.

{\bf Step 2}. Since $x_*$ is a hyperbolic saddle, we  assume the Jacobi matrix $D{\bf b}(x_*)$ is given by \eqref{liu0722-4}
with some $\lambda_1,\lambda_2>0$, see Fig. \ref{figure1-1liu-1}. 
Let $O_\delta$ and $\tilde{O}_\delta$ be defined as in Part 1. Given any $\epsilon>0$,
we shall  construct a positive super-solution $\overline{\varphi}\in C(\overline{\Omega})$ such that
\begin{equation}\label{liu0819-1}
 \left\{\begin{array}{ll}
-\Delta\overline{\varphi}-A\mathbf{b}\cdot\nabla\overline{\varphi}+c(x)\overline{\varphi}\geq\Big[\min\limits_{1\leq i\leq k}\{c(x_i)\}-\epsilon\Big] \overline{\varphi}&\mathrm{in} \,\, \Omega\setminus(\partial O_\delta\cup\partial \tilde{O}_\delta\cup\{x_*\}),\\
\medskip
(\nabla_-\overline{\varphi}(x)-\nabla_+\overline{\varphi}(x))\cdot \nu(x)> 0 & \mathrm{on}~\partial O_\delta\cup\partial \tilde{O}_\delta,\\
  \nabla\overline{\varphi}\cdot \mathbf{n}\geq 0 & \mathrm{on}~\partial\Omega  
  \end{array}\right.
 \end{equation}
for sufficiently large $A$ and small $\delta$, and moreover
\begin{equation}\label{liu0819-2}
\begin{split}
    &\lim_{t\to 0^+} \frac{1}{t}\Big[\overline{\varphi}(x_*+t(\cos\alpha,\sin\alpha))-\overline{\varphi}(x_*)\Big]< 0, \quad\forall \alpha\in(-\tfrac{\pi}{4}, \tfrac{5\pi}{4}),\\
    &\lim_{t\to 0^+} \frac{1}{t}\Big[\overline{\varphi}(x_*+t(\cos\alpha,\sin\alpha))-\overline{\varphi}(x_*)\Big]>0, \quad\forall \alpha\in( \tfrac{5\pi}{4}, \tfrac{7\pi}{4}).
    \end{split}
\end{equation}
Since  the homoclinic orbit $\mathcal{C}_+$ is unstable from the inside, the super-solution $\overline{\varphi}$ on $O_{\delta+}\cup\tilde{O}_\delta$  can be constructed as in Part 1 such that \eqref{liu0819-1} holds and \eqref{liu0819-2} is satisfied  for $\alpha\in\left(-\frac{\pi}{4}, \frac{3\pi}{4}\right)$.  To complete the construction of $\overline{\varphi}$ on the remaining regions, we will utilize the coordinate $x\mapsto(y,\ell)$ in $O_{\delta-}$ which satisfies $x=x_\ell(\tau(y))$ (where $x_\ell$ denotes the periodic solution of \eqref{liu_0730-1}). By the stability of $\mathcal{C}_-$, it turns out that $\ell \eta_\ell\geq 0$, $\forall \ell\in(-\delta,\delta)$, and $|\eta_\ell|\searrow  0$ as $\ell\to 0$.

{\bf (1)} The construction of $\overline{\varphi}$ on $O_{\frac{1}{A}-}:=\{x\in O_{\delta-}:{d_\mathcal{H}}(x,\mathcal{C}_-)\leq 1/A\}$.  For each $A>0$, let $\phi_{A,\alpha_A}>0$ be the  principal  eigenfunction of problem \eqref{liu0816-1} corresponding to principal eigenvalue $\underline{\lambda}(A,\alpha_A)$ such that $\phi_{A,\alpha_A}(0)=\phi_{A,\alpha_A}(L)$ and $\min\limits_{x\in[0,L]}\phi_{A,\alpha_A}(x)=1$. Under the coordinate $x\mapsto(y,\ell)$, we define $\overline{\varphi}\in C^2(O_{\frac{1}{A}-}\setminus\{x_*\})$ such that
\begin{equation}\label{liu0819-3-1}
\overline{\varphi}(y,0)=\phi_{A,\alpha_A}(y)\quad\text{and}\quad \partial_\ell\overline{\varphi}(y,0)=0, \quad \forall y\in[0,L],
\end{equation}
\begin{equation}\label{liu0819-3}
    {\bf b}(y,\ell)\cdot\nabla\overline{\varphi}\leq 0,\quad \forall (y,\ell)\in B_{4\delta}\cap (O_{\delta-}\setminus O_{\frac{1}{A}-}),
\end{equation}
and \eqref{liu0819-2} holds for $\alpha\in\left( \frac{3\pi}{4},\frac{7\pi}{4}\right)$.   We now illustrate that such construction is possible.  Due to $\overline{\varphi}(y,0)=\phi_{A,\alpha_A}(y)$, as in \eqref{liu0819-3-1}, the boundary conditions in \eqref{liu0816-1}  imply $\partial_y\overline{\varphi}(y,0)< 0$ for $y\in(L-4\delta,L)$ and $\lim\limits_{y\to 0^+}\partial_y\overline{\varphi}(y,0)=0$, so that we can extend $\overline{\varphi}$ to $ B_{4\delta}\cap  O_{\frac{1}{A}-}$
 such that ${\bf b}(y,\ell)\cdot\nabla\overline{\varphi}\leq 0$, $\forall (y,\ell)\in B_{4\delta}\cap \partial O_{\frac{1}{A}-}$, and thus \eqref{liu0819-3} is possible.
Moreover, noting from that $\partial_y \overline{\varphi}(L,0)<0$, we can verify that \eqref{liu0819-2} holds for $\alpha=\frac{3\pi}{2}$, so that the construction of $\overline{\varphi}$ such that \eqref{liu0819-2} holds for $\alpha\in\left( \frac{3\pi}{4},\frac{7\pi}{4}\right)$ is possible.

Next, we shall verify that the constructed $\overline{\varphi}$ satisfies \eqref{liu0819-1}. Indeed, in view of $\eta_0=0$, by \eqref{liu0815-6} we find
  \begin{equation*}
\begin{split}
    \mathbf{b}(y,0)\cdot\nabla\overline{\varphi}(y,0)
=&\frac{|({\bf I}+\eta_\ell \mathcal{J})\mathbf{b}|}{\partial_\ell x\cdot(\mathcal{J}\mathbf{b}-\eta_\ell \mathbf{b} )}\left(\partial_\ell x\cdot\mathcal{J}{\bf b}\partial_y \overline{\varphi}-\partial_y x\cdot\mathcal{J}{\bf b}\partial_\ell \overline{\varphi}\right) \Big|_{\ell=0}=|{\bf b}(y,0)|\partial_y \overline{\varphi}(y,0).
\end{split}
\end{equation*}
For any $(y,\ell)\in O_{\frac{1}{A}-}$, by \eqref{liu0819-3-1} and \eqref{liu0816-1}   we have
\begin{equation}\label{liu0819-5}
    \begin{split}
        &-\Delta\overline{\varphi}-A\mathbf{b}\cdot\nabla\overline{\varphi}+c(x)\overline{\varphi}\\
=&\left[-\Delta\overline{\varphi}+c(y,\ell)\overline{\varphi}\right]\Big |_{\ell=0}+O(1/A)-A\mathbf{b}(y,0)\cdot\nabla\overline{\varphi}(y,0)+O(1)\\
=&-\partial_{yy}\overline{\varphi}(y,0)-A|\mathbf{b}(y,0)|\partial_y\overline{\varphi}(y,0)+c(y,0)\overline{\varphi}(y,0)+O(1/A)+O(1)\\
=&\underline{\lambda}(A,\alpha_A)\overline{\varphi}(y,0)+O(1/A)+O(1)\\
\geq &(\underline{\lambda}(A,\alpha_A)/2)\overline{\varphi}(y,\ell), \quad \forall (y,\ell)\in O_{\frac{1}{A}-},
    \end{split}
\end{equation}
provided that $A$ is chosen large. 
Here the last inequality is due to $\underline{\lambda}(A,\alpha_A)\to \infty$ as $A\to\infty$ and $\overline{\varphi}(y,0)\geq 1$ by construction. Hence, \eqref{liu0819-1} follows immediately.

{\bf (2)} The construction of $\overline{\varphi}$ on $O_{\delta-}\cap B_{4\delta}$.  Based on the construction on $O_{\frac{1}{A}-}$ in Case {\bf (1)}, by \eqref{liu0819-3} and \eqref{liu0819-5}, we shall  extend $\overline{\varphi}$ to $O_{\delta-}\cap B_{4\delta}$ such that $\overline{\varphi}\in C^2((O_{\delta-}\cap B_{4\delta})\setminus\{x_*\})$ and  \eqref{liu0819-1} holds  by choosing $A$ large if necessary.

{\bf (3)} The construction of $\overline{\varphi}$ on $\Omega_{\rm remd}:=\Omega\setminus(\tilde{O}_\delta\cup O_{\frac{1}{A}-}\cup (B_\delta\cap O_{\delta-}))$. Since there are no limit points of  system \eqref{system}  in $\Omega_{\rm remd}$, as in Step 2 of Part 1,  we can define  $\overline{\varphi}\in C^2(\Omega_{\rm remd})$ such that ${\bf b}\cdot\nabla \overline{\varphi}<0$ on $\Omega_{\rm remd}$. Noting that $|{\bf b}(x)|\geq \varepsilon(\delta)$ for some constant $\varepsilon(\delta)>0$ independence of $A$, we can choose $A$ large such that  \eqref{liu0819-1} holds on $\Omega_{\rm remd}$. Step 2 is now complete.

{\bf Step 3}. We finally prove \eqref{liu0815-1}. Let $\varphi\in C^2(\Omega)$ be the principal eigenfunction of \eqref{1} corresponding to $\lambda(A)$. Up to some  multiplication we may assume that $\varphi\leq\overline{\varphi}$ in $\bar{\Omega}$ and there exists 
some $x_0\in\overline{\Omega}$ such that  $\varphi(x_0)=\overline{\varphi}(x_0)$  and $x_0\in\partial\{x\in\overline{\Omega}: \varphi(x)=\overline{\varphi}(x)\}$, where $\overline{\varphi}$ is defined in Step 2 satisfying \eqref{liu0819-1} and \eqref{liu0819-2}. It suffices to prove that for large $A$,
\begin{equation}\label{liu0820-1}
    \lambda(A)\geq \min\limits_{1\leq i\leq k}\{c(x_i)\}-\epsilon.
\end{equation}
Set $W:=\overline{\varphi}-\varphi\geq 0$, which satisfies \eqref{liu0819-1} and $W(x_0)=0$. In view of $\nabla W\cdot \mathbf{n}\geq 0$ on $\partial\Omega$, by the Hopf's boundary lemma for elliptic equations we can deduce $x_0\not\in\partial\Omega$. We next claim $x_0\in \Omega\setminus(\{x_*\}\cup\partial O_\delta\cup\partial \tilde{O}_\delta)$. Indeed, note from the boundary condition in \eqref{liu0819-1} that
$$\nabla_-W(x_0)\cdot \nu(x_0)>\nabla_+W(x_0)\cdot \nu(x_0),$$
which together with the definition of the notation $\nabla\overline{\varphi}_\pm$  in \eqref{liu0621-3} yields $x_0\not\in \partial O_\delta\cup\partial \tilde{O}_\delta$ immediately. If $x_0=x_*$, then $W(x_*)=0$. By  \eqref{liu0819-2} we find that for any $\alpha\in\left(-\frac{\pi}{4}, \frac{5\pi}{4}\right)$,
the function
$\overline{\varphi}(x_*+t(\cos\alpha,\sin\alpha))$ acting on $[0,\varepsilon]$ for some small $\varepsilon>0$, attains its maximum at $t=0$. Since $\varphi\leq\overline{\varphi}$ and  $\varphi(x_*)=\overline{\varphi}(x_*)$, it holds that $\varphi(x_*+t(\cos\alpha,\sin\alpha))$ also attains its maximum on $[0,\varepsilon]$  at $t=0$, namely, $\nabla\varphi(x_*)\cdot (\cos\alpha,\sin\alpha)\leq 0$ for all $ \alpha\in\left(-\frac{\pi}{4}, \frac{5\pi}{4}\right)$, from which we choose $\alpha=0$ and $\alpha=\pi$ particularly to derive $\partial_{x_1}\varphi(x_*)=\nabla\varphi(x_*)\cdot (1,0)=0$. Hence, letting $\alpha=0$ in \eqref{liu0819-2} gives
$$\lim_{t\to 0^+} \frac{W(x_*+t(1,0))}{t}=\lim_{t\to 0^+} \frac{1}{t}\Big[\overline{\varphi}(x_*+t(1,0))-\overline{\varphi}(x_*)\Big]<0.$$
 Since $W(x_*)=0$, this contradicts $W \geq 0$. Therefore,
there holds $x_0\in \Omega\setminus(\{x_*\}\cup\partial O_\delta\cup\partial \tilde{O}_\delta)$.

Now, we can choose some $\varepsilon>0$ small such that $B_\varepsilon(x_0)\subset \Omega\setminus(\{x_*\}\cup\partial O_\delta\cup\partial \tilde{O}_\delta)$. If \eqref{liu0820-1} fails, then by definition we have $W\in C^2(B_\varepsilon(x_0))$ and
$$-\Delta W-A\mathbf{b}\cdot\nabla W+\left[c(x)-\min\limits_{1\leq i\leq k}\{c(x_i)\}+\epsilon\right] W\geq 0\quad \text{ on }\,\, B_\varepsilon(x_0). $$
Due to $W(x_0)=0$, we may apply the classical strong maximum principle for elliptic equations  to arrive at $W\equiv 0$ on $B_\varepsilon(x_0)$. This is a contradiction since $x_0\in\partial\{x\in\overline{\Omega}: W(x)=0\}$ by our assumption. Therefore, \eqref{liu0820-1} holds and \eqref{liu0815-1} is proved. The proof is now complete.
\end{proof}

 Our next result concerns the situation where the limit set of system \eqref{system} contains an isolated homoclinic orbits with respect to a hyperbolic saddle. See Fig. \ref{figureliuintroduction3}  for an illustration.

\begin{theorem}\label{theorem20220827}
Suppose that the limit set of system \eqref{system} consists of a finite number of  fixed points and  a homoclinic orbit with respect to some hyperbolic saddle. Then
$$\lim_{A\to\infty}\lambda(A)=\min_{1\leq i\leq k}\{c(x_i)\},$$
where $\{x_1,\cdots,x_k\}$ denotes the set of stable fixed points.
\end{theorem}
\begin{proof}
By the same arguments as in  Step 2 of Theorem \ref{liuprop1}, we can establish the upper bound estimate \eqref{liu0621-2}.
 It remains to show the lower bound estimate
\begin{equation}\label{liu0830-1}
  \liminf_{A\to\infty}\lambda(A)\geq \min\limits_{1\leq i\leq k}\{c(x_i)\}.
\end{equation}

To this end, we first introduce some notations. Let $\mathcal{C}$ be the homoclinic loop and $x_*$ be the associated hyperbolic saddle.  Set  $O_\delta:=\{x\in\Omega: {d_\mathcal{H}}(x,\mathcal{C})<\delta\}$ as the $\delta$-neighborhood of the  homoclinic loop $\mathcal{C}$. Since the saddle point $x_*$ is hyperbolic, we  assume the Jacobi matrix $D{\bf b}(x_*)$ is given \eqref{liu0722-4}
with $\lambda_1,\lambda_2>0$. We then define two  heteroclinic orbits $\xi_1$ and $\xi_2$ connecting $x_*$ to a stable fixed point and a  unstable fixed point, respectively. 
The domain $O_\delta$ can be divided  into two parts $O_{\delta+}$ and $O_{\delta-}$ by $\xi_1\cup\xi_2$ such that $\mathcal{C}\subset \overline{O}_{\delta+}$ and $O_{\delta-}\subset B_{2\delta}(x_*)$. Moreover, define $\{x_{k+1},\cdots,x_n\}\subset\Omega$ with $n>k$ as the set of the unstable fixed points  of \eqref{system}. Set  $\tilde{O}_\delta:=\{x\in\Omega:\exists 1\leq i\leq n, \,\, |x-x_i|<\delta \}$.

To prove \eqref{liu0830-1}, for any given  $\epsilon>0$, we shall construct the positive super-solution $\overline{\varphi}\in C(\overline{\Omega})$ 
such that for the sufficiently large $A$, there holds
\begin{equation}\label{liu0830-2}
 \left\{\begin{array}{ll}
-\Delta\overline{\varphi}-A\mathbf{b}\cdot\nabla\overline{\varphi}+c(x)\overline{\varphi}\geq\Big[\min\limits_{1\leq i\leq k}\{c(x_i)\}-\epsilon\Big] \overline{\varphi}\\
\qquad\qquad\qquad \qquad\qquad \qquad\qquad\qquad\quad\,\,\,\mathrm{in} \,\, \Omega\setminus(\partial O_\delta\cup\partial \tilde{O}_\delta\cup(O_\delta\cap \Gamma)),\\
\medskip
(\nabla_-\overline{\varphi}(x)-\nabla_+\overline{\varphi}(x))\cdot \nu(x)> 0 \qquad\qquad \mathrm{on}~\partial O_\delta\cup\partial \tilde{O}_\delta\cup (O_\delta\cap \Gamma),\\
  \nabla\overline{\varphi}\cdot \mathbf{n}\geq 0  \qquad\qquad\qquad\qquad\qquad\qquad\quad \mathrm{on}~\partial\Omega,  
  \end{array}\right.
 \end{equation}
and \eqref{liu0819-2} holds for $\alpha\in\left( \frac{3\pi}{4},\frac{7\pi}{4}\right)$, where $\nu(x)$ denotes the outward unit normal vector on $\partial O_\delta\cup\partial \tilde{O}_\delta$. Then \eqref{liu0830-1}  can follow from the same arguments as in Step 3 in Part 2 of Theorem \ref{theorem20220809}.

 The construction of such $\overline{\varphi}$ on the region $\Omega\setminus O_\delta$ can be completed by same arguments  in the proof of Theorem \ref{liuprop1}, as the limit set of system \eqref{system} on $\Omega\setminus O_\delta$  includes only  the stable and unstable fixed points. It remains to construct the super-solution $\overline{\varphi}\in  C^2(O_\delta\setminus(\{x_*\}))$ such that \eqref{liu0830-2} holds on $O_\delta$.

We first define
\begin{equation}\label{liu0831-2-1}
    \overline{\varphi}(x):=1+(x-x_*)^{\rm T}\left(\begin{array}{cc}
    -\lambda_1/\sqrt{\delta} & 0  \\
   0  & \lambda_2
\end{array}\right)(x-x_*), \qquad \forall x\in O_{\delta-}.
\end{equation}
Analogue to {\rm (iii)} in Step 1 of Theorem \ref{liuprop1}, we can verify that function $\overline{\varphi}$ satisfies \eqref{liu0830-2}  on $O_{\delta-}$ by choosing $\delta$ small if necessary. The construction of $\overline{\varphi}$ on $O_{\delta+}$ is similar to  Step 2 of Theorem \ref{theorem20220809}, and we only give a sketch here. As in the proof of Theorem \ref{theorem20220722}, we introduce the coordinate $x\mapsto(y,\ell)$ on $O_{\delta+}$ such that $x=x_\ell(\tau(y))$, where $x_\ell$ denotes the periodic solution of \eqref{liu_0730-1}, and  $y$ denotes the arc length along the solution  $x_\ell(\tau)$ as in \eqref{liu_0730-2}. Denote by $L>0$ the arc length of $\mathcal{C}$. For any $x\in O_{\frac{1}{A}+}:=O_{\frac{1}{A}}\cap O_{\delta+}$, we define $\overline{\varphi}\in C^2(O_{\frac{1}{A}+})$ such that
\begin{equation*}
\begin{array}{ll}
\medskip
   \overline{\varphi}(y,0)=\phi_{A,\alpha_A}(y),  & \forall y\in[0,L],  \\
    {\bf b}(y,\ell)\cdot\nabla\overline{\varphi}\leq 0, & \forall (y,\ell)\in B_{2\delta}(x_*)\cap (O_{\delta+}\setminus O_{\frac{1}{A}+}),
\end{array}
\end{equation*}
and \eqref{liu0819-2} holds for $\alpha\in\left( \frac{3\pi}{4},\frac{7\pi}{4}\right)$, where $\phi_{A,\alpha_A}(y)>0$ is the principal eigenfunction of problem \eqref{liu0816-1} satisfying $\phi_{A,\alpha_A}(0)=\phi_{A,\alpha_A}(L)$, and moreover $\phi_{A,\alpha_A}(0)=1$.  This together with \eqref{liu0831-2-1} implies  the continuity of function $\overline{\varphi}$ at $x_*$.  The realizability of the choice of such $\overline{\varphi}$ is illustrated below \eqref{liu0819-3}. By \eqref{liu0819-5} we can verify \eqref{liu0830-2} holds.
Then we extent the construction on $x\in O_{\frac{1}{A}+}$ to $B_{2\delta}(x_*)\cap O_{\delta+}$ such that \eqref{liu0830-2} holds as in case {\bf (2)} in Step 2 of Theorem \ref{theorem20220809}.  Finally,
for the construction of $\overline{\varphi}$ on $O_{\delta+}\setminus O_{\frac{1}{A}+}$, noting  that there are no limit points of  system \eqref{system}  in $O_{\delta+}\setminus O_{\frac{1}{A}+}$, we can define  $\overline{\varphi}\in C^2(O_{\delta+}\setminus O_{\frac{1}{A}+})$ such that ${\bf b}\cdot\nabla \overline{\varphi}<0$ on $O_{\delta+}\setminus O_{\frac{1}{A}+}$. Then \eqref{liu0830-2} follows by letting  $A$ large. The proof is  complete.
\end{proof}

\section{\bf Case of closed orbits family 
}\label{S4-1}

In this section, we are concerned with the case where the limit set of system  \eqref{system} may include the family of closed orbits. The results turn out to be analogue to those in \cite{BHN2005}, although the divergence-free assumption for the vector field ${\bf b}$ is not necessary in our setting.
\begin{proposition}\label{theorem0820}
Suppose that  
the orbits of \eqref{system} in $\Omega$ are composed by the periodic orbits, the union of two homoclinic orbits connected by a hyperbolic saddle, 
and the center points.  Assume \eqref{system} admits a finite number of fixed points in $\Omega$, and any saddle point is  hyperbolic. 
Then the following assertions hold.
\begin{itemize}
    \item[\rm(i)] 
     Let $\lambda(A)$ be the principal eigenvalue of \eqref{1}.
Then there holds
\begin{equation*}
    \lim_{A\rightarrow\infty}\lambda(A)=\inf_{\phi\in \mathbf{I}}\left\{\frac{
    \int_\Omega (|\nabla\phi|^2+c(x)\phi^2)\mathrm{d}x}{\int_\Omega \phi^2\mathrm{d}x}\right\},
\end{equation*}
whereas $\mathbf{I}=\{\phi\in H^1(\Omega): \mathbf{b}\cdot\nabla\phi =0 \text{ a.e. in }\Omega\}$.
\medskip
\item[{\rm (ii)}] Let $\lambda_\infty(A)$ be the principal eigenvalue of the problem
\begin{equation*}
 \left\{\begin{array}{ll}
\medskip
-\Delta\varphi-A\mathbf{b}\cdot\nabla\varphi+c(x)\varphi=\lambda_\infty(A)\varphi &\mathrm{in} \,\, \Omega,\\
  \varphi=0 & \mathrm{on}~\partial\Omega.  
  \end{array}\right.
 \end{equation*}
Then we have
\begin{equation*}
    \lim_{A\rightarrow\infty}\lambda_\infty(A)=\inf_{\phi\in \mathbf{I}_0}\left\{\frac{\int_\Omega (|\nabla\phi|^2+c(x)\phi^2)\mathrm{d}x}{\int_\Omega \phi^2\mathrm{d}x}\right\},
\end{equation*}
whereas $\mathbf{I}_0=\{\phi\in H^1_0(\Omega): \mathbf{b}\cdot\nabla\phi =0 \text{ a.e. in }\Omega\}$.
\end{itemize}
\end{proposition}


\begin{proof}
We only prove the assertion {\rm (i)}, as the assertion {\rm (ii)} can be proved by the rather similar arguments. We first introduce some notations. Denote by $\{x_{\rm c}^i\}_{i\in \mathbb{C}}$  the set of  center points of \eqref{system}, and define $\{x_{\rm s}^i\}_{i\in \mathbb{S}}$ as the set of saddles connecting two
homoclinic orbits, the combination of which is denoted by $\mathcal{C}_i$. For each $\delta>0$,  set
\begin{equation*}
  \begin{array}{l}
  \smallskip
 O^i_\delta:=\{\mathcal{C}\subset\Omega: \mathcal{C} \text{ is a periodic orbit and }{d_\mathcal{H}}(\mathcal{C}, \mathcal{C}_i)<\delta\}, \quad i\in \mathbb{S},\\
    \tilde{O}^i_{\delta}:=\{\mathcal{C}\subset\Omega: \mathcal{C} \text{ is a periodic orbit and }{d_\mathcal{H}}(\mathcal{C}, x_{\rm c}^i)<\delta\},\quad i\in \mathbb{C}.
   \end{array}
\end{equation*}
Set $\Omega_\delta:=\Omega\setminus[(\cup_{i\in \mathbb{S}}O^i_\delta)\cup (\cup_{i\in \mathbb{C}}\tilde{O}^i_\delta)]$ and let $\nu(x)$ be  the outward unit normal vector on $\partial  \Omega_\delta$, which implies $\nu(x)={\bf n}(x)$ on $\partial \Omega$. 
 For clarify,  we  divide the proof into the following two steps.

\medskip
{\bf Step 1. Lower bound estimate}:
We establish the lower bound estimate
\begin{equation}\label{liu20220714-1}
 \liminf_{A\to\infty}\lambda(A)\geq \inf_{\phi\in \mathbf{I}}\left\{ \frac{
 \int_\Omega (|\nabla\phi|^2+c(x)\phi^2)\mathrm{d}x}{\int_\Omega \phi^2\mathrm{d}x}\right\} =:\Lambda.
\end{equation}
To this end,  for each $\delta>0$, we  define
\begin{equation}\label{liu20220714-3}
    \Lambda_\delta:=\inf_{\phi\in \mathbf{I}}\left\{\frac{
    \int_{\Omega_\delta} (|\nabla\phi|^2+c(x)\phi^2)\mathrm{d}x}{\int_{\Omega_\delta} \phi^2\mathrm{d}x}\right\}.
\end{equation}
Notice 
that $\Lambda_\delta\to\Lambda$ as $\delta\to 0$. Given any $\epsilon>0$, it suffices to prove
\begin{equation}\label{liu20220113}
 \liminf_{A\to\infty}\lambda(A)\geq \Lambda_\delta-{\rm C}\sqrt{\epsilon},
\end{equation}
where ${\rm C}$ is some positive constant independent of $\epsilon$ to be determined later.
Then \eqref{liu20220714-1} follows from the arbitrariness of $\delta$ and $\epsilon$. To this end, we shall construct a positive super-solution  $\overline{\varphi}\in C(\overline{\Omega})\cap C^2(\Omega\setminus\partial\Omega_\delta)$ such that
\begin{equation}\label{super-1}
 \left\{\begin{array}{ll}
\medskip
-\Delta\overline{\varphi}-A\mathbf{b}\cdot\nabla\overline{\varphi}+c(x)\overline{\varphi}\geq (\Lambda_\delta-{\rm C}\sqrt{\epsilon})\overline{\varphi} &\mathrm{in} \,\, \Omega\setminus\partial\Omega_\delta,\\
\medskip
(\nabla_+\overline{\varphi}(x)-\nabla_-\overline{\varphi}(x))\cdot \nu(x)>0 & \mathrm{on}~\partial\Omega_\delta,\\
  \nabla\overline{\varphi}\cdot \mathbf{n}\geq 0 & \mathrm{on}~\partial\Omega,  
  \end{array}\right.
 \end{equation}
provided that $A$ is taken large. Here $\nabla\overline{\varphi}_\pm\cdot \nu$ is defined as in \eqref{liu0621-3}.
    Then \eqref{liu20220113} follows from the comparison principle. The construction is divided into the following three parts.

{\bf (1)} The construction of $\overline{\varphi}$ on $\Omega_\delta$.
By the proof of \cite[Theorem2.2]{BHN2005}, there exists some $\phi\in{\rm I}$ attaining the infimum  in \eqref{liu20220714-1}, which implies
\begin{equation}\label{liu1017}
    \int_{\Omega_\delta}\nabla\phi\cdot\nabla\zeta{\rm d}x+\int_{\Omega_\delta}c(x)\phi\zeta{\rm d}x=\Lambda_\delta\int_{\Omega_\delta}\phi\zeta{\rm d}x,\quad \forall \zeta\in {\bf I}.
\end{equation}
Choose $\{\phi_\ell\}_{\ell>0}\subset{\bf I}$ satisfying $\phi_\ell\in C^2(\overline{\Omega}_\delta)$,   $\phi_\ell\geq \sqrt{\ell}$  in $\Omega_\delta$, and  $\nabla\phi_\ell\cdot \nu \geq \ell \phi_\ell$ on  $\partial\Omega_\delta$ such that $\phi_\ell\to\phi$ in $H^1(\Omega_\delta)$ as $\ell\to 0$. Then \eqref{liu1017} implies
$$\lim_{\ell\to 0}\left[\int_{\Omega_\delta} (-\Delta \phi_\ell+(c(x)-\Lambda_\delta) \phi_\ell)\zeta \mathrm{d}x\right]=0, \quad \forall \zeta\in {\bf I}.$$
Given any $\epsilon>0$, by  Banach-Steinhaus theorem, there exists
 $\overline{\phi}_\epsilon\in\{\phi_\ell\}_{\ell>0}$, which satisfies  $\overline{\phi}_\epsilon\geq \sqrt{\epsilon}$  in $\Omega_\delta$ and $\nabla\overline{\phi}_\epsilon\cdot {\bf n} \geq \sqrt{\epsilon} \overline{\phi}_\epsilon $ on $\partial\Omega$, such that
\begin{equation}\label{liu-1213}
 \left|\int_{\Omega_\delta} (-\Delta \overline{\phi}_\epsilon+(c(x)-\Lambda_\delta) \overline{\phi}_\epsilon)\zeta \mathrm{d}x\right|\leq \epsilon,\quad \forall \zeta\in{\bf I}.
\end{equation}
By choosing the  test function $\zeta\in{\bf I}$ in \eqref{liu-1213}, we can assert that  for any periodic orbit $\mathcal{C}\subset\Omega_\delta$,
\begin{equation}\label{liu-20220111-1}
  \left|\int_{\mathcal{C}} (-\Delta \overline{\phi}_\epsilon+(c(x)-\Lambda_\delta) \overline{\phi}_\epsilon)\mathrm{d}s\right|\leq \epsilon,
\end{equation}
where  ${\rm d}s$ denotes the arc-length element along the periodic orbit $\mathcal{C}$. We next define $\overline{\psi}\in C^2(\overline{\Omega}_\delta)
$ such that for any periodic orbit $\mathcal{C}=\{\xi(t): 0\leq t<T\}\subset\Omega_\delta$ with period $T$,
  it holds that
\begin{equation}\label{20220110-2}
\begin{split}
 \frac{\mathrm{d}\overline{\psi}(\xi(t))}{\mathrm{d}t}=&-\Delta \overline{\phi}_\epsilon(\xi(t))+(c(\xi(t))-\Lambda_\delta) \overline{\phi}_\epsilon(\xi(t))\\
 &-\int_{0}^T (-\Delta \overline{\phi}_\epsilon+(c(\xi(t))-\Lambda_\delta) \overline{\phi}_\epsilon)\mathrm{d}t,
 \end{split}
\end{equation}
where $\overline{\psi}\in C^2(\Omega_\delta)$ is possible since there are no fixed points of \eqref{system} in $\Omega_\delta$.
Define
\begin{equation}\label{liu0810-2}
    \overline{\varphi}(x):=\overline{\phi}_\epsilon(x)+\frac{\overline{\psi}(x)}{A},\qquad \forall x\in \Omega_\delta.
\end{equation}
 For any periodic orbit $\mathcal{C}=\{\xi(t): 0\leq t<T\}\subset\Omega_\delta$, in view of $\overline{\phi}_\epsilon\in {\bf I}$,   we can deduce from  \eqref{20220110-2} and \eqref{liu0810-2} that for any $x\in\mathcal{C}$,
 \begin{align}\label{liu20220111}
    &-\Delta\overline{\varphi}(x)-A\mathbf{b}\cdot\nabla\overline{\varphi}(x)+c(x)\overline{\varphi}(x)\notag\\
    =&-\Delta\overline{\phi}_\epsilon(x)-\mathbf{b}\cdot\nabla\overline{\psi}(x)+c(x)\overline{\phi}_\epsilon(x)+(-\Delta\overline{\psi}(x)+c(x)\overline{\psi}(x))/A \notag\\
    = &-\Delta\overline{\phi}_\epsilon(\xi(t))-\frac{\mathrm{d}\overline{\psi}(\xi(t))}{\mathrm{d}t}+c(\xi(t))\overline{\phi}_\epsilon(\xi(t))+O(1/A)\\
    =&\int_{0}^T (-\Delta \overline{\phi}_\epsilon(\xi(t))+(c(\xi(t))-\Lambda_\delta)\overline{\phi}_\epsilon(\xi(t)))\mathrm{d}t\notag\\
    &+\Lambda_\delta\overline{\phi}_\epsilon(\xi(t))+O(1/A).\notag
 \end{align}
In light of $\sqrt{\epsilon}\leq \overline{\phi}_\epsilon\leq \overline{\varphi}$, by choosing $A$ large such that $\Lambda_\delta\overline{\psi}/A\leq \epsilon$, we can deduce from  \eqref{liu-20220111-1} and \eqref{liu20220111}  that
 \begin{align*}
 &-\Delta\overline{\varphi}_\epsilon^r(x)+A\mathbf{b}\cdot\nabla\overline{\varphi}_\epsilon^r(x)+c(x)\overline{\varphi}_\epsilon(x) \\
 \geq& \Lambda_\delta\overline{\phi}_\epsilon(x)-\frac{\epsilon}{\min\limits_{x\in\mathcal{C}}|{\bf b}(x)|}\\
 \geq& \Lambda_\delta\overline{\phi}_\epsilon(x)+\Lambda_\delta\frac{\overline{\psi}(x)}{A}-{\rm C}\epsilon\\
 \geq& (\Lambda_\delta-{\rm C}\sqrt{\epsilon})\overline{\varphi}_\epsilon(x),\qquad \forall x\in \mathcal{C},
  \end{align*}
with some constant ${\rm C}>0$ 
independent of $A$ and $\epsilon$,
  which implies $\overline{\varphi} \in C^2(\Omega_\delta)$ defined by \eqref{liu0810-2} satisfies \eqref{super-1} on $\Omega_\delta$.

{\bf (2)} The construction of $\overline{\varphi}$ on $\tilde{O}_\delta^i$ for any $i\in \mathbb{C}$. Let $x_{\rm c}^i\in \tilde{O}_\delta^i$  be the   center point of \eqref{system}. For each $r\in (0,\delta)$, we define $\mathcal{C}_r\subset \tilde{O}_\delta^i$ as a periodic orbit of \eqref{system} such that ${d_\mathcal{H}}(\mathcal{C}_r, x_{\rm c}^i)=r$ which can be parameterized by $\{\xi_r(t): 0\leq t<T_r\}$ with period $T_r$. For any $x\in \tilde{O}_\delta^i$, we  perform a $C^2$ smooth change of coordinate $x\mapsto(t,r)$ such that $x=\xi_r(t)$. Since $\overline{\phi}_\epsilon\in{\bf I}$, we find $\overline{\phi}_\epsilon(\xi_\delta(t))$ is a constant for all $t\in[0,T_\delta)$, which is denoted by  $\overline{\phi}_\epsilon(\xi_\delta)$ for simplicity. Recall $\overline{\psi}$ defined in \eqref{20220110-2}.
We define
\begin{equation}\label{liu0810-3}
    \overline{\varphi}(t,r):=\overline{\phi}_\epsilon(\xi_\delta)-\frac{r^2}{\sqrt{\delta}}+\delta^{3/2}+\frac{\overline{\psi}(\xi_\delta(t))}{A}, \quad \forall (t,r)\in \tilde{O}_\delta^i.
\end{equation}
This together with the construction in part {\bf (1)} implies obviously that $\overline{\varphi}$ is continuous on $\partial \tilde{O}_\delta^i$. We can choose $\delta$ small and $A$ large if necessary such that
\begin{equation*}
    \begin{split}
        \nabla_-\overline{\varphi}(x)\cdot \nu(x)&=-\partial_r\overline{\varphi}(t,r)\Big |_{r=\delta}+O(1/A)=2\sqrt{\delta}+O(1/A)\\
        &< \sqrt{\epsilon}\overline{\varphi}(x)=\nabla_+\overline{\varphi}(x)\cdot \nu(x), \quad \forall x\in \partial \tilde{O}_\delta^i,
    \end{split}
\end{equation*}
which verifies the boundary conditions in \eqref{super-1} holds on $\partial \tilde{O}_\delta^i$. Next, we shall verify the defined $\overline{\varphi}$ in \eqref{liu0810-3}  satisfies the first inequality in \eqref{super-1}. To see this, direct calculations yield
\begin{align}\label{liu0821-6}
    &-\Delta\overline{\varphi}(x)-A\mathbf{b}\cdot\nabla\overline{\varphi}(x)+c(x)\overline{\varphi}(x)-\Lambda_\delta\overline{\varphi}(x)\notag\\
    = &\frac{2}{\sqrt{\delta}}+O(\sqrt{\delta})+O(1/A)-\frac{\mathrm{d}\overline{\psi}(\xi_\delta(t))}{\mathrm{d}t}+(c(x)-\Lambda_\delta)\overline{\varphi}\\
    \geq&\frac{2}{\sqrt{\delta}}-\max_{x\in \partial \tilde{O}_\delta^i}|-\Delta \overline{\phi}_\epsilon(x)+c(x)|-1 -(c(x)-\Lambda_\delta)\overline{\varphi},\notag
 \end{align}
 whence we can take $\delta$ small such that \eqref{super-1} holds on $\tilde{O}_\delta^i$.

 {\bf (3)} The construction on $O_\delta^i$ for any $i\in \mathbb{S}$. The construction is rather similar to that in Step 1 of Theorem \ref{theorem20220804}, and we only give a sketch.  Under the coordinate $x\mapsto(y,\ell)$ introduced in  the proof of Theorem \ref{theorem20220722}, we shall define $\overline{\varphi}\in C^2(O_\delta^i)$ as in \eqref{liu0807-4} and \eqref{liu0807-7},
where the functions $\overline{\phi}$ and $\rho$ there is chosen such that $\overline{\varphi}$ is continuous  on $\partial O_\delta^i$. Observe that
$$\nabla_+\overline{\varphi}(x)\cdot \nu(x)=(\nabla\overline{\phi}_\epsilon+\nabla\overline{\psi}/A)\cdot \nu \geq (\sqrt{\epsilon}+O(1/A)) \overline{\varphi},\quad \forall x\in \partial O_\delta.$$
Applying \eqref{liu0821-1} we can confirm the boundary conditions on $\partial O_\delta$ in \eqref{super-1} can be satisfied by choosing $\delta>0$ small if necessary.
 Then by the same arguments as in Step 1 of  Theorem \ref{theorem20220804}, we can verify the first inequality in \eqref{super-1} also holds on  $O_\delta^i$.

Summarily, we have constructed the super-solution $\overline{\varphi}\in C(\overline{\Omega})\cap C^2(\Omega_\delta)$ satisfying \eqref{super-1}, then \eqref{liu20220113} can be derived by the comparison principle. Letting $\epsilon\to 0$, we conclude \eqref{liu20220113} holds. This complets Step 1.

\medskip
{\bf Step 2. Upper bound estimate}:
We next prove the upper bound estimate
\begin{equation}\label{liu20220113-1}
 \limsup_{A\to\infty}\lambda(A)\leq \Lambda=\inf_{\phi\in \mathbf{I}}\left\{\frac{
 \int_\Omega (|\nabla\phi|^2+c(x)\phi^2)\mathrm{d}x}{\int_\Omega \phi^2\mathrm{d}x}\right\}.
\end{equation}
Let $\Lambda_\delta$ be defined by 
\eqref{liu20220714-3}.
It suffices to show
\begin{equation}\label{liu0821-2}
 \limsup_{A\to\infty}\lambda(A)\leq \Lambda_\delta+{\rm C}\sqrt{\epsilon}
\end{equation}
for some ${\rm C}>0$ independent of $\epsilon$. Then \eqref{liu20220113-1} follows from the arbitrariness of $\delta$ and $\epsilon$. As in Step 1, \eqref{liu0821-2} can be proved by constructing a sub-solution  $\underline{\varphi}\in C(\overline{\Omega})\cap C^2(\Omega\setminus\Omega_\delta)$ satisfying
\begin{equation}\label{liu0821-3}
 \left\{\begin{array}{ll}
\medskip
-\Delta\underline{\varphi}-A\mathbf{b}\cdot\nabla\underline{\varphi}+c(x)\underline{\varphi}\leq (\Lambda_\delta+{\rm C}\sqrt{\epsilon})\underline{\varphi} &\mathrm{in} \,\, \Omega,\\
\medskip
(\nabla_+\underline{\varphi}(x)-\nabla_-\underline{\varphi}(x))\cdot \nu(x)<0 & \mathrm{on}~\partial\Omega_\delta\\
  \nabla\underline{\varphi}\cdot \mathbf{n}\leq 0 & \mathrm{on}~\partial\Omega.  
  \end{array}\right.
 \end{equation}
The construction is similar to that in Step 1, and we shall give a sketch for completeness.

{\bf (1)} The construction of $\underline{\varphi}$ on $\Omega_\delta$. For any $\epsilon>0$, similar to Step 1 we can define  $\underline{\phi}_\epsilon\in\mathbf{I}$ satisfying $\underline{\phi}_\epsilon\in C^2(\Omega_\delta)$,  $\underline{\phi}_\epsilon\geq \sqrt{\epsilon}$  in $\Omega_\delta$, and $\nabla\underline{\phi}_\epsilon\cdot \nu\leq -\sqrt{\epsilon} \underline{\phi}_\epsilon$ on $\partial\Omega_\delta$
 such that
\begin{equation*}
 \left|\int_{\Omega_\delta} (-\Delta \underline{\phi}_\epsilon+(c(x)-\Lambda_\delta) \underline{\phi}_\epsilon)\zeta \mathrm{d}x\right|\leq \epsilon,\quad \forall \zeta\in{\bf I},
\end{equation*}
and thus \eqref{liu-1213} holds for $\underline{\phi}_\epsilon$. We then define
\begin{equation}\label{liu0821-3-1}
    \underline{\varphi}(x):=\underline{\phi}_\epsilon(x)+\frac{\underline{\psi}(x)}{A},\qquad \forall x\in \Omega_\delta,
\end{equation}
where $\underline{\psi}$ is defined in \eqref{20220110-2} by replacing $\overline{\phi}_\epsilon$ by $\underline{\phi}_\epsilon$. Then analogue to \eqref{liu20220111}, we can verify that such $\underline{\varphi}$ satisfies \eqref{liu0821-3} on $\Omega_\delta$.

{\bf (2)}  The construction of $\underline{\varphi}$ on $\tilde{O}_\delta^i$ for any $i\in \mathbb{C}$. Under the coordinate $x\mapsto(t,r)$ introduced in Step 1, we define
\begin{equation}\label{liu0821-5}
    \underline{\varphi}(t,r):=\underline{\phi}_\epsilon(\xi_\delta)+\frac{r^2}{\sqrt{\delta}}-\delta^{3/2}+\frac{\underline{\psi}(\xi_\delta(t))}{A}, \quad \forall (t,r)\in \tilde{O}_\delta^i.
\end{equation}
Together with \eqref{liu0821-3-1}, we find that $\overline{\varphi}$ is continuous on $\partial \tilde{O}_\delta^i$. Similar to \eqref{liu0821-6}, by direct calculations  we can choose $\delta$ small such that the first inequality in \eqref{liu0821-3} holds on $ \tilde{O}_\delta^i$. Furthermore, in light of $ \nabla\underline{\phi}_\epsilon\cdot \nu\leq -\sqrt{\epsilon}\underline{\phi}_\epsilon$ on $\partial\tilde{O}_\delta$, by \eqref{liu0821-3-1}  we can derive that
\begin{equation*}
    \begin{split}
        \nabla_-\underline{\varphi}(x)\cdot \nu(x)&=-\partial_r\underline{\varphi}(t,r)\Big |_{r=\delta}+O(1/A)=-2\sqrt{\delta}+O(1/A)\\
        &>-\sqrt{\epsilon}\overline{\varphi}(x)=\nabla_+\underline{\varphi}(x)\cdot \nu(x), \quad \forall x\in \partial \tilde{O}_\delta^i
    \end{split}
\end{equation*}
by letting $\delta$ small and $A$ large if necessary. This implies  the boundary conditions in \eqref{liu0821-3} also  hold on $\partial \tilde{O}_\delta^i$.

{\bf (3)} The construction of $\underline{\varphi}$ on $O_\delta^i$ for any $i\in \mathbb{S}$. We assume the Jacobi matrix $D{\bf b}(x_{\rm s}^i)$ is given by \eqref{liu-0722-8}.
Let the coordinate $x\mapsto(y,\ell)$ be introduced as in  the proof of Theorem \ref{theorem20220722} such that $x=x_\ell(\tau(y))$, where $x_\ell$ denots the periodic solution of \eqref{liu_0730-1} with $\eta_\ell=0$.  Analogue to \eqref{liu0807-4} and \eqref{liu0807-7}, we  define
 \begin{equation}\label{liu0821-6-1}
\underline{\varphi}(y,\ell):=\exp\left\{\frac{\rho_1(y)}{A}+\frac{\ell^2}{\sqrt{\delta}}\right\},\qquad \forall (y,\ell)\in O_\delta^i\setminus B_{1/\sqrt{\lambda_2 A}}(x_{\rm s}^i),
\end{equation}
and
\begin{equation}\label{liu0821-7}
\underline{\varphi}(x):=\exp\left\{D\|x-x_{\rm s}^i\|^2
+\frac{\rho_2(x)}{A}\right\},\qquad \forall x\in B_{1/\sqrt{\lambda_2 A}}(x_{\rm s}^i),
\end{equation}
where $\lambda_2>0$ be defined in \eqref{liu-0722-8}, and the functions  $\rho_1$ and $\rho_2$ are chosen such that $\underline{\varphi}\in C^2(O_\delta^i)$ as well as $\underline{\varphi}$ is continuous on $\partial O_\delta^i$.

For any $x\in O_\delta^i\setminus B_{1/\sqrt{\lambda_2 A}}(x_{\rm s}^i)$, due to $\eta_\ell=0$, using the similar arguments as in  \eqref{liu0815-6}  and \eqref{liu0815-7}, by \eqref{liu0821-6-1} we can choose $\delta$ small such that
 \begin{align*}
          &-\Delta\underline{\varphi}-A\mathbf{b}(x)\cdot\nabla\underline{\varphi}+c(x)\underline{\varphi}\\
          =&-\Delta\underline{\varphi}-\frac{A|({\bf I}+\eta_\ell \mathcal{J})\mathbf{b}|}{\partial_\ell x\cdot(\mathcal{J}\mathbf{b}-\eta_\ell \mathbf{b} )}\partial_\ell x\cdot\mathcal{J}{\bf b}\partial_y \underline{\varphi}+c(x)\underline{\varphi}\\
=&\left[-\frac{2}{\sqrt{\delta}}+O(1/A)+O(\sqrt{\delta})\right]\underline{\varphi}+(|{\bf b}(x)|+O(\delta))\left|\frac{{\rm d}\rho_1(y)}{{\rm d}y}\right|\underline{\varphi}+c(x)\underline{\varphi}\\
\leq& \Lambda_\delta\underline{\varphi}, \qquad \forall x\in O_\delta^i\setminus B_{1/\sqrt{\lambda_2 A}}(x_{\rm s}^i).
 \end{align*}
Hence, \eqref{liu0821-3}  holds on $O_\delta^i\setminus B_{1/\sqrt{\lambda_2 A}}(x_{\rm s}^i)$. For any $x\in  B_{1/\sqrt{\lambda_2 A}}(x_{\rm s}^i)$,  the first inequality in \eqref{liu0821-3} can also be verified by letting $D$ large. Moreover, similar to \eqref{liu0821-1}, we can find the boundary condition in  \eqref{liu0821-3} on $\partial O_\delta^i$ holds true.

Until now we have constructed the postive sub-solution  $\underline{\varphi}\in C(\overline{\Omega})\cap C^2(\Omega\setminus\partial \Omega_\delta)$ satisfying \eqref{liu0821-3}. Then by comparison principle we can derive \eqref{liu0821-2}, which implies
the upper bound estimate
\eqref{liu20220113-1} holds. The proof is now complete.
\end{proof}

The main result in this section can be formulated as follows.

\begin{theorem}\label{theorem0825}
Suppose that  the limit set of system  \eqref{system} consists of a finite number of hyperbolic fixed points, and a  family $\Omega_{\mathbb{F}}$  of closed orbits which is made of the periodic orbits, two homoclinic orbits connected by a hyperbolic saddle, or  center points as in Proposition {\rm \ref{theorem0820}}.
Assume ${\bf b}(x)\neq 0$, $\forall x\in\partial\Omega_{\mathbb{F}}$.
Let $\{x_1,\cdots,x_k\}$ denote the set of stable fixed points. 
\begin{itemize}
    \item[{\rm(i)}] If the family $\Omega_{\mathbb{F}}$ of closed orbits  is stable (attracting), then
    $$\lim_{A\to\infty}\lambda(A)= \min\left\{\min_{1\leq i\leq k}\{c(x_i)\},\,\,\,\, \inf_{\phi\in \mathbf{I}}\left[\frac{\int_{\Omega_{\mathbb{F}}} (|\nabla\phi|^2+c(x)\phi^2)\mathrm{d}x}{\int_{\Omega_{\mathbb{F}}} \phi^2\mathrm{d}x}\right] \right\};$$
     whereas $\mathbf{I}=\{\phi\in H^1(\Omega_{\mathbb{F}}): \mathbf{b}\cdot\nabla\phi =0 \text{ a.e. in }\Omega_{\mathbb{F}}\}$.

     \medskip

    \item[(ii)] If the family $\Omega_{\mathbb{F}}$ of closed orbits is unstable (repelling), then
    $$\lim_{A\to\infty}\lambda(A)= \min\left\{\min_{1\leq i\leq k}\{c(x_i)\},\,\,\,\, \inf_{\phi\in \mathbf{I}_0}\left[\frac{\int_{\Omega_{\mathbb{F}}} (|\nabla\phi|^2+c(x)\phi^2)\mathrm{d}x}{\int_{\Omega_{\mathbb{F}}} \phi^2\mathrm{d}x}\right] \right\}.$$
   whereas $\mathbf{I}_0=\{\phi\in H^1_0(\Omega_{\mathbb{F}}): \mathbf{b}\cdot\nabla\phi =0 \text{ a.e. in }\Omega_{\mathbb{F}}\}$.
\end{itemize}
Here the definition of the stability of $\Omega_\mathbb{F}$ is given in Definition {\rm \ref{definition2}}.
\end{theorem}

\begin{proof}
The proof is based on Proposition \ref{theorem0820}.

\smallskip

{\bf Step 1}. We first prove the assertion {\rm (i)}. We only  establish the lower bound estimate
\begin{equation}\label{liu0825-2}
 \liminf_{A\to\infty}\lambda(A)\geq \min\left\{\min_{1\leq i\leq k}\{c(x_i)\},\,\,\,\, \inf_{\phi\in \mathbf{I}}\left[\frac{\int_{\Omega_{\mathbb{F}}} (|\nabla\phi|^2+c(x)\phi^2)\mathrm{d}x}{\int_{\Omega_{\mathbb{F}}} \phi^2\mathrm{d}x}\right] \right\}=:\Lambda,
\end{equation}
since the corresponding upper bound estimate can be proved similarly. 
Define $O_\delta(\Omega_{\mathbb{F}}):=\{x\in\Omega: {d_\mathcal{H}}(x,\Omega_{\mathbb{F}})<\delta\}$ as a $\delta$-neighbourhood of  $\Omega_{\mathbb{F}}$, and denote by $\nu(x)$ the the outward unit normal vector on $\partial\Omega_{\mathbb{F}}\cup\partial O_\delta(\Omega_{\mathbb{F}})$. For any given $\epsilon>0$,  we next construct a super-solution  $\overline{\varphi}\in C(\overline{\Omega})\cap C^2(\Omega\setminus\Omega_\delta)$ such that
\begin{equation}\label{liu0825-3}
 \left\{\begin{array}{ll}
\medskip
-\Delta\overline{\varphi}-A\mathbf{b}\cdot\nabla\overline{\varphi}+c(x)\overline{\varphi}\geq (\Lambda-2\epsilon)\overline{\varphi} &\mathrm{in} \,\, \Omega\setminus(\partial\Omega_{\mathbb{F}}\cup\partial O_\delta(\Omega_{\mathbb{F}})\cup\partial\Omega_\delta),\\
\medskip
(\nabla_-\overline{\varphi}(x)-\nabla_+\overline{\varphi}(x))\cdot \nu(x)>0 & \mathrm{on}~\partial\Omega_{\mathbb{F}}\cup\partial O_\delta(\Omega_{\mathbb{F}})\cup\partial\Omega_\delta,\\
  \nabla\overline{\varphi}\cdot \mathbf{n}\geq 0 & \mathrm{on}~\partial\Omega,  
  \end{array}\right.
 \end{equation}
provided that $A$ is chosen large, where $\Omega_\delta\subset\Omega_{\mathbb{F}}$ is defined in the proof of Proposition \ref{theorem0820}. Then \eqref{liu0825-2} follows from the comparison principle. 

{\bf (1)} The construction on $\Omega_{\mathbb{F}}$. 
Let $\Lambda_\delta$ be defined by \eqref{liu20220714-3}.
By the proof of Proposition \ref{theorem0820} we can construct $\overline{\varphi}\in C(\Omega_\mathbb{F})\cap C^2(\Omega_{\mathbb{F}}\setminus\partial\Omega_\delta)$ satisfying
\begin{equation}\label{liu0825-6-1}
 \left\{\begin{array}{ll}
\medskip
-\Delta\overline{\varphi}-A\mathbf{b}\cdot\nabla\overline{\varphi}+c(x)\overline{\varphi}\geq(\Lambda_\delta-\epsilon)\overline{\varphi} &\mathrm{in} \,\, \Omega_{\mathbb{F}}\setminus\partial\Omega_\delta,\\
\medskip
(\nabla_-\overline{\varphi}(x)-\nabla_+\overline{\varphi}(x))\cdot \nu(x)>0 & \mathrm{on}~\partial\Omega_\delta,\\
  \nabla\overline{\varphi}\cdot \nu (x)\geq\sqrt{\ep}\overline{\varphi} & \mathrm{on}~\partial\Omega_{\mathbb{F}},  
  \end{array}\right.
 \end{equation}
 for sufficiently large $A$.
Hence, 
the first inequality in \eqref{liu0825-3} holds for small $\delta$.  Furthermore, by \eqref{liu0810-2} and $\overline{\phi}_\epsilon\in{\bf I}$ we find that the constructed  function $\overline{\varphi}$  satisfies
\begin{equation}\label{liu0825-7-1}
    \overline{\varphi}(x)={\rm C}+\frac{\overline{\psi}(x)}{A},\qquad \forall x\in \partial\Omega_{\mathbb{F}},
\end{equation}
where ${\rm C}$ is a constant determined by $\overline{\phi}_\epsilon(x)\equiv {\rm C}$ on  $\partial\Omega_{\mathbb{F}}$,  and $\overline{\psi}$ is defined as in \eqref{20220110-2}.

{\bf (2)} The construction on $O_\delta(\Omega_{\mathbb{F}})\setminus\Omega_{\mathbb{F}}$. We
first introduce the coordinate $x\mapsto(\tau, \ell)$ in each connected component of $O_\delta(\Omega_{\mathbb{F}})\setminus \Omega_{\mathbb{F}}$ as in Part 2 of Theorem \ref{Limit-cycle}. Indeed,  given any closed orbit on $\partial \Omega_{\mathbb{F}}$ which can be parameterized by $\{P(t):0\leq t<T\}$, due to the stability of $\Omega_{\mathbb{F}}$, for any $\ell\in(0,\delta)$, we can
 choose $\eta_\ell>0$ such that equation \eqref{liu_04-0729} admits a periodic solution  $x_\ell(\tau)$. Then we  perform a $C^2$ smooth change of coordinate $x\mapsto(\tau,\ell)$ such that $x=x_\ell(\tau)$.

 Then we define
\begin{equation}\label{liu0825-7}
   \overline{\varphi}(\tau, \ell):=\left[{\rm C}+\frac{\overline{\psi}(\tau, 0)}{A}\right]\exp\left\{\frac{\sqrt{\epsilon}}{2}\ell-\frac{\ell^2}{\sqrt{\delta}}\right\}, \qquad \forall (\tau, \ell)\in O_\delta(\Omega_{\mathbb{F}})\setminus\Omega_{\mathbb{F}}.
\end{equation}
 Here the constant ${\rm C}$ and function $\overline{\psi}$ are given as in \eqref{liu0825-7-1} to ensure the continuity of $\overline{\varphi}$ on $\partial \Omega_{\mathbb{F}}$. As in \eqref{liu0729-5}, by direct calculations we can choose $\delta$ small such that
 \begin{align*}
    \mathbf{b}(\tau, \ell)\cdot\nabla\overline{\varphi}
=&\frac{\partial_\ell x\cdot \mathcal{J}{\bf b} }{\partial_\ell x\cdot \mathcal{J}{\bf b}+O(\eta_\ell)}\partial_\tau \overline{\varphi}-\frac{\partial_\tau x\cdot \mathcal{J}{\bf b} }{\partial_\ell x\cdot \mathcal{J}{\bf b}+O(\eta_\ell)}\partial_\ell \overline{\varphi}  \\
=&(1+O(\eta_\ell))\partial_\tau \overline{\varphi}-\frac{({\bf I}+\eta_\ell \mathcal{J})\mathbf{b}(x_\ell(\tau))\cdot \mathcal{J}{\bf b}((x_\ell(\tau)) }{\partial_\ell x\cdot \mathcal{J}{\bf b}+O(\eta_\ell)}\partial_\ell \overline{\varphi}  \\
=&\frac{(1+O(\eta_\ell))\partial_\tau\overline{\psi}}{A}-\underbrace{\frac{(\sqrt{\epsilon}-2\ell/\sqrt{\delta})\eta_\ell|\mathbf{b}(x_\ell(\tau))|^2 }{\partial_\ell x\cdot \mathcal{J}{\bf b}+O(\eta_\ell)}}_{\geq 0} \\
\leq &\frac{(1+O(\eta_\ell))\partial_\tau\overline{\psi}}{A},  \qquad \forall (\tau, \ell)\in O_\delta(\Omega_{\mathbb{F}})\setminus\Omega_{\mathbb{F}},
\end{align*}
where the last inequality is due to $\eta_\ell\geq 0$ and $\ell\in (0,\delta)$.
Hence, by \eqref{liu0825-7} we have
\begin{align*}
    &-\Delta\overline{\varphi}-A\mathbf{b}\cdot\nabla\overline{\varphi}+\left(c(x)-\Lambda\right)\overline{\varphi}\notag\\
   \geq &\frac{2}{\sqrt{\delta}}\left(\left(\partial_{x_1} \ell\right)^2+\left(\partial_{x_2} \ell\right)^2\right)\overline{\varphi}-\frac{\sqrt{\epsilon}}{2}\left(\partial^2_{x_1} \ell+\partial^2_{x_2} \ell\right)\overline{\varphi}\\
   &-(1+O(\eta_\ell)))\partial_\tau\overline{\psi}(\tau,0)+\left(c(x)-\Lambda\right)\overline{\varphi}.
\end{align*}
In view of  $(\partial_{x_1}r)^2+(\partial_{x_2}r)^2>0$,  by choosing $\delta$ small we can derive  the first inequality in \eqref{liu0825-3} holds on $O_\delta(\Omega_{\mathbb{F}})\setminus\Omega_{\mathbb{F}}$.

Furthermore, combining \eqref{liu0825-6-1}  and \eqref{liu0825-7}, similar to \eqref{liu_0730-6}, we can verify that
\begin{equation}\label{liu0826-1}
\begin{split}
\nabla_+\overline{\varphi}(x)\cdot \nu(x)\Big |_{\ell=0}=&\frac{\mathcal{J}\partial_\tau x \cdot\mathcal{J}({\bf I}+\eta_{\delta}\mathcal{J}){\bf b}\partial_\ell \overline{\varphi}+O(1/A)}{(\partial_\ell x\cdot \mathcal{J}{\bf b}+O(\eta_\delta))|\frac{{\rm d} x_{\delta}(\tau)}{{\rm d} \tau}|}\Bigg |_{\ell=0}\\
=&\frac{(\sqrt{\epsilon}/2)|({\bf I}+\eta_{\delta}\mathcal{J}){\bf b}|^2\overline{\varphi}+O(1/A)}{(\partial_\ell x\cdot \mathcal{J}{\bf b}+O(\eta_\delta))(|{\bf b}(x_\delta(\tau))|+O(\eta_\delta))}\\
=&\frac{\sqrt{\epsilon}}{2}(1+O(\delta))\overline{\varphi}+O(1/A)\\
<&\sqrt{\epsilon} \overline{\varphi}\leq\nabla_-\overline{\varphi}(x)\cdot \nu(x)\Big |_{\ell=0},\quad \forall x\in\partial\Omega_{\mathbb{F}},
\end{split}
\end{equation}
by choosing $A$ large and $\delta$ small if necessary. This implies the boundary conditions in \eqref{liu0825-2} hold on $\partial\Omega_{\mathbb{F}}$. The construction on $O_\delta(\Omega_{\mathbb{F}})\setminus\Omega_{\mathbb{F}}$ is completed.

{\bf (3)} The construction on $\Omega\setminus O_\delta(\Omega_{\mathbb{F}})$. For any $x\in\partial O_\delta(\Omega_{\mathbb{F}})$, analogue to \eqref{liu0826-1}, by \eqref{liu0825-7} we  choose $A$ large and $\delta$ small such that
\begin{equation*}
\begin{split}
\nabla_+\overline{\varphi}(x)\cdot \nu(x)\Big |_{\ell=\delta}
=&\frac{(\sqrt{\epsilon}/2-2\sqrt{\delta})|({\bf I}+\eta_{\delta}\mathcal{J}){\bf b}|^2\overline{\varphi}+O(1/A)}{(\partial_\ell x\cdot \mathcal{J}{\bf b}+O(\eta_\delta))(|{\bf b}(x_\delta(\tau))|+O(\eta_\delta))}\\
\geq &\frac{\sqrt{\epsilon}}{4} \overline{\varphi},\quad \forall x\in\partial O_\delta(\Omega_{\mathbb{F}}).
\end{split}
\end{equation*}
Hence, we can apply the arguments developed in Step 2 of Part 1 in the proof of Theorem \ref{Limit-cycle} to complete the construction of $\overline{\varphi}$  on $\Omega\setminus O_\delta(\Omega_{\mathbb{F}})$ such that  \eqref{liu0825-2} holds. The proof of assertion {\rm (i)} is now complete.

{\bf Step 2}. We next prove the assertion {\rm (ii)}. By same arguments as in Step 2 of Theorem \ref{liuprop1}, we can drive that $$\limsup\limits_{A\to\infty}\lambda(A)\leq \min\limits_{1\leq i\leq k}\{c(x_i)\}.$$
Let $\lambda_\infty(A)$ denote the principal eigenvalue of
the problem
\begin{equation*}
 \left\{\begin{array}{ll}
\medskip
-\Delta\varphi-A\mathbf{b}\cdot\nabla\varphi+c(x)\varphi=\lambda\varphi&\mathrm{in} \,\, \Omega_{\mathbb{F}},\\
\varphi=0 & \mathrm{on}~\partial\Omega_{\mathbb{F}}.  
  \end{array}\right.
 \end{equation*}
 By the comparison principle, it is easily seen that $\lambda(A)\leq \lambda_\infty(A)$ for all $A>0$, which together with Proposition \ref{theorem0820}{\rm (ii)} yields
\begin{equation*}
     \limsup_{A\rightarrow\infty}\lambda(A)\leq \lim_{A\rightarrow\infty}\lambda_\infty(A)=\inf_{\phi\in \mathbf{I}_0}\left[\frac{\int_{\Omega_{\mathbb{F}}} (|\nabla\phi|^2+c(x)\phi^2)\mathrm{d}x}{\int_{\Omega_{\mathbb{F}}} \phi^2\mathrm{d}x}\right].
\end{equation*}
Hence, the upper bound estimate of assertion {\rm (ii)} follows. It remains to establish 
\begin{equation}\label{liu0827-1}
        \liminf_{A\to\infty}\lambda(A)\geq \min\left\{\min_{1\leq i\leq k}\{c(x_i)\},\,\,\,\, \inf_{\phi\in \mathbf{I}_0}\left[\frac{\int_{\Omega_{\mathbb{F}}} (|\nabla\phi|^2+c(x)\phi^2)\mathrm{d}x}{\int_{\Omega_{\mathbb{F}}} \phi^2\mathrm{d}x}\right] \right\}.
\end{equation}

As in Step 1, we set
$$\Lambda_\delta:=\inf_{\phi\in \mathbf{I}_0}\left[\frac{\int_{\Omega_{\delta}} (|\nabla\phi|^2+c(x)\phi^2)\mathrm{d}x}{\int_{\Omega_{\delta}} \phi^2\mathrm{d}x}\right]. $$
For any given $\epsilon>0$, the fact that $\Omega_{\mathbb{F}}$ is unstable
enables us to apply the same arguments as in Step 1 and Proposition \ref{theorem0820} to construct a super-solution  $\overline{\varphi}\in C(\overline{\Omega})\cap C^2(\Omega\setminus\Omega_\delta)$ such that
\begin{equation*}
 \left\{\begin{array}{ll}
\medskip
-\Delta\overline{\varphi}-A\mathbf{b}\cdot\nabla\overline{\varphi}+c(x)\overline{\varphi}\geq \left[\min\left\{\min\limits_{1\leq i\leq k}\{c(x_i)\},\,\, \Lambda_\delta \right\}-2\epsilon\right]\overline{\varphi}\\
\qquad\qquad\qquad\qquad\qquad\qquad\qquad\qquad\quad\qquad\mathrm{in} \,\, \Omega\setminus(\partial\Omega_{\mathbb{F}}\cup\partial O_\delta(\Omega_{\mathbb{F}})\cup\partial\Omega_\delta),\\
\medskip
(\nabla_-\overline{\varphi}(x)-\nabla_+\overline{\varphi}(x))\cdot \nu(x)>0 \qquad\qquad\quad\,\, \mathrm{on}~\partial\Omega_{\mathbb{F}}\cup\partial O_\delta(\Omega_{\mathbb{F}})\cup\partial\Omega_\delta,\\
  \nabla\overline{\varphi}\cdot \mathbf{n}\geq 0 \qquad\qquad\qquad\qquad\qquad\qquad\qquad\,\, \mathrm{on}~\partial\Omega,  
  \end{array}\right.
 \end{equation*}
for sufficiently large $A$.  Then by the comparison principle and letting $|\alpha|\to\infty$ we can deduce \eqref{liu0827-1}. The proof is now complete.
\end{proof}

By  rather similar arguments as in the proof of Theorem \ref{theorem0825}, we can derive the following result, for which the proof is omitted.
\begin{theorem}\label{theorem0912}
Under the assumptions in Theorem {\rm\ref{theorem0825}},
let $\Gamma$ be the set of limit cycles on $\partial \Omega_\mathbb{F}$ such that  $\Omega_\mathbb{F}$ is unstable  on $\Gamma$ and is stable on  $\partial\Omega_\mathbb{F}\setminus \Gamma$, which might be an empty set.
Then
  $$\lim_{A\to\infty}\lambda(A)= \min\left\{\min_{1\leq i\leq k}\{c(x_i)\},\,\,\,\, \inf_{\phi\in \mathbf{I}}\left[\frac{\int_{\Omega_{\mathbb{F}}} (|\nabla\phi|^2+c(x)\phi^2)\mathrm{d}x}{\int_{\Omega_{\mathbb{F}}} \phi^2\mathrm{d}x}\right] \right\},$$
 where $\mathbf{I}:=\{\phi\in H^1(\Omega_{\mathbb{F}}): \phi=0 \text{ on }  \Gamma, \quad \mathbf{b}\cdot\nabla\phi =0 \text{ a.e. in }\Omega_{\mathbb{F}}\}$
\end{theorem}

We are now ready to prove Theorem \ref{mainresult}.

\begin{proof}[\bf Proof of Theorem {\rm \ref{mainresult}}]
     Theorem \ref{mainresult} can be established by constructing the suitable super/sub-solutions as before.
     Let  $\{\mathcal{K}_i:1\leq i\leq n\}$ be the connected components in
     the limit set of system \eqref{system}, which satisfy Hypothesis  \ref{assum1}.   Define $O_\delta(\mathcal{K}_i):=\{x\in\Omega: d_{\mathcal{H}}(x,\mathcal{K}_i)<\delta\}$ with some small $\delta>0$ to be determined later.  On the region $O_\delta(\mathcal{K}_i)$, we may
construct the desired super/sub-solution  by using directly the arguments
in Sections \ref{S2}-\ref{S4-1}.  More precisely, if $\mathcal{K}_i$ is
 a hyperbolic fixed point, then  the constructions can follow
those in Theorem \ref{liuprop1}, while 
the proofs in Theorems \ref{Limit-cycle}-\ref{Limit-cycle-3} are applicable  when $\mathcal{K}_i$ is a limit cycle. If $\mathcal{K}_i$ contains homoclinic orbits satisfying Hypothesis  \ref{assum1} ({\rm iii}) and ({\rm v}), we can apply  the arguments in Theorem \ref{theorem20220722}-\ref{theorem20220827} to complete the constructions on $O_\delta(\mathcal{K}_i)$. Moreover, if  $\mathcal{K}_i$ is a family of closed orbits as given by Hypothesis  \ref{assum1}({\rm v}),  the constructions are same to those in Theorem \ref{theorem0825}.  Finally, the constructions for the remaining region $\Omega\setminus \cup_{i=1}^n O_\delta(\mathcal{K}_i)$ can be completed by integrating the ideas in Theorems \ref{liuprop1}-\ref{theorem0825}. Therefore,
 Theorem \ref{mainresult} can be proved by choosing $\delta$ small.
\end{proof}

\section{\bf  Discussions on the degenerate case}\label{S5-1}
Set $\Omega_*:=\{x\in\Omega: {\bf b}(x)=0\}$. It is assumed in Hypothesis \ref{assum1} that $|\Omega_*|=0$. This section is devoted to some  discussions on the degenerate case $|\Omega_*|>0$ for complement. We first state the following result.

\begin{theorem}\label{theorem0831}
Suppose that ${\bf b}\equiv 0$ on $\Omega_*$ and there are no limit sets of \eqref{system} in $\Omega\setminus\Omega_*$ for some connected subset $\Omega_* \subset\Omega$ satisfying $|\Omega_*|>0$. Assume that there exists some  $\delta_*>0$ such that for any $\delta\in(0,\delta_*)$, there exists some neighborhood $\Omega_\delta$ satisfying $\Omega_*\subset\Omega_\delta$ and ${d_\mathcal{H}}\{\partial\Omega_\delta,\Omega_*\}=\delta$ such that {\rm (1)} ${\bf b}(x)\cdot \nu_\delta(x)<0$, $\forall x\in\partial\Omega_\delta$, or {\rm (2)} ${\bf b}(x)\cdot \nu_\delta(x)>0$, $\forall x\in\partial\Omega_\delta$, where  $\nu_\delta$ denotes  the unit normal vector on   $\partial\Omega_\delta$.
\begin{itemize}
    \item[{\rm(i)}] If  {\rm (1)} holds,  then $\lambda(A)\to \Lambda_{\mathbf{N}}$ as $A\to\infty$,
     where $\Lambda_{\mathbf{N}}$ is the principal eigenvalue of 
     \begin{equation}\label{liu0902-4}
-\Delta\varphi+c(x)\varphi=\Lambda_{\mathbf{N}}\varphi \,\,\,\,\mathrm{in} \,\, \Omega_*, \qquad
  \nabla\varphi\cdot \nu=0 \,\,\,\, \mathrm{on}~\partial\Omega_*,
 \end{equation}
 where $\nu$ is the unit normal vector on   $\partial\Omega_*$.

     \medskip

    \item[{\rm(ii)}] If  {\rm (2)} holds,  then $\lambda(A)\to \Lambda_{\mathbf{D}}$ as $A\to\infty$,
     where $\Lambda_{\mathbf{D}}$ is the principal eigenvalue of 
     \begin{equation}\label{liu0902-5}
-\Delta\varphi+c(x)\varphi=\Lambda_{\mathbf{D}}\varphi\,\,\,\,\mathrm{in} \,\, \Omega_*, \qquad
  \varphi=0 \,\,\,\, \mathrm{on}~\partial\Omega_*.
 \end{equation}
\end{itemize}
\end{theorem}
\begin{proof}
Given any $\alpha\in\mathbb{R}$, we begin with the
following auxiliary  problem:
\begin{equation}\label{liu0831-1}
-\Delta\varphi+c(x)\varphi=\lambda\varphi \,\,\,\,\mathrm{in} \,\, \Omega_*,\qquad
  \nabla\varphi\cdot\nu=\alpha\varphi\,\,\,\, \mathrm{on}~\partial\Omega_*.
 \end{equation}
Denote by $\Lambda_{\alpha}$ its principal eigenvalue and let $\varphi_{\alpha}>0$ be the associated eigenfunction. It is well-known that $\Lambda_\alpha$ is increasing and analytical with respect to $\alpha\in\mathbb{R}$. Clearly,
\begin{equation}\label{liu0831-2}
\lim_{\alpha\to 0}\Lambda_\alpha=\Lambda_{\mathbf{N}}\quad\,\, \text{and} \quad \,\,\lim_{\alpha\to-\infty}\Lambda_\alpha=\Lambda_{\mathbf{D}}.
\end{equation}


\noindent {\bf Step 1. } We establish the assertion {\rm(i)}.  We shall prove the  lower  estimate
 $\liminf\limits_{A\to\infty}\lambda(A)\geq\Lambda_{\mathbf{N}}$,
and the  upper  estimate can be proved by the similar arguments.
By \eqref{liu0831-2} it suffices to show
\begin{equation}\label{ldn5}
 \liminf_{A\to\infty}\lambda(A)\geq\Lambda_\alpha
\end{equation}
for sufficiently small $\delta>0$ and $\alpha>0$. To this end,
given any $\epsilon>0$, we will construct a positive  super-solution $\overline \varphi$ such that for sufficiently large $A$ and small $\delta$, there holds
\begin{equation}\label{ldn6}
 \left\{\begin{array}{ll}
\medskip
-\Delta\overline{\varphi}-A\mathbf{b}\cdot\nabla\overline{\varphi}+c(x)\overline{\varphi}\geq (\Lambda_\alpha-\epsilon) \overline{\varphi} &\mathrm{in} \,\, \Omega\setminus\partial\Omega_\delta,\\
\medskip
(\nabla_+\overline{\varphi}(x)-\nabla_-\overline{\varphi}(x))\cdot \nu_\delta(x)>0 & \mathrm{on}~\partial\Omega_\delta,\\
  \nabla\overline{\varphi}\cdot {\bf n}(x)\geq0 & \mathrm{on}~\partial\Omega.  
  \end{array}\right.
 \end{equation}
  Then \eqref{ldn5} follows from the camparison principle. Such a super-solution $\overline \varphi$ will be constructed  on the following different regions.

\smallskip

 $\mathbf{(1)}$ The construction on $\Omega_*$. For $ x \in \Omega_*$, we define $\overline \varphi(x):=\varphi_\alpha(x)$ with $\alpha>0$. In view of ${\bf b}=0$ on  $\Omega_*$, by the definition of $\varphi_\alpha$ in \eqref{liu0831-1}, one can check that  function $\overline \varphi$ satisfies the first inequality in \eqref{ldn6}  on $\Omega_*$ for all $A>0$.

\smallskip

 $\mathbf{(2)}$ The construction on $\Omega_\delta\setminus\Omega_*$.
We shall define $\overline{\varphi}\in C^2(\Omega_\delta)$  such that
 \begin{equation}\label{liu0902-1}
     {\bf b}\cdot\nabla\overline{\varphi}(x)\leq 0,\quad \forall x\in\Omega_\delta\setminus\Omega_*, \quad \text{ and }\quad
  \nabla\overline{\varphi}(x)\cdot\nu_\delta(x)\geq (\alpha/2)\overline{\varphi}(x), \,\,\,\,\forall x\in\partial\Omega_\delta.
 \end{equation}
 Since $\nabla\varphi\cdot\nu=\alpha\varphi$ on $\partial\Omega_*$, by our assumption on  vector field ${\bf b}$, 
 \eqref{liu0902-1} holds for small $\delta$. Then by 
 the continuity of $\overline{\varphi}$, we can choose $\delta$ small
 such that
 \begin{align*}
    &-\Delta\overline{\varphi}-A\mathbf{b}\cdot\nabla\overline{\varphi}+\left(c(x)-\Lambda_\alpha+\epsilon\right)\overline{\varphi}\notag\\
   \geq&-\Delta\overline{\varphi}+\left(c(x)-\Lambda_\alpha\right)\overline{\varphi}+\epsilon\overline{\varphi}
   \geq 0, \quad \forall x\in \Omega_\delta\setminus\Omega_*.
\end{align*}
 This implies \eqref{ldn6} holds on $\Omega_\delta\setminus\Omega_*$.

 $\mathbf{(3)}$ The construction on $\Omega\setminus\Omega_\delta$. By assumption, there are no limit sets in $\Omega\setminus\Omega_\delta$, so that the orbits of \eqref{system}
remain in $\Omega\setminus\Omega_\delta$ only a finite time.
Hence, we can apply  \cite[Lemma 2.3]{DEF1974}  to define $\overline{\varphi}\in C(\Omega)\cap C^2(\Omega\setminus\Omega_\delta)$ satisfying 
\begin{equation}\label{liu0902-2}
{\bf b}(x)\cdot\nabla \overline{\varphi}(x)<0, \,\,\quad \forall x\in \Omega\setminus\Omega_\delta,
\end{equation}
\begin{equation}\label{liu0902-3}
 \nabla_+ \overline{\varphi}(x)\cdot \nu_\delta(x)<(\alpha/2)\overline{\varphi}(x), \,\,\forall x\in \partial \Omega_\delta, \quad \text{and}\quad \nabla \overline{\varphi}\cdot\mathbf{n}(x)\geq 0,\,\,\forall x\in\partial\Omega.
\end{equation}
Since $|{\bf b}(x)|$ has a uniform lower bound  on $\Omega\setminus\Omega_\delta$  independent of $A$, by \eqref{liu0902-2} we can choose $A$ large such that the first inequality in \eqref{ldn6} holds on $\Omega\setminus\Omega_\delta$. Finally, the boundary conditions in  \eqref{ldn6} can be verified by \eqref{liu0902-1} and \eqref{liu0902-3}.

Therefore, we  conclude \eqref{ldn5} holds and the proof of the assertion {\rm (i)} is complete.

\medskip

\noindent {\bf Step 2. } We prove the assertion {\rm(ii)}. Let $\underline{\varphi}>0$ denote the principal eigenfunction of \eqref{liu0902-5} corresponding to $\Lambda_{\bf D}$, which can be viewed as a sub-solution of \eqref{1} as ${\bf b}=0$ on $\Omega_*$. By the comparison principle, it is easily seen that $\lambda(A)\leq \Lambda_{\bf D}$ for all $A>0$. Hence, it remains to show the lower  estimate
 $\liminf\limits_{A\to\infty}\lambda(A)\geq\Lambda_{\mathbf{D}}$.  Letting $\alpha<0$ in \eqref{liu0831-1}, by \eqref{liu0831-2} it suffices to prove \eqref{ldn5} when $|\alpha|$ is sufficiently large. This can be proved by the rather similar arguments as in Step 1, and thus the details are omitted. The proof is now complete.
\end{proof}

Theorem \ref{theorem0831} establishes the asymptotic behavior of  principal eigenvalue when the degenerate region  $\Omega_*$ is a  sink or source, which is determined by  the principal eigenvalue of the operator $-\Delta+c(x)$ in $\Omega_*$  complemented by zero Neumann or Dirichlet boundary conditions. The next result is concerned with the more complicated case.

\begin{theorem}\label{theorem0920}
Suppose that ${\bf b}=0$ on $\Omega_*$ and there are no limit sets of \eqref{system} in $\Omega\setminus\Omega_*$ for some connected domain $\Omega_* \subset\Omega$ satisfying $|\Omega_*|>0$. Let $\nu$ be  the unit outward normal vector on   $\partial\Omega_*$. Assume that $\partial\Omega_*=\overline{\Gamma}_1\cup\overline{\Gamma}_2$  for two connected open sets 
$\Gamma_1$ and $\Gamma_2$ satisfying  $\Gamma_1\cap\Gamma_2=\emptyset$. If there exits some  $\delta_*>0$ such that for any $\delta\in(0,\delta_*)$, it holds ${\bf b}(x+\delta \nu (x))\cdot \nu (x)>0$ on $\Gamma_1$ and ${\bf b}(x+\delta \nu (x))\cdot \nu (x)<0$ on $\Gamma_2$, then  $\lambda(A)\to \Lambda_{\mathbf{DN}}$ as $A\to\infty$,
     where $\Lambda_{\mathbf{DN}}$ is the principal eigenvalue of the problem
    \begin{equation*}
 \left\{\begin{array}{ll}
-\Delta\varphi+c(x)\varphi=\lambda\varphi &\mathrm{in} \,\, \Omega_*,\\
 \varphi=0 & \mathrm{on}~\overline{\Gamma}_1\\
  \nabla\varphi\cdot \nu=0 & \mathrm{on}~\Gamma_2.
  \end{array}\right.
 \end{equation*}
 \end{theorem}
 \begin{figure}[http!!]
  \centering
\includegraphics[height=2.0in]{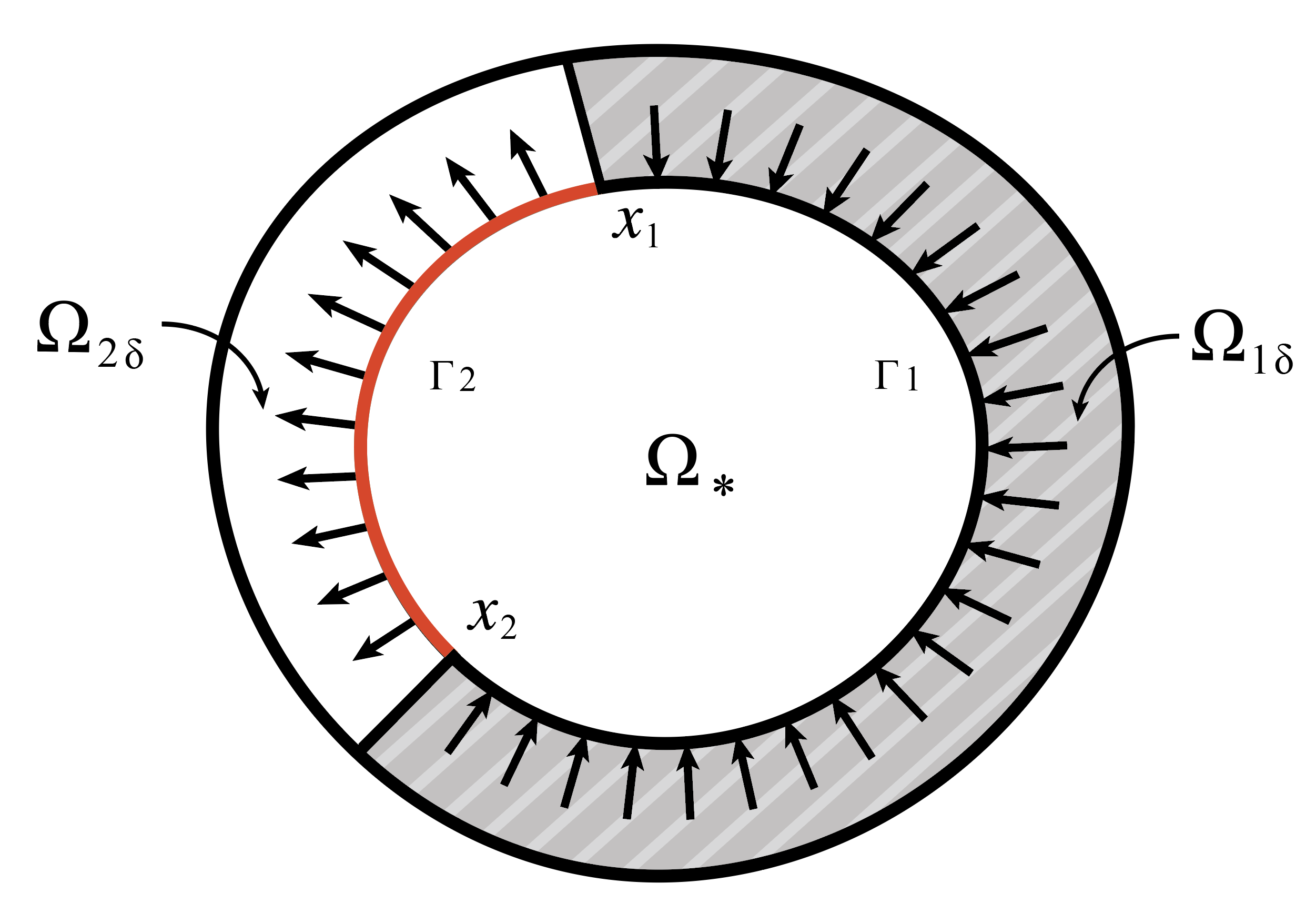}
  \caption{\small An example for the phase-portrait of the \eqref{system} near the boundary of degenerate region $\Omega_*$ as assumed in Theorem  \ref{theorem0920}.}\label{figure1-1-3}
  \end{figure}
  \begin{proof}
 The strategy of the proof follows ideas as in Theorem \ref{theorem0831}. We first prove
 \begin{equation}\label{liu0921-3}
 \liminf\limits_{A\to\infty}\lambda(A)\geq\Lambda_{\mathbf{DN}}.
\end{equation}
Denote $\overline{\Gamma}_1\cap \overline{\Gamma}_2=\{x_1,x_2\}$ with $x_1,x_2\in\partial\Omega_*$.  Given any $\varepsilon>0$, we denote by $\Gamma_{\varepsilon}$   the $\varepsilon$-neighbourhood of $\{x_1,x_2\}$ on $\partial \Omega_*$.  For each $\alpha_1<0$ and $\alpha_2>0$, we define $\alpha\in C(\partial \Omega_*)$ such that  $\alpha$ is monotone on $\Gamma_\varepsilon$, and
 \begin{equation}\label{liu0921-1}
 \alpha(x)=\left\{\begin{array}{ll}
\alpha_1, & x\in \Gamma_1\setminus \Gamma_{\varepsilon},\\
  \alpha_2, & x\in\Gamma_2\setminus \Gamma_{\varepsilon},\\
  0, &x=x_1,x_2.
  \end{array}\right.
 \end{equation}
We consider the
following auxiliary  problem:
  \begin{equation*}
 \left\{\begin{array}{ll}
-\Delta\varphi+c(x)\varphi=\lambda\varphi &\mathrm{in} \,\, \Omega_*,\\
  \nabla\varphi\cdot \nu=\alpha(x) \varphi & \mathrm{on}~\partial\Omega_*,
  \end{array}\right.
 \end{equation*}
 for which we denote $\Lambda(\alpha_1,\alpha_2,\varepsilon)$ as the principal eigenvalue and $\varphi_{\alpha,\varepsilon}>0$ as the associated  eigenfunction.  The principal eigenvalue can be variationally characterized as 
 \begin{equation}\label{liu0921-2}
     \Lambda(\alpha_1,\alpha_2,\varepsilon)=\inf_{\phi\in H^1(\Omega_*)}\left[\frac{-\int_{\partial\Omega_*}\alpha(x)\phi^2{\rm d} S+\int_{\Omega_*} (|\nabla\phi|^2+c(x)\phi^2)\mathrm{d}x}{\int_{\Omega_*} \phi^2\mathrm{d}x}\right].
 \end{equation}
Together with \eqref{liu0921-1} and \eqref{liu0921-2}, it is easily seen that
$$\lim_{(\alpha_1,\alpha_2,\varepsilon)\to (-\infty,0,0)} \Lambda(\alpha_1,\alpha_2,\varepsilon)=\Lambda_{\mathbf{DN}}.$$
Hence, it suffices to prove
\begin{equation}\label{liu0921-4}
 \liminf\limits_{A\to\infty}\lambda(A)\geq\Lambda(\alpha_1,\alpha_2,\varepsilon),
\end{equation}
for all $\alpha_1<0$, $\alpha_2>0$, and  $\varepsilon>0$. To this end, 
given any $\delta>0$, let $\Omega_\delta$ be the $\delta$-neighborhood of $\Omega_*$. Divide $\Omega_\delta$ into three disjoint connected components $\Omega_*$, $\Omega_{1\delta}$, and $\Omega_{2\delta}$ such that $\Omega_\delta=\overline{\Omega}_*\cup \overline{\Omega}_{1\delta} \cup \overline{\Omega}_{2\delta}$, 
which satisfies 
$\Gamma_2\subset \overline{\Omega}_{2\delta}$ and the omega-limit set of each point $x\in \Omega_{1\delta}$  locates on $\Gamma_1$, while the alpha-limit set of  each point $x\in \Omega_{2\delta}$ locates on $\Gamma_2$. The division is possible due to our assumption on  vector field ${\bf b}$ near $\Omega_*$. See Fig. \ref{figure1-1-3} for some illustrations. Denote $\Gamma_{\rm b}$ as the boundary of the open set $\Omega_*\cup\Omega_{1\delta}\cup\Omega_{2\delta}$, which can be depicted by $\Gamma_{\rm b}=\partial\Omega_\delta\cup \partial\Omega_*\cup(\overline{\Omega}_{1\delta}\cap\overline{\Omega}_{2\delta})$. We then denote by $\nu(x)$ the unit normal vector on $\Gamma_{\rm b}$ such that $\nu(x)$ serves as the outward normal vector of $\Omega_\delta$ on $\partial\Omega_\delta$, the outward normal vector of $\Omega_*$ on $\partial\Omega_*$, as well as the outward normal vector of $\Omega_{1\delta}$ on $\overline{\Omega}_{1\delta}\cap\overline{\Omega}_{2\delta}$.  Fix any $\alpha_1<0$, $\alpha_2>0$, and  $\varepsilon>0$. We shall construct some positive super-solution $\overline \varphi$ such that
\begin{equation}\label{liu0921-5}
 \left\{\begin{array}{ll}
\medskip
-\Delta\overline{\varphi}-A\mathbf{b}\cdot\nabla\overline{\varphi}+c(x)\overline{\varphi}\geq \Lambda(\alpha_1,\alpha_2,\varepsilon)\overline{\varphi} &\mathrm{in} \,\, \Omega\setminus\Gamma_{\rm b},\\
\medskip
(\nabla_+\overline{\varphi}(x)-\nabla_-\overline{\varphi}(x))\cdot \nu(x)>0 & \mathrm{on}~\Gamma_{\rm b},\\
  \nabla\overline{\varphi}\cdot {\bf n}(x)\geq0 & \mathrm{on}~\partial\Omega,  
  \end{array}\right.
 \end{equation}
provided that $A$ is sufficiently large.
  Then \eqref{liu0921-4} follows from the comparison principle.

  The construction of such super-solution $\overline{\varphi}$ can be performed by similar arguments in the proof of Theorem \ref{theorem0831}. More precisely,  we first define $\overline \varphi(x):=\varphi_{\alpha,\varepsilon}(x)$ for all $x\in\Omega_*$. Due to ${\bf b}=0$ in  $\Omega_*$, by the definition of $\varphi_{\alpha,\varepsilon}$, one can check  function $\overline \varphi$ satisfies the first inequality in \eqref{liu0921-5}  on $\Omega_*$ for all $A>0$. On $\Omega_\delta\setminus\Omega_*$,
we define $\overline{\varphi}\in C(\Omega_\delta) \cap C^2(\Omega_\delta\setminus\Gamma_{\rm b})$  such that
 \begin{equation}\label{liu0921-6}
 \begin{array}{cc}
 \smallskip
     {\bf b}\cdot\nabla\overline{\varphi}(x)\leq 0, & \forall x\in\Omega_\delta\setminus\Omega_*,  \\
     (\nabla_+\overline{\varphi}(x)-\nabla_-\overline{\varphi}(x))\cdot \nu(x)>0, & \forall x\in\Gamma_{\rm b}\setminus \partial\Omega_\delta.
 \end{array}
 \end{equation}
 This is possible by observing the boundary conditions satisfied by  $\varphi_{\alpha,\varepsilon}$.  As in the proof of Theorem \ref{theorem0831}, by \eqref{liu0921-6} and the continuity of $\overline{\varphi}$, we can choose $\delta$ small if necessary such that
 \eqref{liu0921-5} holds on $\Omega_\delta\setminus\Omega_*$.
  The contraction of  $\overline{\varphi}$ on $\Omega\setminus\Omega_\delta$ can be completed by applying \cite[Lemma 2.3]{DEF1974}  as before. Hence, the super-solution satisfying \eqref{liu0921-5} does exist, which implies \eqref{liu0921-4}, and thus \eqref{liu0921-3} holds true.

  It remains to  establish the upper bound estimate
 \begin{equation}\label{liu0921-8}
 \liminf\limits_{A\to\infty}\lambda(A)\leq\Lambda_{\mathbf{DN}}.
\end{equation}
For any $\alpha<0$, let $\Lambda_\alpha$ denote the principal eigenvalue of the following auxiliary  problem:
   \begin{equation}\label{liu0921-7}
 \left\{\begin{array}{ll}
-\Delta\varphi+c(x)\varphi=\lambda\varphi &\mathrm{in} \,\, \Omega_*,\\
 \varphi=0 & \mathrm{on}~\overline{\Gamma}_1\\
  \nabla\varphi\cdot \nu=\alpha\varphi & \mathrm{on}~\Gamma_2,
  \end{array}\right.
 \end{equation}
 for which $\varphi_\alpha>0$ denotes the corresponding eigenfunction.
 It is clearly that $\Lambda_\alpha\to \Lambda_{\mathbf{DN}}$ as $\alpha\to 0$.
 We shall define some nonnegative function $\underline{\varphi}\in C(\Omega_\delta)$ such that ${\rm spt} (\underline{\varphi})=\overline{\Omega}_*\cup\overline{\Omega}_{2\delta}$, and   $\underline{\varphi}=\varphi_\alpha$ on $\Omega_*$. Due to $\alpha<0$, similar to \eqref{liu0921-5} we can construct $\underline{\varphi}\in C(\Omega_\delta)$ such that
 \begin{equation*}
 \left\{\begin{array}{ll}
\medskip
-\Delta\underline{\varphi}-A\mathbf{b}\cdot\nabla\underline{\varphi}+c(x)\underline{\varphi}\leq \Lambda_\alpha\underline{\varphi} &\mathrm{in} \,\, \Omega\setminus\Gamma_{\rm b},\\
\medskip
(\nabla_+\underline{\varphi}(x)-\nabla_-\underline{\varphi}(x))\cdot \nu(x)<0 & \mathrm{on}~\Gamma_{\rm b},\\
  \nabla\underline{\varphi}\cdot {\bf n}(x)\leq0 & \mathrm{on}~\partial\Omega,  
  \end{array}\right.
 \end{equation*}
for sufficiently large $A$. Then by the comparison principle we derive $\lambda(A)\leq\Lambda_\alpha$ for large $A$.
Then \eqref{liu0921-8} follows by letting $\alpha\to0$. This completes the proof.
  \end{proof}

\begin{appendices}
\section{Proof of Proposition \ref{theorem0822}}
This section is devoted to proving Proposition \ref{theorem0822}.

\begin{proof}[Proof of Proposition {\rm\ref{theorem0822}}]
{\bf Part 1. Proof of the assertion (i)}. If $\alpha=0$, then Proposition \ref{theorem0822}{\rm (i)} is a direct consequence of Proposition \ref{theorem0820}, see also  \cite[Theorem 2.2]{BHN2005}. In this part,  we assume $\alpha\in (0,\frac{1}{2})$ and prove Proposition \ref{theorem0822}{\rm (i)} by the following two steps.

{\bf Step 1}.
We first prove the lower bound estimate
\begin{equation}\label{liu-1}
\liminf_{A\to\infty}\lambda_\alpha(A)\geq \inf_{\phi\in \mathbf{I}_\alpha}\frac{\int_{B_\alpha} (|\nabla\phi|^2+c(x)\phi^2){\rm d}x}{\int_{B_\alpha} \phi^2{\rm d}x}=:\Lambda_\alpha.
\end{equation}
 For any $\delta>0$, we set $O_{\delta}=\{x\in \mathbb{R}^2: |x-(0,-\frac{\alpha}{1-\alpha})|
    <\frac{1-2\alpha}{1-\alpha}+\delta\}$. Given any $\epsilon>0$, by the proof of Proposition \ref{theorem0820}{\rm(i)} we can construct a positive super-solution $\overline{\phi}_{\delta}\in C^2(O_\delta)$ satisfying
 \begin{equation}\label{liu-2}
 \left\{\begin{array}{ll}
\medskip
-\Delta\overline{\phi}_{\delta}-A\mathbf{b}\cdot\nabla\overline{\phi}_{\delta}+c(x)\overline{\phi}_{\delta}\geq (\Lambda_\alpha-\epsilon)\overline{\phi}_{\delta} & \mathrm{in} \,\, \,\,O_{\delta},\\
  \nabla\overline{\phi}_{\delta}\cdot \nu\geq\sqrt{\epsilon}\overline{\phi}_{\delta} &\mathrm{on}\,\,\,\,\partial  O_{\delta} 
 \end{array}\right.
 \end{equation}
for sufficiently large $A$.
In what follows, we shall construct  
$\overline{\varphi}\in C^2(\Omega)$  such that
 \begin{equation}\label{liu1004-1}
 \left\{\begin{array}{ll}
\medskip
-\Delta\overline{\varphi}+A\mathbf{b}\cdot\nabla\overline{\varphi}+c(x)\overline{\varphi}\geq (\Lambda_\alpha-2\epsilon)\overline{\varphi} & \mathrm{in} \,\, \Omega,\\
  \nabla\overline{\varphi}\cdot \mathbf{n}\geq 0 &\mathrm{on}~\partial \Omega, 
 \end{array}\right.
 \end{equation}
 provided  that $A$ is chosen large. 
Then the desired \eqref{liu-1} follows by the comparison principle. 


 To this end,  we introduce the polar coordinate $x\mapsto(r,\theta)$ such that
 $x-(0,-\frac{\alpha}{1-\alpha})=(r\cos \theta, r\sin \theta)$ with $r\in(0,\frac{1}{1-\alpha})$ and  $\theta\in(-\pi,\,\pi)$.
Set $\zeta\in C^2([0,\frac{1-2\alpha}{1-\alpha}+\delta])$ to satisfy
\begin{equation}\label{zeta_1}
    \left\{\begin{array}{cl}
    \medskip
       \zeta(r)\equiv 0,  &  r\in [0,\tfrac{1-2\alpha}{1-\alpha}+\delta/2],  \\
        \medskip
        0<\zeta(r)\leq \delta, &  r\in (\tfrac{1-2\alpha}{1-\alpha}+\delta/2,\,\, \tfrac{1-2\alpha}{1-\alpha}+\delta],\\
         |\zeta'(r)|+|\zeta''(r)|\leq \delta, & r\in (0,\tfrac{1-2\alpha}{1-\alpha}+\delta].
    \end{array}
    \right.
\end{equation}
Then we define
 \begin{equation}\label{liu-definition1}
   \overline{\varphi}(x):=\overline{\phi}_{\delta}(x)-\zeta(r) \theta, \quad\,\, \forall x\in O_{\delta}.
 \end{equation}
 Direct calculations give
  \begin{align*}
\mathbf{b}\cdot\nabla(\zeta(r) \theta)=&(-(1-\alpha)x_2-\alpha, \,\,\,\,(1-\alpha)x_1)\\ &\cdot\left[\theta\zeta'(r)\frac{x_1}{r}-\zeta(r)\frac{x_2+\frac{\alpha}{1-\alpha}}{r^2},\,\,\,\,\theta\zeta'(r)\frac{x_2+\frac{\alpha}{1-\alpha}}{r}+\zeta(r)\frac{x_1}{r^2}\right]\\
&=(1-\alpha)\zeta(r)\geq 0, \qquad \forall x\in O_\delta.
\end{align*}
In view of $\Delta r=1/r$, $\Delta \theta=0$, and  $|\nabla  r|=r$, by \eqref{zeta_1} we can calculate that
  \begin{equation*}
\begin{split}
\Delta(\zeta(r) \theta)=\zeta''(r) \theta+ \frac{\zeta'(r)}{r}\theta\leq \pi \left[1+\frac{1-2\alpha}{1-\alpha}\right]\delta, \quad \forall x\in O_\delta.
\end{split}
 \end{equation*}
 Hence, by \eqref{liu-definition1} and the definition of $\overline{\phi}_{\delta}$ in \eqref{liu-2}, we can derive
 \begin{align*}
&-\Delta\overline{\varphi}-A\mathbf{b}\cdot\nabla\overline{\varphi}+c(x)\overline{\varphi}\notag\\
\geq&-\Delta\overline{\phi}_{\delta}-A\mathbf{b}\cdot\nabla \overline{\phi}_{\delta}+c(x)\overline{\phi}_{\delta}+\Delta(\zeta(r) \theta)-c\zeta(r) \theta\\
\geq& (\Lambda_\alpha-\epsilon-{\rm C}\delta)\overline{\varphi},\quad \forall  x\in O_{\delta},\notag
 \end{align*}
 where ${\rm C}>0$ depends only upon $\alpha$, $\|c\|_\infty$, and the uniform lower bound of $\overline{\phi}_{\delta}$. Hence,  such $\overline{\varphi}$ 
 verifies the first inequality of \eqref{liu1004-1} on $O_{\delta}$  by choosing $\delta$ small if necessary.

 For any $x\in \partial  \Omega\cap O_{\delta}$, due to $|x|=1$ we calculate that
 \begin{align*}
    \nabla\overline{\varphi}\cdot \mathbf{n}&=\nabla\overline{\phi}_{\delta}\cdot \mathbf{n}+\left[\theta\zeta'(r)\frac{x_1}{r}-\zeta(r)\frac{x_2+\frac{\alpha}{1-\alpha}}{r^2},\,\,\,\,\theta\zeta'(r)\frac{x_2+\frac{\alpha}{1-\alpha}}{r}+\zeta(r)\frac{x_1}{r^2}\right]\cdot x\\
    &=\nabla\overline{\phi}_{\delta}\cdot \mathbf{n}+\theta\zeta'(r)\frac{1+\frac{\alpha}{1-\alpha} x_2}{r}-\zeta(r)\frac{\alpha x_1}{(1-\alpha)r^2}.
 \end{align*}
 Hence, by \eqref{liu-2} and the choice of $\zeta$ in \eqref{zeta_1}, we can choose $\delta$ small such that
 \begin{equation}\label{liu-9}
\nabla\overline{\varphi}\cdot \mathbf{n}\geq (\sqrt{\epsilon}/2)\overline{\varphi}, \quad \forall x\in \partial  \Omega\cap O_{\delta}.
\end{equation}

Next, we construct the super-solution $\overline{\varphi}$ on $\Omega\setminus \overline{O}_{\delta}$. Set $\zeta_*:=\zeta(\frac{1-2\alpha}{1-\alpha}+\delta)\in (0,\delta)$.
For any $x\in\partial O_{\delta}\cap \Omega$, by \eqref{liu-definition1}  we have $\overline{\varphi}(x)=\overline{\phi}_{\delta}(x)-\zeta_* \theta$, and thus
$$\partial_\theta\overline{\varphi}=\partial_\theta\overline{\phi}_{\delta}-\zeta_*={\bf b}\cdot \nabla \overline{\phi}_{\delta} -\zeta_*\in ( -2\zeta_*, -\zeta_*/2),$$
provided that $A$ is chosen 
large enough. Therefore, we can extend $\overline{\varphi}$ to $\Omega$ such that $\overline{\varphi}\in C^2(\Omega\setminus \overline{O}_{\delta})$ and the following inequalities hold:
\begin{equation}\label{liu0828-1}
 \left\{\begin{array}{lll}
\medskip
{\rm (i)} &-3\zeta_* \leq \partial_\theta \overline{\varphi}(x)\leq -\zeta_*/4, &\forall x\in \Omega\setminus \overline{O}_{\delta}, \\
\medskip
{\rm (ii)} &  \nabla\overline{\varphi}\cdot \mathbf{n}\geq 0, & \forall x\in \partial \Omega\cap(O_{2\delta}\setminus O_{\delta}), \\
{\rm (iii)} &  \partial_r \overline{\varphi}(x)\geq \frac{3\alpha(1-\alpha)\zeta_*}{(1-2\alpha)^2}, & x\in \overline{\Omega}\setminus O_{2\delta},
 \end{array}\right.
 \end{equation}
where the choice {\rm (iii)}  is possible due to  \eqref{liu-9}.

We next verify the constructed  $\overline{\varphi}$ satisfies \eqref{liu1004-1}.
Using {\rm (i)} in \eqref{liu0828-1}, we calculate that 
 \begin{equation*}
 \begin{split}
-\Delta\overline{\varphi}-A\mathbf{b}\cdot\nabla\overline{\varphi}+c(x)\overline{\varphi}
=&-\Delta\overline{\varphi}-A\partial_\theta \overline{\varphi}+c(x)\overline{\varphi}\\
\geq&-\Delta\overline{\varphi}+A\zeta_*/4+c(x)\overline{\varphi}
\geq \Lambda_\alpha\overline{\varphi},\qquad \forall x\in \Omega\setminus \overline{O}_{\delta}
 \end{split}
 \end{equation*}
 for sufficiently large $A$.
 For any $x\in \partial \Omega\setminus O_{2\delta}$, by {\rm(i)} and {\rm(iii)} in \eqref{liu0828-1} we deduce that
 \begin{align}\label{liu-11}
    \nabla\overline{\varphi}\cdot \mathbf{n}&= \left[\frac{x_1\partial_r \overline{\varphi}}{r}-\frac{(x_2+\frac{\alpha}{1-\alpha})\partial_\theta\overline{\varphi}}{r^2},\,\,\frac{(x_2+\frac{\alpha}{1-\alpha})\partial_r \overline{\varphi}}{r}+\frac{x_1\partial_\theta\overline{\varphi}}{r^2}\right]\cdot x\notag\\
    &=\frac{(1+\frac{\alpha}{1-\alpha} x_2)\partial_r \overline{\varphi}}{r}-\frac{\frac{\alpha}{1-\alpha} x_1\partial_\theta \overline{\varphi}}{r^2}\\
    &\geq \frac{(1-2\alpha)\partial_r \overline{\varphi} }{(1-\alpha)r}-\frac{3\alpha\zeta_* }{(1-\alpha)r^2}\geq 0, \quad \forall x\in \partial \Omega\setminus O_{2\delta}. \notag
 \end{align}
 This verifies \eqref{liu1004-1}, and thus
\eqref{liu-1} follows.


 {\bf Step 2}. We next establish the upper bound estimate
\begin{equation}\label{liu-1-1}
\liminf_{A\to\infty}\lambda_\alpha(A)\leq \Lambda_\alpha=\inf_{\phi\in \mathbf{I}_\alpha}\frac{\int_{B_\alpha} (|\nabla\phi|^2+c(x)\phi^2){\rm d}x}{\int_{B_\alpha} \phi^2{\rm d}x}.
\end{equation}
Let $O_\delta$ be given in Step 1. Given any $\epsilon>0$, similar to the proof of Proposition \ref{theorem0820}{\rm(i)} we can construct  $\underline{\phi}_{\delta}\in C^2(O_\delta)$ such that
 \begin{equation*}
 \left\{\begin{array}{ll}
\medskip
-\Delta\underline{\phi}_{\delta}-A\mathbf{b}\cdot\nabla\underline{\phi}_{\delta}+c(x)\underline{\phi}_{\delta}\geq (\Lambda_\alpha+\epsilon)\underline{\phi}_{\delta} & \mathrm{in} \,\, \,\,O_{\delta},\\
  \nabla\underline{\phi}_{\delta}\cdot \nu\leq -\sqrt{\epsilon}\underline{\phi}_{\delta} &\mathrm{on}\,\,\,\,\partial  O_{\delta}, 
 \end{array}\right.
 \end{equation*}
provided that $A$ is sufficiently large.
Let $\zeta\in C^2([0,1-\alpha+\delta])$ be defined in Step 1. We define
 \begin{equation}\label{liu-definition1-1}
   \underline{\varphi}(x):=\underline{\phi}_{\delta}(x)+\zeta(r) \theta, \qquad \forall x\in O_{\delta},
 \end{equation}
  and for any $x\in \Omega\setminus \overline{O}_{\delta}$,  we define $\underline{\varphi}$ such that $\underline{\varphi}\in C^2(\Omega\setminus \overline{O}_{\delta})$, $\underline{\varphi}>0$, and
\begin{equation}\label{liu-2-1}
 \left\{\begin{array}{lll}
\medskip
{\rm (i)} &\zeta_*/4 \leq \partial_\theta \underline{\varphi}(x)\leq 3\zeta_*, &\forall x\in \Omega\setminus \overline{O}_{\delta}, \\
\medskip
{\rm (ii)} &  \nabla\underline{\varphi}\cdot \mathbf{n}\leq 0, & \forall x\in \partial \Omega\cap(O_{2\delta}\setminus O_{\delta}), \\
{\rm (iii)} &  \partial_r \underline{\varphi}(x)\leq- \frac{3\alpha(1-\alpha)\zeta_*}{(1-2\alpha)^2}, & x\in \overline{\Omega}\setminus O_{2\delta}.
 \end{array}\right.
 \end{equation}
 Similar to Step 1, we can verify that the sub-solution $\underline{\varphi}$ defined by \eqref{liu-definition1-1} and \eqref{liu-2-1} satisfies
\begin{equation*}
 \left\{\begin{array}{ll}
\medskip
-\Delta\underline{\varphi}-A\mathbf{b}\cdot\nabla\underline{\varphi}+c(x)\underline{\varphi}\leq (\Lambda_\alpha+2\epsilon)\underline{\varphi} & \mathrm{in} \,\, \Omega,\\
  \nabla\underline{\varphi}\cdot \mathbf{n}\leq 0 &\mathrm{on}~\partial \Omega. 
 \end{array}\right.
 \end{equation*}
 Then the upper bound estimate follows by the comparison principle, which completes Step 2.

Therefore, Proposition \ref{theorem0822}{\rm (i)} follows from Step 1 and Step 2.
\qed


\medskip

{\bf Part 2. Proof of the assertion (ii)}.
We assume $\alpha\in [\frac{1}{2},1]$ and prove Proposition \ref{theorem0822}{\rm (ii)}.
We only prove the lower bound estimate
\begin{equation}\label{liu-06}
\liminf_{A\to\infty}\lambda_\alpha(A)\geq c(x_\alpha),
\end{equation}
and the upper bound estimate follows by a similar argument. Given any $\epsilon>0$, we will construct a positive super-solution $\overline{\varphi}\in C^2(\Omega)$ such that
\begin{equation}\label{liu-06-1}
 \left\{\begin{array}{ll}
\medskip
-\Delta\overline{\varphi}+A\mathbf{b}\cdot\nabla\overline{\varphi}+c(x)\overline{\varphi}\geq (c(x_\alpha)-\epsilon)\overline{\varphi} & \mathrm{in} \,\, \Omega,\\
  \nabla\overline{\varphi}\cdot \mathbf{n}\geq 0 &\mathrm{on}~\partial \Omega, 
 \end{array}\right.
 \end{equation}
provided $A$ is sufficiently large. Then \eqref{liu-06} follows from the comparison principle.

To this end,
we  choose $\delta>0$ small such that
  $|c(x)-c(x_\alpha)|\leq \epsilon$, $ \forall x\in O_{2\delta}$,
where $O_{2\delta}=\{x\in\mathbb{R}^2: |x-x_\alpha|<2\delta\}$.
Set
$$\underline{r}_\delta:=\min_{x\in\overline{O}_{2\delta}}|x-(0,-\tfrac{\alpha}{1-\alpha})|,\quad\overline{r}_\delta:=\max_{x\in\overline{O}_{2\delta}}|x-(0,-\tfrac{\alpha}{1-\alpha})|.$$
It is easily verified that
$|x_\alpha-(0,-\tfrac{\alpha}{1-\alpha})|=\tfrac{\sqrt{2\alpha-1}}{1-\alpha} $ and $\underline{r}_\delta<\tfrac{\sqrt{2\alpha-1}}{1-\alpha} < \overline{r}_\delta.$

 \begin{figure}[http!!]
  \centering
\includegraphics[height=2.3in]{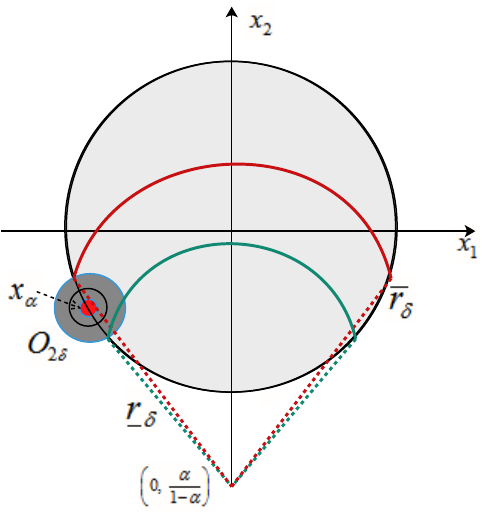}
  \caption{\small Illustrations for some notations in the proof of Proposition \ref{theorem0822}{\rm (ii)}. 
  }\label{figure1-1-1-liu}
  \end{figure}

We  first define $\overline{\varphi}\equiv 1$ in $\overline{O}_\delta\cap \overline{\Omega}$ and  introduce the polar coordinate
 $x\mapsto(r,\theta)$ such that $x-(0,-\frac{\alpha}{1-\alpha})=(r\cos \theta, r\sin \theta)$ with $r\in(0,1+\frac{\alpha}{1-\alpha})$ and  $\theta\in(0,\pi)$.
For any $r\in(0,1+\frac{\alpha}{1-\alpha})$, set $B_r:=\{x\in\mathbb{R}^2: |x-(0,-\tfrac{\alpha}{1-\alpha})|<r\}$.
We then define $\overline{\varphi}\in C^2(\Omega)$ such that
\begin{equation}\label{liu-09-1}
 \partial_r \overline{\varphi}\leq 0 \,\,\text{ on }\,\,B_{\tfrac{\sqrt{2\alpha-1}}{1-\alpha} }\cap \Omega,\quad \partial_r \overline{\varphi}\geq 0 \,\,\text{ on }\,\,\Omega\setminus \overline{B}_{\tfrac{\sqrt{2\alpha-1}}{1-\alpha}},  \quad\text{and}\quad \partial_\theta \overline{\varphi}\leq 0 \,\,\text{ in }\,\,\Omega,
\end{equation}
and moreover
\begin{equation}\label{liu-09}
  \left\{\begin{array}{ll}
           \medskip
          {\rm (i)} & \partial_r \overline{\varphi}< -\epsilon \overline{\varphi},  \quad\forall x\in B_{\underline{r}_\delta}\cap \Omega, \qquad  \partial_r \overline{\varphi}>\epsilon \overline{\varphi}, \quad \forall x\in \Omega\setminus \overline{B}_{\overline{r}_\delta}, \\
           \medskip
           {\rm (ii)} &  \partial_\theta \overline{\varphi}\leq -\epsilon_* \overline{\varphi},  \quad\forall x\in \Omega\setminus O_{2\delta}, \\
          {\rm (iii)} &   |\partial_r \overline{\varphi}|+|\partial_{rr} \overline{\varphi}|+|\partial_{\theta\theta} \overline{\varphi}|\leq \epsilon \overline{\varphi},  \quad \forall x\in  \Omega\cap O_{2\delta},\\
         \end{array}
  \right.
\end{equation}
where $\epsilon_*\in(0,\epsilon)$ will be determined later.

We now verify \eqref{liu-06-1}. On  $\Omega\cap O_{2\delta}$, by  \eqref{liu-09-1} as well as {\rm (ii)} and  {\rm (iii)} in \eqref{liu-09} we have
 \begin{equation*}
 \begin{split}
-\Delta\overline{\varphi}-A\mathbf{b}\cdot\nabla\overline{\varphi}+c(x)\overline{\varphi}
=&-\partial_{rr}\overline{\varphi}-\frac{\partial_r \overline{\varphi}}{r}-\frac{\partial_{\theta\theta} \overline{\varphi}}{r^2}-A\partial_\theta \overline{\varphi}+c(x)\overline{\varphi}\\
\geq&-\partial_{rr}\overline{\varphi}-\frac{\partial_r \overline{\varphi}}{r}-\frac{\partial_{\theta\theta} \overline{\varphi}}{r^2}+c(x)\overline{\varphi}\\
\geq&\left[c(x_\alpha) -(2+1/r+1/r^2)\epsilon\right]\overline{\varphi},\quad \forall x\in \Omega\cap O_{2\delta}.
 \end{split}
 \end{equation*}
For any $x\in \Omega\setminus O_{2\delta}$, by \eqref{liu-09} we can choose $A$ large such that
\begin{equation*}
 \begin{split}
-\Delta\overline{\varphi}+A\mathbf{b}\cdot\nabla\overline{\varphi}+c\overline{\varphi}
=&-\Delta\overline{\varphi}-A\partial_\theta \overline{\varphi}+c(x)\overline{\varphi}\\
\geq&-\Delta\overline{\varphi}+ A\epsilon_* \overline{\varphi}+c\overline{\varphi}\geq c(x_\alpha) \overline{\varphi},\quad \forall x\in \Omega\setminus O_{2\delta}.
 \end{split}
 \end{equation*}
 Hence, the first inequality in \eqref{liu-06-1} holds. It remains to prove $\nabla\overline{\varphi}\cdot \mathbf{n}\geq 0$ on $\partial \Omega$.

 For any $x\in \partial \Omega$, 
we observe from \eqref{liu-11} that
\begin{equation*}
\begin{split}
 \nabla\overline{\varphi}\cdot \mathbf{n}
    =\frac{(1+\frac{\alpha}{1-\alpha} x_2)\partial_r \overline{\varphi}}{r}-\frac{\alpha x_1\partial_\theta \overline{\varphi}}{(1-\alpha)r^2}, \quad \forall x\in\partial\Omega,
\end{split}
\end{equation*}
which together with \eqref{liu-09-1} yields $\nabla\overline{\varphi}\cdot \mathbf{n}\geq 0$ on $\partial \Omega \cap \{x\in\mathbb{R}^2: x_1>0\}$.
Moreover,  by  \eqref{liu-09}-{\rm (i)} we infer that for any $x\in (\partial \Omega\cap B_{\underline{r}_\delta})\cup (\partial \Omega\cap(\Omega\setminus \overline{B}_{\overline{r}_\delta}))$,
\begin{align*}
    \nabla\overline{\varphi}\cdot \mathbf{n}\geq 
    \frac{\delta\epsilon \overline{\varphi}}{r}-\frac{\alpha \epsilon_* \overline{\varphi}}{(1-\alpha)r^2}> 0 
 \end{align*}
by choosing $\epsilon_*$ small. This allows us to define 
$\overline{\varphi}\in C^2(\Omega)$ such that
  $\nabla\overline{\varphi}\cdot \mathbf{n}\geq 0$ on $\partial \Omega\cap O_{2\delta}$.
This verifies $\nabla\overline{\varphi}\cdot \mathbf{n}\geq 0$ on $\partial \Omega$ and the proof is complete.
\end{proof}

\section{Proof of Corollary \ref{coro1}}
This section is devoted to proving Corollary \ref{coro1}.

\begin{proof}[Proof of Corollary {\rm\ref{coro1}}]
Due to $\Omega={\rm H}^{-1}(1)$ and $\alpha<1$, by definition ${\bf b}_\alpha(x)=(-\partial_{x_2}{\rm H}, \partial_{x_1}{\rm H})-({\rm H}(x)-\alpha)\nabla {\rm H}$, it can be verified directly that assumption \eqref{assumption1} holds true.  When $\alpha\in(-\frac{1}{4},1)$, it follows from \eqref{liu1026-1} that ${\rm H}^{-1}(\alpha)$ is the unique stable connected component in the limit set of  \eqref{system} and Hypothesis \ref{assum1} is satisfied. Then Corollary \ref{coro1} is a direct consequence of Theorem \ref{mainresult}. It remains to consider the case $\alpha=-\frac{1}{4}$. 
By \eqref{liu1026-1} we find that  \eqref{system} has no periodic orbits, and $x_+=(1,0)$ and $x_-=(-1,0)$ are two stable fixed points which are not hyperbolic. Thus Theorem \ref{mainresult} is inapplicable to this case, but the proof can follow by the ideas in Theorem  \ref{liuprop1}. We only show the lower bound estimate
\begin{equation}\label{liu1026-2}
  \liminf_{A\to\infty}\lambda(A)\geq \min\{c(x_+),\,\, c(x_-)\},
\end{equation}
and the upper bound estimate is similar by referring that of Theorem  \ref{liuprop1}.

Given any $\epsilon>0$, we first choose $\delta$ small such that $d_{\mathcal{H}}(x,\{x_+,x_-\})\leq \epsilon$ for all $x\in{\rm H}^{-1}([-\frac{1}{4},-\frac{1}{4}+\delta))$ with $d_{\mathcal{H}}(\cdot,\cdot)$ being the distance between sets in the Hausdorff sense.  Set $\Gamma_\delta:={\rm H}^{-1}(-\frac{1}{4}+\delta)\cup {\rm H}^{-1}(\pm\delta)$, which serves as  the boundary of ${\rm H}^{-1}([-\frac{1}{4},-\frac{1}{4}+\delta))\cup {\rm H}^{-1}((-\delta,\delta))$ and the  outward unit normal vector is given by $\nu=\nabla{\rm H}/|\nabla{\rm H}|.$ We shall construct  a positive super-solution $\overline{\varphi}\in C(\Omega)\cap C^2(\Omega\setminus\Gamma_\delta)$ such that
\begin{equation}\label{liu1026-3}
 \left\{\begin{array}{ll}
\medskip
-\Delta\overline{\varphi}-A\mathbf{b}_{-\frac{1}{4}}\cdot\nabla\overline{\varphi}+c(x)\overline{\varphi}\geq\Big[\min\{c(x_+),\,\, c(x_-)\}-2\epsilon\Big] \overline{\varphi}&\mathrm{in} \,\, \Omega\setminus\Gamma_\delta,\\
\medskip
(\nabla_-\overline{\varphi}(x)-\nabla_+\overline{\varphi}(x))\cdot \nu(x)> 0 & \mathrm{on}~\Gamma_\delta,\\
  \nabla\overline{\varphi}\cdot \mathbf{n}\geq 0 & \mathrm{on}~\partial\Omega,  
  \end{array}\right.
 \end{equation}
provided that  $A$ is sufficiently large, where $\nabla\overline{\varphi}_\pm\cdot \nu$ is defined by \eqref{liu0621-3}. Then \eqref{liu1026-2} follows from the comparison principle and the arbitrariness of $\epsilon$.

To this end, we first define
\begin{equation}\label{liu1026-definition}
    \overline{\varphi}(x):=\frac{\epsilon }{4}{\rm H}(x)+1, \quad\forall x\in {\rm H}^{-1}([-\frac{1}{4},-\frac{1}{4}+\delta)).
\end{equation}
By definition, we observe that ${\bf b}_{-\frac{1}{4}}\cdot \nabla\overline{\varphi}=-\frac{\epsilon}{4}({\rm H}+\frac{1}{4})|\nabla {\rm{H}}|^2\leq 0$ on ${\rm H}^{-1}([-\frac{1}{4},-\frac{1}{4}+\delta))$. By \eqref{liu1026-definition}, direct calculations yield
\begin{align*}
&-\Delta\overline{\varphi}-A\mathbf{b}_{-\frac{1}{4}}\cdot\nabla\overline{\varphi}+c(x)\overline{\varphi}\\
\geq &-\Delta\overline{\varphi}+\Big[\min\{c(x_+),\,\, c(x_-)\}-\epsilon\Big] \overline{\varphi}\\
=&-\frac{3\epsilon}{4}x_1^2+\Big[\min\{c(x_+),\,\, c(x_-)\}-\epsilon\Big] \overline{\varphi}\\
\geq&\Big[\min\{c(x_+),\,\, c(x_-)\}-2\epsilon\Big] \overline{\varphi}, \quad\,\, \forall x\in{\rm H}^{-1}([-\tfrac{1}{4},-\tfrac{1}{4}+\delta)).
\end{align*}
This verifies the first equation in \eqref{liu1026-3} on ${\rm H}^{-1}([-\frac{1}{4},-\frac{1}{4}+\delta))$. Moreover, it holds that
\begin{equation}\label{liu1026-4}
  \nabla_-\overline{\varphi}\cdot \nu=\frac{\epsilon}{4}\nabla {\rm H}\cdot \frac{\nabla {\rm H}}{|\nabla {\rm H}|}=\frac{\epsilon}{4}|\nabla {\rm H}|>0, \quad \forall x\in {\rm H}^{-1}(-\tfrac{1}{4}+\delta).
\end{equation}
Notice that the limit set of \eqref{system}  restricted in $\Omega\setminus{\rm H}^{-1}([-\frac{1}{4},-\frac{1}{4}+\delta))$ contains the hyperbolic saddle only. Based on \eqref{liu1026-4}, we can use the same arguments  in the proof of Theorem  \ref{liuprop1} to construct such super-solution $\overline{\varphi}$ satisfying \eqref{liu1026-3}. This complets the proof.
\end{proof}

\end{appendices}
\bigskip
\noindent{\bf Acknowledgments.}
SL is partially supported by the NSFC grant No. 12201041, China National Postdoctoral Program for
Innovative Talents (No. BX20220377), China Postdoctoral Science Foundation (No. 2022M710391),
Beijing Institute of Technology Research Fund Program for Young Scholars, and the Shanghai Frontier Research Center of Modern Analysis.
MZ was partially supported by  National Key R\&D Program of China (2020YFA0713300) and Nankai ZhiDe Foundation.

\bigskip

\baselineskip 18pt
\renewcommand{\baselinestretch}{1.2}

\end{document}